\documentclass[12pt,a4paper]{article}
\usepackage[utf8]{inputenc}
\usepackage{amsmath}
\usepackage{amsthm}
\usepackage{amsfonts}
\usepackage{amssymb}
\usepackage{mathrsfs}
\usepackage{dsfont}
\usepackage[left=3cm,right=2cm,top=3cm,bottom=3cm]{geometry}
\usepackage{setspace}
\usepackage{color}
\usepackage{float}
\usepackage{graphicx} 
\usepackage{tikz}
\usepackage{hyperref}
\hypersetup{
    colorlinks,
    citecolor=black,
    filecolor=black,
    linkcolor=black,
    urlcolor=black
}
\allowdisplaybreaks
\begin{document}
\begin{spacing}{1}
\bibliographystyle{alpha}
\newtheoremstyle{plain}
  {15pt}   
  {15pt}   
  {\itshape}  
  {0pt}       
  {\bfseries} 
  {.}         
  {5pt plus 1pt minus 1pt} 
  {}          
\newtheorem{thm}{Theorem}[section]
\newtheorem{lem}[thm]{Lemma}
\newtheorem{cor}[thm]{Corollary}
\newtheorem{prop}[thm]{Proposition}
\newtheoremstyle{my_definition}
  {15pt}   
  {15pt}   
  {\upshape}  
  {0pt}       
  {\bfseries} 
  {.}         
  {5pt plus 1pt minus 1pt} 
  {}          
\theoremstyle{my_definition}
\newtheorem{defi}[thm]{Definition} 
\newtheorem{rem}[thm]{Remark}
\newtheorem{exa}[thm]{Example}
\thispagestyle{empty}
\begin{flushright}
\includegraphics[width=0.25\textwidth]{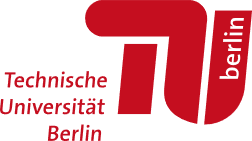}
\end{flushright}
\begin{center}
\vspace*{2cm}
\huge
Percolation phase transitions for the \\
SIR model with random powers\\
\Large
\vspace{2cm}
Masterarbeit in Mathematik\\
\vspace{2cm}
Juni 2019\\
\vspace{4cm}
\end{center}
\begin{tabular}{ll}
Vorgelegt von: & Regine Löffler\\
                     & Matrikelnummer: 396671\\
                     & \\
Erstprüfer:      & Prof. Dr. Wolfgang König \\
                    & Institut für Mathematik \\
                    & TU Berlin \\
                     & \\
Zweitprüfer:      & Dr. Benedikt Jahnel \\
                    & WIAS Berlin \\
\end{tabular}
\newpage
\thispagestyle{empty}
\quad
\newpage
\pagenumbering{Roman}
\large
\noindent Eigenständigkeitserklärung\\

\normalsize
\noindent
Hiermit erkläre ich, dass ich die vorliegende Arbeit selbstständig und eigenhändig sowie ohne unerlaubte fremde Hilfe und ausschließlich unter Verwendung der aufgeführten Quellen und Hilfsmittel angefertigt habe.  
\vspace{2cm}
\begin{flushright}
Berlin, den 7.\,Juni 2019 \,\,\,\,\,\,
\end{flushright}

\newpage
\thispagestyle{empty}
\quad
\newpage
\newpage
\large
\noindent Abstract\\

 \normalsize
\noindent
This thesis considers three models which describe a \textit{multi-hop ad-hoc system}, a certain type of telecommunication systems. 
It means that they consist of users who are sending messages, which can jump to other users in the system in order to reach the target user.
The basic structure of this thesis is to study each of the three models in a separate chapter.
Thereby, the first two models have already been examined extensively, whereas it is the first time the third model is studied.
\medskip
\\ \indent
In all the models our goal is to understand under which conditions users, who are far away from each other, are able to communicate. 
This connectivity problem is the fundamental question of \textit{continuum percolation}, which was introduced in 1961 by Gilbert. He established the main mathematical model which describes a multi-hop ad-hoc system: the Boolean model for Poisson point processes.
\medskip
\\ \indent
In the first part of this thesis we start with introducing the Poisson point process and the Boolean model. 
Afterwards, the proof of the existence of a phase transition between a subcritical phase, where no infinite connected component of users exists, and a supercritical phase, where such a component exists, is given.
First, we study the case of constant radii within the Boolean model as done by Gilbert, then we consider the case of random radii.
We show the latter proof as done in the book 'Continuum percolation' (1996) of Meester and Roy. According to them, the paper \cite{Hall1985} is one of the first ones which focuses on the study of this percolation problem.
\medskip
\\ \indent
In the second part of the thesis we consider the SINR model for Cox point processes.
It has been extensively studied by Tóbiás within his thesis 'Message routeing and percolation in interference limited multihop networks' from 2019. 
He proves that there exists a phase transition under certain stability and connectedness conditions on the intensity measure. 
Note that his work is the first one which studies SINR percolation in dimensions $d \geq 3$ and not only for $d=2$. 
The SINR model for homogeneous Poisson point processes in $\mathbb{R}^2$ has been introduced and studied in \cite{Dousse2005, Dousse2006, Franceschetti2007}.
\medskip
\\ \indent
In the third part of this thesis we study continuum percolation in the SINR model with random powers for Poisson point processes.
Note that the model has already been studied by \cite{KongYeh2007} for Poisson point processes in $\mathbb{R}^2$ and under 
strong boundedness conditions on the random power. In this thesis we weaken these conditions and only assume that the random power is nonnegative, does not equal zero with probability one and has finite expectation. 
We show that there exists a subcritical phase if the path-loss function has a strong enough decay. Furthermore, we prove the existence of a supercritical phase under the condition that the random power exceeds a certain large enough constant with positive probability and some exponential moments are finite.

\newpage
\large
\noindent Zusammenfassung\\

\normalsize
\noindent
Die vorliegende Masterarbeit befasst sich mit drei Modellen, die eine bestimmte Art von Telekommunikationssystemen beschreiben:
Es besteht aus Nutzern, die einander Nachrichten senden, welche auf dem Weg zu ihrem Ziel über andere Nutzer weitergeleitet werden können.
Diese Arbeit ist so strukturiert, dass jedes Modell in einem eigenen Kapitel vorgestellt und analysiert wird.
Die ersten beiden Modelle wurden bereits ausführlich analysiert, das dritte Modell wird im Rahmen dieser Arbeit das erste Mal in der vorliegenden Art und Weise untersucht.
\medskip
\\ \indent
Unser Ziel ist es zu verstehen, unter welchen Bedingungen Nutzer, die weit voneinander entfernt sind, kommunizieren können. 
Dieses Verbindungsproblem ist die fundamentale Frage der \textit{kontinuierlichen Perkolation}, die 1961 von Gilbert eingeführt wurde. Er etablierte das mathematische Grundmodell, welches ein wie oben beschriebenes Telekommunikationssystem modeliert: das Boolesche Modell für Poisson-Punktprozesse.
\medskip
\\ \indent
Im ersten Teil dieser Arbeit führen wir Poisson-Punktprozesse und das Boolesche Modell ein. 
Anschließend zeigen wir, dass ein Phasenübergangs zwischen einer subkritischen Phase, in der keine unendlich große verbundene Komponente von Nutzern vorhanden ist, und einer superkritischen Phase, in der eine solche Komponente vorhanden ist, existiert.
Zuerst untersuchen wir dabei das Boolesche Modell mit konstanten Radien, welches von Gilbert stammt, anschließend betrachten wir 
das Boolesche Modell mit zufälligen Radien.
Wir zeigen den Perkolationsbeweis im letztgenannten Modell wie in dem Buch "Continuum percolation" (1996) von Meester und Roy. Laut ihnen ist  \cite{Hall1985} einer der ersten, der sich auf die Untersuchung dieses Perkolationsproblems konzentriert.
\medskip
\\ \indent
Im zweiten Teil der Arbeit betrachten wir das SINR-Modell für Cox-Punktprozesse.
Es wurde von Tóbiás im Rahmen seiner Dissertation "Message routeing and percolation in interference limited multihop networks" (2019) eingehend untersucht. 
Er zeigt, dass unter bestimmten Stabilitäts- und Vernetzungsbedingungen an das Intensitätsmaß ein Phasenübergang existiert. 
Seine Arbeit ist die erste, die SINR-Perkolation in den Dimensionen $d \geq 3$ untersucht und nicht nur für $d=2$. 
Das SINR-Modell für homogene Poisson-Punktprozesse in $\mathbb{R}^2$ wurde in \cite{Dousse2005, Dousse2006, Franceschetti2007} eingeführt und analyisert.
\medskip
\\ \indent
Im dritten Teil dieser Arbeit untersuchen wir kontinuierliche Perkolation im SINR-Modell mit zufälligen Sendungsleistungen für Poisson-Punktprozesse.
Das Modell wurde bereits von \cite{KongYeh2007} für Poisson-Punktprozesse in $\mathbb{R}^2$ und unter der Annahme, dass die zufällige Sendungsleistung beschränkt ist, untersucht. 
In dieser Arbeit schwächen wir diese Bedingungen ab und gehen nur davon aus, dass die zufällige Sendungsleistung nicht negativ ist, nicht gleich null mit Wahrscheinlichkeit eins ist und eine endliche Erwartung hat. 
Wir zeigen, dass es eine subkritische Phase gibt, wenn die Pfad-Verlust-Funktion stark genug abfällt. Darüber hinaus beweisen wir die Existenz einer superkritischen Phase unter der Bedingung, dass die zufällige Leistung eine bestimmte, ausreichend große Konstante mit positiver Wahrscheinlichkeit überschreitet und einige Exponentialmomente endlich sind.

\tableofcontents
\newpage
\thispagestyle{empty}
\quad
\newpage
\setcounter{page}{1}
\pagenumbering{arabic}
\section{Introduction}
In this thesis, we consider three models which all describe a certain type of telecommunication systems: a \textit{multi-hop ad-hoc system}. 
This means that they consist of users (or devices), which like to send messages (or signals) to each other.
Thereby, a message does not have to jump directly back to a base station and afterwards to the user, to whom it is adressed. Instead, it can jump to any other user in the system in order to reach its target user. The advantages of such direct message trajectories are, that there is no need for expensive base stations and the users might be able to receive and send more messages. \medskip
\\ \indent
The basic structure of this thesis is to study each of the three models in a separate chapter.
Thereby, the idea is to present first two models, which already have been examined extensively, and then, to define and study a third model on our own. In order to be able to refer to these models, we shortly give the names of the different models: In the second chapter we consider the \textit{Boolean model}, in the third chapter we study the \textit{SINR model} and the last model, which is analysed in the fourth chapter, is called \textit{SINR model with random powers}. 
\medskip
\\ \indent
The target of this introduction is to provide an overview of the topics which are discussed in this thesis. Furthermore, we give information about previous and related work on which this thesis is based on. Therefore, we briefly introduce the most important definitions and results. Note that they are explained in more detail in the corresponding chapters.
\medskip
\\ \indent
Before we study how the models are defined, we briefly explain in the following two possible sources of randomness in a multi-hop ad-hoc system which we described above. The first one is the location of the users. 
This source of randomness is included in all our models presented in this thesis. 
We always assume that the user locations are modeled via a random point process, more precisely, we consider two different ones in this thesis. The basic one is the \textit{Poisson point process (PPP)} $X^\lambda= (X_i)_{ i \in I}$, where $\lambda >0$ is called \textit{intensity} and describes the expected (or average) number of Poisson points in a bounded region. 
The advantage of the PPP is, that it has a high degree of independence and even though the process itself is very elementary, the mathematical study can become highly sophisticated. 
The PPP is used in the Boolean model and the SINR model with random powers in this thesis (Chapter 2 and 4) and hence introduced in the beginning of Chapter 2.
The more advanced point process is the \textit{Cox point process (CPP)} $X^\lambda= (X_i)_{ i \in I}$ with intensity measure $\lambda \Lambda$, where $\lambda >0$ and $\Lambda$ is a stationary random measure on $\mathbb{R}^d$, $d \geq 2$. 
Conditioned on this random measure $\Lambda$, the CPP $X^\lambda$ is a Poisson point process with intensity $\lambda \Lambda$.
Even though the mathematical analysis of a CPP becomes more complex, it is still interesting to study this kind of point process, because it enables us to model the telecommunication area in a more realistic way. More precisely, we can incorporate a random environment in our model. For example, it is possible to include random street systems.
In this thesis, the SINR model (Chapter 3) considers the CPP in order to model the locations of the users. The CPP is introduced in the beginning of Chapter 3.
\medskip
\\ \indent
The second source of randomness lies in the message trajectories. This means that it is possible to include a random factor in the message transmission between users in order to model a more realistic situation. In this thesis, not all models incorporate this source of randomness.
In our first model, which is the Boolean model, we consider the case of non-random message trajectories (see Chapter 2.2.2), as well as the case where they are random (see Chapter 2.2.3). 
In the SINR model, the message trajectories are not random, whereas - as the name suggests - the SINR model with random powers incorporates random message trajectories. 
Note that there are several ways of including a random message transmission in the different models (random radii or random powers), but they all have in common that the strength of the emitted signal from a user becomes random. 
\medskip
\\ \indent
In difference to a model with base stations, the position of the users in space, which cannot be controlled as we just explained, is more relevant for a successful message transmission. For example it is interesting to examine if every user has another user close enough such that she can send her a message. In the next step, one can study if it is possible to send a message even further to other users until it reaches the target user. 
Another question is, how long such message trajectories can become.
All together, we can say that two of the most fundamental topics concerning spatial telecommunication models are \textit{coverage} and \textit{connectivity}.
The first topic studies how much of the area can be reached by the signals, which are sent from the users.
The second topic, on which we focus in this thesis, examines how far a message can travel through the above described multi-hop ad-hoc system.
This is the fundamental question of \textit{continuum percolation}, which was introduced in 1961 by Gilbert in \cite{Gilbert1961}. He established the main mathematical model which describes a multi-hop ad-hoc system: the Boolean model. 
Thereby we interpret the points of a homogeneous PPP $X^\lambda$ in $\mathbb{R}^2$ with intensity $\lambda >0$ as the users of the network.
Then, we put closed areas around these users and assume that a user can successfully send a message to another user if she is located within her attached area. Such areas could for example be closed balls with a constant radius $r>0$.
Two points, respectively users, are then connected if their distance is smaller than $r$. The resulting graph is called \textit{Gilbert's graph} and two users are able to communicate along the edges of this graph. 
\medskip
\\ \indent
In order to analyse the connectivity in this model we introduce the notion of \textit{percolation} and say that the graph \textit{percolates}, if there exists an infinite connected component and far away users are able to communicate.
Gilbert has shown that there exists a phase transition in the Boolean model with constant radii. This means that there exists a critical density depending on the radius $\lambda_{\mathrm{c}}(r) \in (0,\infty)$ such that, almost surely, for $\lambda<\lambda_{\text{c}}(r)$, the graph does not percolate and for $\lambda >\lambda_{\text{c}}(r)$, the graph percolates. 
The phase, where $\lambda<\lambda_{\text{c}}(r)$ is called the \textit{subcritical phase} or \textit{subcritical regime}, whereas for $\lambda>\lambda_{\text{c}}(r)$ we have a so called \textit{supercritical phase} or \textit{supercritical regime}.
A precise definition of the Boolean model and the proof of the phase transition in the case where the radii are constant is given in the Chapters 2.2.1 and 2.2.2.
Note, that it is possible to show Gilbert's result as well for larger dimensions,  which means for a PPP in $\mathbb{R}^d$, where $d \geq2$.
Thorough analyses and generalizations of Gilbert's model for Poisson point processes can be found for example in \cite{MeesterRoy1996}, \cite{Franceschetti2007} and \cite{Baccelli2009}.
The model has as well been studied for other point processes. In \cite{Blaszczyszyn2010, Blaszczyszyn2013} the authors consider for example sub Poisson point processes, \cite{Ghosh2016} examines Ginibre and Gaussian zero point processes and \cite{Stucki2013, Jansen2016} study Gibbsian point processes.
\medskip
\\ \indent
As we aim to model a telecommunication system in a realistic way, we consider in this thesis as well the more advanced CPP,  which we already introduced before. The CPP enables us to create a random environment in which the users are located.
In the paper 'Continuum percolation for Cox point processes' from 2018, the authors Hirsch, Jahnel and Cali 
consider the Boolean model with a CPP $X^\lambda$ with intensity measure $\lambda \Lambda$ on $\mathbb{R}^d$, where $d \geq 2$. They show that under certain conditions on the stability and the connectedness of the random measure $\Lambda$, there exists as well a phase transition.
More precisely, it is shown in \cite{HirschJahnelCali2018}, that, if $\Lambda$ is stabilizing, the critical density is strictly positive, which means that there exists a subcritical phase. Furthermore, it is proven that the critical density is finite, i.e., there exists a supercritical phase, if $\Lambda$ is not only stabilizing, but as well asymptotically essentially connected, which is a stronger assumption on the random measure.
In this thesis we do not present the results of \cite{HirschJahnelCali2018} in detail, but we explain them briefly in Chapter 3.2, where they are used in order to prove the subcritical phase within the SINR model.
\medskip
\\ \indent
Before we continue to explain this SINR model, we first consider a small adaption of the Boolean model, which is studied in Chapter 2.2.3: The Boolean model with random radii.
Thereby, we consider in addition to the PPP $X^\lambda = (X_i)_{ i \in I}$, which represents the locations of the users, another random point process $\rho= (\rho_i)_{ i \in I}$.
The latter point process expresses the random radii of the closed balls, which are attached to the Poisson points. Note that we assume that the points of the point process $\rho$ are independent of each other and independent of the PPP $X^\lambda$.
As the points of $\rho$ represent a radius, it makes sense to assume that $\rho \geq 0$. Furthermore, we do not want to analyse a trivial model and hence we only consider the case where the probability that $\rho$ equals zero is strictly smaller than $1$.
Then, one can show that there exists a subcritical phase if $\mathbb{E}[\rho^{2d-1}]$ is finite. Furthermore, there exists a supercritical phase if $\mathbb{E}[\rho^d]$ is finite. This is an interesting result as these conditions are not very strong.
In this thesis, we show the proof of the just described phase transition as done in the book 'Continuum percolation' (1996) by Meester and Roy. According to them, the paper \cite{Hall1985} is one of the first ones, which focuses on the study of continuum percolation within a Boolean model with random radii (see \cite[page 90]{MeesterRoy1996}).
\medskip
\\ \indent
Another topic that arises when we try to model a telecommunication system is the interference from other users. In order to understand this problem, let us consider a space with only two users who are close enough to each other to send a message. If we now include lots of other users next to them, which are as well sending messages at the same time, it could happen that the message transmission between the first two users is no longer possible.
The model which analyses this problem is the SINR model. It is formally introduced in Chapter 3.2 of this thesis.
As in the Boolean model, we express the position of the users in space via a random point process $X^\lambda$ in $\mathbb{R}^d$ for $d \geq 2$. 
Again, we define in which case a message transmission between these user is successfull and in which it is not. 
Remember that in the Boolean this was expressed via the radius parameter. In the SINR model the situation becomes a bit more complex as we like to incorporate the interference from the other users which we explained above. 
Let us consider two points of our random point process $X^\lambda$, where point $X_i$ would like to send a signal to point $X_j$.
Then we define the signal to interference and noise ratio (SINR) as follows:
\begin{equation*}
\text{SINR}(X_i, X_j, X^\lambda, P, N_0, \gamma):=
\frac{P \cdot \ell(|X_i-X_j|)}{N_0+\gamma \sum\limits_{k \neq i,j} P \cdot \ell(|X_k-X_j|)}.
\end{equation*}
Thereby, the parameter $P>0$ represents a fixed signal power, the parameter $N_0 \geq 0$ expresses a constant environment noise and 
$\gamma \geq 0$ is a technology factor, which expresses how much weight is given to the interference compared to the signal.
The so-called path-loss function
$\ell \colon [0, \infty) \rightarrow [0, \infty)$ is a decreasing function and describes the propagation of the signal strength over distance.
Note that the interference term is given by the following sum over all points of the point process: $\sum\limits_{k \neq i,j} P \cdot \ell(|X_k-X_j|)$.
Now, we connect the two points $X_i$  and $X_j$ if the SINR between them is larger than a given technical threshold $\tau>0$ in both directions. The resulting graph is called SINR graph.
We can analyse this SINR model for various point processes. 
As explained before, the basic mathematical approach is the PPP. The SINR model with a homogeneous PPP in $\mathbb{R}^2$ has been introduced and studied in \cite{Dousse2005, Dousse2006, Franceschetti2007}.
\medskip
\\ \indent
Remember that we focus in this thesis on the analysis of the existence of a phase transition within telecommunication models. For $\lambda >0$ let us therefore write $\gamma^*(\lambda)$ for the supremum of all $\gamma >0$ for which the SINR graph percolates.
It is very easy to understand that there exists a subcritical phase in the SINR model. Remember that in a Boolean model with any radius $r>0$ the critical density $\lambda_{\text{c}}(r)$ is positive and finite. Furthermore, note that for $\gamma =0$, the SINR graph equals the Gilbert's graph with radius $r_{\text{B}}=\ell^{-1}(\tau N_0/P)$.
It follows that this Gilbert's graph contains all SINR graphs with positive $\gamma$. We can conclude that for $\lambda > \lambda_{\text{c}}(r_{\text{B}})$, we have $\gamma^*(\lambda)=0$ and hence the existence of a subcritical phase is shown.
Under suitable integrability and boundedness conditions on the path-loss function $\ell$, it is possible to show the existence of a supercritical phase. More precisely, \cite{Dousse2006} has shown that for any intensity $\lambda > \lambda_{\text{c}}(r_{\text{B}})$, we obtain $\gamma^*(\lambda)>0$.
\medskip
\\ \indent
In this thesis, we consider the SINR model for a CPP $X^\lambda$ with intensity measure $\lambda \Lambda$ on $\mathbb{R}^d$, where $d \geq 2$. It has been extensively studied by Tóbiás within the second part of his thesis 'Message routeing and percolation in interference limited multihop networks' from 2019, which is an extended version of his previous paper 'Signal to interference ratio percolation for Cox point processes' from 2018. Note that his work is the first one which studies SINR percolation in dimensions $d \geq 3$ and not only for $d=2$. 
An important part of the work of Tóbiás is the proof of a phase transition, which we present in the third chapter. Thereby, we start with the proof of the existence of a subcritical phase. Remember that \cite{HirschJahnelCali2018} has shown for the Boolean model with CPPs that there exists a subcritical phase under the condition that $\Lambda$ is stabilizing. Now, \cite{Tobias2019} has proven that there exists a subcritical phase in the SINR model with CPPs under the same stability condition. The proof can be shown easily using a similar argument as we explained above in the case of a SINR model for a PPP. 
The only difference is, that we do not base the proof on Gilbert's result, which considers a PPP, but on \cite{HirschJahnelCali2018}, which considers a CPP as we already explained. 
A formal proof is given in Chapter 3.2.
To show the existence of a supercritical phase is much more complex. 
Analogous to the Boolean model with a CPP, we thereby need a stronger connectedness condition on the random measure $\Lambda$. 
More precisely, we assume that $\Lambda$ is assymptotically essentially connected. In addition, at least one of the following assumptions has to hold true:
(1) The path-loss function $\ell$ has compact support, (2) $\Lambda(\text{Q}_1)$ is almost surely bounded, or (3) there exists a parameter $\alpha >0$ such that $\mathbb{E}[\exp(\alpha\Lambda(\text{Q}_1))]$ is finite and the path-loss function $\ell$ satisfies a certain stronger decay property. Thereby, $\text{Q}_1$ describes the cube with side length one, which is centered at the origin.
The proof of the supercritical phase is explained in detail in Chapter 3.3. 
It consists of three steps. First, we map the continuous percolation problem to a discrete percolation on a lattice. 
In the second step, we prove that the lattice percolates and in the last step, we show that this implies percolation in our original continuous model. The idea of this structure comes from \cite[Theorem 1]{Dousse2006}, whereas the discrete model has its origins in \cite{HirschJahnelCali2018}.
\medskip
\\ \indent
The third model, which we study in this thesis incorporates features of the previously introduced models and aimes to describe a more realistic telecommunication system. 
We call this model \textit{SINR model with random powers}. 
As before, the position of the users (or devices) within the telecommunication system is expressed via a random point process $X^\lambda = (X_i)_{i \in I}$. In this thesis, we consider the PPP with intensity $\lambda$. As explained before, it is easier to handle than for example the CPP, but still allows us to model a realistic situation.
As the name of the model indicates, the successfull message transmission depends again on the signal to interference and noise ratio (SINR), but instead of a fixed signal power $P$ we consider a random point process $\rho=(\rho_i)_{i \in I}$. 
We make three (not very strong) conditions on $\rho$: It is nonnegative, the probability that it equals zero is strictly smaller than one and its expectation is finite. 
Now the idea is to attach to each user $X_i$ a random signal power $\rho_i$.
Hence, the SINR is defined as
\begin{equation*}
\text{SINR}((X_i, \rho_i), (X_j, \rho_j), X^\lambda, \rho, N_0, \gamma):=\frac{\rho_i \cdot \ell(|X_i-X_j|)}{N_0+\gamma \sum\limits_{k \neq i,j} \rho_k \cdot \ell(|X_k-X_j|)}.
\end{equation*}
Apart from the signal power, the parameters are defined as before. Moreover, we connect two points $X_i$ and $X_j$ if the SINR between them is larger than a constant threshold $\tau >0$ in both directions.
A detailed definition of the SINR model with random powers is given in Chapter 4.1.
Note that the model already has been studied by \cite{KongYeh2007} for PPPs in $\mathbb{R}^2$ and under 
stronger conditions on the random power $\rho$. They assumed that it is bounded by a strictly positive minimal and a finite maximal value, both having a positive probability.
\medskip
\\ \indent
We focus in this thesis on the study of continuum percolation and hence we first analyse the existence of a subcritical phase and afterwards the existence of a supercritical phase. As we elaborated this part of the thesis on our own, we now explain a bit more detailed how the proofs are carried out.
In Chapter 4.2 we show that there exists a subcritical phase if there exist parameters $r^* \in (0,\infty)$ and $\alpha, \beta >0$ with $\alpha \beta > 2d-1$ such that $\ell(r) \leq r^{-\beta}$ and $\mathbb{P}(\rho > r) \leq r^{-\alpha}$ for $r \geq r^*$.
The conditions show that we need a strong enough decay of the path-loss function in order to prove a subcritical phase. 
Examples for path-loss functions which fulfill these conditions are those with bounded support, $\ell(r) = (1+r)^{-p}$ and $\ell(r) = \min\{1,r^{-p}\}$, where $p \geq 1$. 
We first show the existence of a subcritical phase in the SINR model with $\gamma =0$, which means we do not consider any interference. Hence, the model is sometimes denoted by SNR (signal to noise ratio) model. 
It clearly holds that all connections in the SINR graph exist in the SNR graph. Thus, if there is no infinite cluster in the latter one, there is no one in the SINR graph. So the main work is to prove the existence of a subcritical phase in the SNR model. Thereby we us the fact that our SNR model is very similar to a Boolean model with random radii. More precisely we can show that all connections of the SNR graph exist in a suitable Boolean model. In Chapter 2.2.3, we have shown that in any Boolean model with random radii, there exists a phase transition if the random radius variable satisfies an integrability condition. Hence, it suffices to show that this condition holds. Thereby we need our  above described assumptions on the path-loss function.
\medskip
\\ \indent
In Chapter 4.3 we prove the existence of a supercritical phase in our SINR model with random powers if the following two conditions hold true: First, there exists
$r \geq N_0 \tau /\ell (0)$ such that $\mathbb{P}(\rho > r) >0$ and secondly, there exists a parameter $\alpha >0$ such that $\mathbb{E} [ \exp( \alpha \rho)]$ is finite.
As in the percolation proof of the SINR model with constant powers in Chapter 3, this is done in three steps:
We first map the continuous percolation problem to a discrete percolation problem. Then, we show the existence of percolation in the discrete case and afterwards, we prove that this implies percolation in our continuous SINR model with random powers.
Let us explain in words how the steps are performed.
In the first step, when we map the continuous percolation on the lattice $\mathbb{Z}^d$, we separate two aspects of the percolation problem. First, we take care of the connectedness. Thereby we define a site of the lattice as \textit{good} if there exists a nearby Poisson point with a large enough transmission power and any two nearby Poisson points are connected by a path in a Gilbert's graph with a constant and small enough radius. Otherwise it is called \textit{bad}.
Secondly, we ensure that the interference is bounded. 
All together, we call a site of the lattice \textit{open} if it is good and has bounded interference.
In the second step of the proof, we show that the lattice percolates. Thereby, we first consider the connectedness problem and show that the probability that any pairwise distinct sites are bad can be bounded arbitrarily small.
Afterwards we continue with the interference problem and prove that the probability that the interference of any pairwise distinct sites is larger than a parameter can be bounded arbitrarily small. In the end of Step 2, we bring the connectedness and the interference problem together and use the famous Peierls argument to show the percolation of the lattice.
In Step 3, we show that this implies continuum percolation in the SINR model with random transmission powers. This is a quite technical step where we basically simply put all our definitions together.
\medskip
\\ \indent
In the end of our thesis, we briefly discuss in Chapter 5 how the models which are presented in this thesis could be modified in order to make them even more realistic.
\newpage
\section{Continuum percolation in a Boolean model}
\noindent 
In this chapter, we describe a very common model of a telecommunication system, which was  introduced in 1961 by Gilbert in \cite{Gilbert1961}: the Boolean model. 
Thereby, the users of the telecommunication system are represented by the points of a Poisson point process (PPP). 
Therefore, we first introduce the PPP and some useful characterizations of it.
Afterwards, we define how the users of the telecommunication system are able to communicate. Thereby, we first consider non-random message trajectories and afterwards random ones.
In both cases, we study how far messages can travel through the system.
\subsection{The Poisson point process}
The basic mathematical method to model random locations of point-like objects in the Euclidean space is the PPP.
It is for example used in spatial telecommunication models, where the Poisson points represent the users or devices, that send and receive messages. 
In this thesis, we will focus on this application of PPPs, but there are as well other application areas.
The main reason, why it is very common to use a PPP, is, that it has a very high degree of independence of the points and hence it is easier to derive tractable mathematical formulas.
We first introduce some basics on topology and measurability in order to give a precise definition of a PPP. 
The content of the following section can be found in \cite[page 3ff.]{JahnelKoenig2018}. \medskip
\\ \indent
Let us fix the dimension $d \in \mathbb{N}$ and a measurable set $D \subset \mathbb{R}^d$, which is called the \textit{communication area}.
Moreover, we assume that a random point process in $D $ is given. We denote it by $\mathbb{X} = (X_i)_{i \in I}$, where $I$ is a random index set.
Hence, all the points of the point process, which means in our application of a telecommunication network all the users, are located in the communication area.
We do not aim to distinguish between the different points $X_i$, $i \in I$, but we are interested in the set $\mathbb{X} = \{X_i \colon i \in I\}$. 
Hence, we would like that $X_i \neq X_j$ for any $i \neq j$ and that in any compact subset of the communication area $D$ there are only finitly many of the points of $\mathbb{X}$.
This means, we would like to have a random variable $\mathbb{X}$ with values in the set
\begin{equation*}
\mathbb{S}(D) := 
\{ A \subset D \colon \# (A \cap B) < \infty \text{ for any bounded set } B \subset D \}.
\end{equation*}
The elements of the above set are called
\textit{locally finite sets} and their point measures are so-called \textit{Radon measures}.
It means that these measures assign to any compact set a finite value.
In the following, we call an element of $\mathbb{S}(D)$ a \textit{point cloud in $D$}.
Note that we can equivalently consider the point measure $\sum\limits_{i \in I} \delta_{X_i}$ instead of the set $\mathbb{X} = \{X_i \colon i \in I\}$. The point measure is as well called \textit{Dirac measure} and defined as $\delta_{X_i}(A) = \mathds{1}_A(X_i)$ for $A \subset D$.
Thereby, $\mathds{1}_A (\cdot)$ describes the indicator function on $A$, i.e., we have 
$\mathds{1}_A(z) = 1$ if $z \in A$ and $\mathds{1}_A (z) = 0$ otherwise.
\medskip
\\ \indent
In order to define the distribution of a \textit{random} point cloud 
$\mathbb{X} = (X_i)_{i \in I}$ in $D$, which is a random point process with values in 
$\mathbb{S}(D)$, we need a measurable structure on its state space.
The one, we introduce in the following, is a Borel $\sigma$-algebra, and hence it suffices to define topologies.
As we can consider $\mathbb{S}(D)$ as a set of point measures, it is natural to describe a topology by testing elements of $\mathbb{S}(D)$ against a suitable class of functions. 
Thereby, the idea is, to test the point cloud only in local areas.
In this thesis, we choose therefore the set of measurable functions from $D$ to $\mathbb{R}$ with compact support. This set of test functions is denoted by $\mathcal{M}(D)$.
Hence, we consider the functional
\begin{equation}\label{welle}
S_f(x) = 
\Big\langle \, f, \sum\limits_{i \in I} \delta_{x_i} \, \Big\rangle
=
\int f(y) \sum\limits_{i \in I} \delta_{x_i} (\text{d}y)
=
 \sum\limits_{i \in I} \int f(y)\delta_{x_i} (\text{d}y)
=
 \sum\limits_{i \in I} f(x_i),
\end{equation}
where $f \in \mathcal{M}(D)$ and $x = (x_i)_{i \in I} \in \mathbb{S}(D)$.
With these information we are now able to define a measurable structure on the set $\mathbb{S}(D)$.
\begin{defi}
The smallest topology on $\mathbb{S}(D)$, such that, for any function $f \in \mathcal{M}(D)$, the map 
$x \mapsto S_f(x)$ is continuous, is called \textit{$\tau$-topology}. 
\end{defi}
Note that the $\tau$-topology is not the only measurable structure one can use in the theory of point processes. 
It is possible to choose in the above described setting another set of test functions, for example the set of continuous functions from $D$ to $\mathbb{R}$ with compact support. But it is as well possible to establish the theory of point processes in a general measurable space
without any reference to topologies, see for example \cite{Last2017}.
Nevertheless, in this thesis we aim to keep the model as simple as possible and consider therfore the $\tau$-topology.
There are a few useful characterizations of it, some of them are introduced in the following.
Remember thereby that our goal is to understand how the random point cloud $\mathbb{X}$ is distributed in the communication area $D$. We first introduce the definitions which help us to define the PPP.
\begin{defi}
For a given point cloud $x = (x_i)_{i \in I}$ in $\mathbb{S}(D)$ and a given measurable set $A$ in $D$, 
we denote the \textit{number of points of $x$ in $A$} by
\begin{equation*}
N_x(A) = \# \{i \in I \colon x_i \in A \}
= S_{\mathds{1}_A} (x) \in \mathbb{N}_0 \cup \{ \infty \}.
\end{equation*}
\end{defi}

With the above definition we can describe the measure $\mu$ on $D$, which gives a first idea of how many points of $\mathbb{X}$ are located in a given set.
It is defined as follows.
\begin{defi}
The \textit{intensity measure} $\mu$ of random point cloud $\mathbb{X}$ in $D$ is defined by 
\begin{equation*}
\mu (A) = \mathbb{E} [N_{\mathbb{X}} (A) ], \,\,\, A \subset D \text{ measurable.}
\end{equation*}
\end{defi}
Now, as we have introduced the topological structure of our model, we are able to give the formal definition of a PPP.
\begin{defi}
Let $\mu$ be a Radon measure on $D$. A random point process $\mathbb{X}$ is called a 
\textit{Poisson point process (PPP) with intensity measure $\mu$} if, for any $k \in \mathbb{N}$ and any pairwise disjoint bounded measurable sets $A_1,...,A_k \subset \mathbb{R}^d$, the counting variables $N_\mathbb{X}(A_1),...,N_\mathbb{X}(A_k)$ are independent Poisson distributed random variables with parameters $\mu(A_1),...,\mu(A_k)$. 
\newpage
\noindent
That means, if
\begin{equation*}
\mathbb{P}(N_\mathbb{X}(A_1) = n_1, ..., N_\mathbb{X}(A_k) = n_k ) = 
\prod\limits_{i=1}^k \Big( \mathrm{e}^{-\mu(A_i)} \frac{\mu(A_i)^{n_i}}{n_i !}\Big) 
\end{equation*}
for any parameteres $n_1,...,n_k$ in $ \mathbb{N}_0$.
\end{defi}
Remember that within our model the Poisson points represent the locations of the users in space. As mentioned before, it is possible to choose therefore as well other random point processes, for example the Cox point process, which is introduced in Chapter 3.1.
Hence we give the following useful characterizations, if possible, for any random point cloud $\mathbb{X}$ with values in $\mathbb{S}(D)$. Of course they hold in particular for the Poisson point process. 
\medskip
\\ \indent
First note that the distribution of a random point cloud $\mathbb{X}$ is not fully characterized by its intensity measure, which we introduced above. 
It gets more precise when we consider a special sort of point clouds, the so-called stationary random point clouds, as we can see in the following. 
\begin{defi}
A random point cloud $\mathbb{X}$ is called \textit{stationary} or \textit{homogeneous} if its distribution is identical to the one of $\mathbb{X}+z$ for any $ z \in \mathbb{R}^d$.
\end{defi}
Hence, stationarity means that the distribution is invariant under spatial shifts.
It follows immediately that the communication area of a stationary point cloud must be the whole space, i.e., $D = \mathbb{R}^d$.
Furthermore, it holds that its intensity measure equals a multiple of the Lebesgue measure as stated in the following lemma.
\begin{lem}
Let $\mathbb{X}$ be a random point cloud in $\mathbb{R}^d$ with intensity measure $\mu$ such that
$\mu  ( [0,1]^d ) < \infty$.
If $\mathbb{X}$ is stationary, then it holds
$\mu = \lambda \cdot \mathrm{Leb}$ with $\lambda = \mu  [0,1]^d$.
\end{lem}
This means, that the intensity measure of a stationary point process can be described by just one value $\lambda  > 0$.
The parameter $\lambda$ is called \textit{intensity} of the stationary point process. \medskip
\\ \indent
In the following we consider two other important characterizations of the distribution of $\mathbb{X}$, the first one is given in terms of the so-called \textit{Laplace transform}, the second one is the Campbell's theorem.
\begin{defi}
For a random variable $\mathbb{X}$ with values in $\mathbb{S}(D)$ and a measurable function $f \colon \mathbb{R}^d \rightarrow [0, \infty)$, the \textit{Laplace transform} is given by
\begin{equation*}
\mathcal{L}_{\mathbb{X}}(f) = 
\mathbb{E}  \Big[ \mathrm{e}^{- \sum\limits_{i \in I} f (X_i) } \Big]
\in [0,1].
\end{equation*}
\end{defi}
Note that we can combine this definition with equation (\ref{welle}) and get another common formula for the Laplace transform, which is
$\mathcal{L}_{\mathbb{X}}(f)  =  \mathbb{E}  [ \mathrm{e}^{- S_f(\mathbb{X})} ]$.
\begin{lem}\label{laplace_unique}
The distribution of a $\mathbb{S}(D)$-valued random variable $\mathbb{X}$ is uniquely determined by its Laplace transform $\mathcal{L}_{\mathbb{X}}(f)$ for all measurable nonnegative functions $f$ with compact support.
\end{lem}
\begin{thm}[Campbell's theorem]\label{Campbell}
Let $\mathbb{X}$ be a point process in $D$ with intensity measure $\mu$ and let $f \colon D \rightarrow \mathbb{R}$ be an integrable function with respect to $\mu$. Then it holds
\begin{equation*}
\mathbb{E} [S_f(\mathbb{X})] =
\int_D f(x) \, \mu(\text{d}x).
\end{equation*}
If $\mathbb{X}$ is even a Poisson point process and the function $f$ is furthermore nonnegative, then it holds
\begin{equation*}
\mathcal{L}_{\mathbb{X}} (f) =
\mathbb{E} [ \mathrm{e}^{-S_f(\mathbb{X})} ] =
\exp \Big( \int_D \Big(\mathrm{e}^{-f(x)} -1 \Big) \mu( \text{d}x) \Big).
\end{equation*}
\end{thm}
\indent
In the next sections, we consider the Boolean model, which can be seen as a so-called \textit{marked} PPP. 
Therefore, we give in the following some information about the mathematical background.
The idea is, to attach to each point of the point process some individual information. This information may either be static or random. If we interpret the points again as users who want to send and receive messages, the information could for example be the signal strength of emitted messages of the user.
The additional information, which is attached to a point $x_i$ is called a \textit{mark} $m_i$.
The emerging process is called \textit{marked point process} and it is a random point process with values in the set
\begin{equation}\label{jaguar}
\mathbb{S}(\mathcal{M} \times D) =
\{ A \subset \mathcal{M} \times D \colon \# (A \cap B) < \infty \text{ for any bounded set } B \subset \mathcal{M} \times D \},
\end{equation}
where $\mathcal{M} $ denotes the set of marks.
An element of \,$\mathbb{S}(\mathcal{M} \times D)$ is written as $\{(x_i, m_i) \colon i \in I\}$
or as the point measure $\sum\limits_{i \in I} \delta_{(x_i,m_i)}$.
Note that equation (\ref{jaguar}) requires a topological structure of the set of marks $\mathcal{M}$. 
Let us therefore equip it with the corresponding Borel $\sigma$-algebra $\mathcal{B}(\mathcal{M})$ and assume that $\mathcal{M}$ is locally compact.
Then we call $(\mathcal{M},\mathcal{B}(\mathcal{M}))$ the \textit{mark space}.
We would like that the marks are spatially independent, which means, that the mark $m_i$ which is attached to the point $x_i$ does not depend on another mark $m_j$ corresponding to the point $x_j$, where $i \neq j$.
Therefore, we need a so-called \textit{probability kernel}
\begin{equation*}
K \colon D \times \mathcal{M} \rightarrow [0,1].
\end{equation*}
This probability kernel has two features.
First, it holds for any $x \in D$ that the map
$K(x, \cdot) \colon \mathcal{M} \rightarrow [0,1]$
is a probability measure on the mark space $(\mathcal{M},\mathcal{B}(\mathcal{M}))$.
Secondly, $K( \cdot, G)$ is measurable for any $G \in \mathcal{B}(\mathcal{M})$.
Then, $K(x_i, \cdot)$ is the distribution of the mark which is attached to the point $x_i$.
This distribution may depend on the point $x_i$, but not on the index $i$, which means we obtain spatial indepence of the marks.
With this topological structure, we are able to define a marked PPP.
\begin{defi}
Let $\mathbb{X} = (X_i)_{i \in I}$ be a PPP in $D$ with finite intensity measure $\mu$ and let $(\mathcal{M},\mathcal{B}(\mathcal{M}))$ be a measurable space.
Furthermore, let $K$ be a probability kernel from $D$ to $\mathcal{M}$. For the given PPP $\mathbb{X}$, let $(m_i)_{i \in I}$ be an independent collection of $\mathcal{M}$-valued random variables with distribution
$\otimes_{i \in I} K(X_i, \cdot)$.
Then, we call the point process
\begin{equation*}
\mathbb{X}_K = ((X_i, m_i))_{i \in I} \,\,\, \text{   in   } \,\,\, D \times \mathcal{M}
\end{equation*}
a \textit{K-marked Poisson point process (K-MPPP)} or a \textit{K-marking of the PPP $\mathbb{X}$}.
\end{defi}
The Laplace transform of a K-marked PPP $\mathbb{X}_K = ((X_i, m_i))_{i \in I}$ can be expressed via the Laplace transform of the underlying PPP $\mathbb{X}$ as stated in the following lemma.
\begin{lem}\label{mark_laplace}
The Laplace transform of the K-marking $\mathbb{X}_K$ is given by
\begin{equation*}
\mathcal{L}_{\mathbb{X}_K} (g ) =
\mathcal{L}_{\mathbb{X}} (g^*),
\end{equation*}
where the function $g \colon D \times \mathcal{M} \rightarrow [0, \infty)$ is measurable and compactly supported and
\begin{equation*}
g^* (x) = - \log \Big( \, \, \int\limits_{ \mathcal{M}} \mathrm{e}^{-g(x,y)} K(x, \text{d}y) \Big), \,\,\, x \in D.
\end{equation*}
\end{lem}
\begin{proof}
By the definition of the Laplace transform, we have
\begin{equation}\label{def_lap}
\mathcal{L}_{\mathbb{X}_K} (g )
=
\mathbb{E} \Big[ \mathrm{e} ^{- \sum\limits_{i \in I} g(X_i,m_i)} \Big]
=
\mathbb{E} \Big[ \prod\limits_{i \in I}  \mathrm{e} ^{- g(X_i,m_i)} \Big].
\end{equation}
Now, we apply the so-called tower property of conditional expectations and use the independence over i. This yields
\begin{equation}\label{tower}
\mathbb{E} \Big[ \prod\limits_{i \in I}  \mathrm{e} ^{- g(X_i,m_i)} \Big]
=
\mathbb{E} \Big[ \mathbb{E} \Big[ \prod\limits_{i \in I}  \mathrm{e} ^{- g(X_i,m_i)} \Big\vert \mathbb{X} \Big]\Big]
=
\mathbb{E} \Big[  \prod\limits_{i \in I} \mathbb{E} \Big[ \mathrm{e} ^{- g(X_i,m_i)} \Big\vert \mathbb{X} \Big]\Big].
\end{equation}
Note that for any $i \in I$ we have by definition
\begin{equation*}
\mathbb{E} \Big[ \mathrm{e} ^{- g(X_i,m_i)} \Big\vert \mathbb{X} \Big]
=
\int\limits_{\mathcal{M}} \mathrm{e}^{-g(X_i,y)} K(X_i, \text{d}y)
=
\mathrm{e}^{-g^*(X_i)}.
\end{equation*}
Inserting this in equation (\ref{tower}) and using again the definition of the Laplace transform yields the following equation
\begin{equation*}
\mathbb{E} \Big[ \prod\limits_{i \in I}  \mathrm{e} ^{- g(X_i,m_i)} \Big]
=
\mathbb{E} \Big[ \prod\limits_{i \in I} \mathrm{e}^{-g^*(X_i)} \Big]
=
\mathbb{E} \Big[ \mathrm{e}^{-\sum\limits_{i \in I}g^*(X_i)} \Big]
=
\mathcal{L}_{\mathbb{X}} (g^*).
\end{equation*}
Together with equation (\ref{def_lap}) we have shown the assertion, i.e., that  
$\mathcal{L}_{\mathbb{X}_K} (g ) = \mathcal{L}_{\mathbb{X}} (g^*)$.
\\
\end{proof}
As we assumed that the mark space $\mathcal{M}$ is a locally compact topological space, the previous lemma implies that a K-MPPP 
$\mathbb{X}_K$ is equal to a usual PPP on $D \times \mathcal{M}$ with intensity measure $\mu \otimes K$.
This is stated in the following theorem.
\newpage
\begin{thm}[Marking Theorem]\label{mark_thm}
Let $\mathbb{X}_K$ be a K-marked Poisson point process and assume that the set of marks $\mathcal{M}$ is locally compact and equipped with the Borel $\sigma$-algebra $\mathcal{B}(\mathcal{M})$. Then the point process $\mathbb{X}_K$ is in distribution equal to the Poisson point process on $D \times \mathcal{M}$ with intensity measure $\mu \otimes K$.
\end{thm}
\begin{proof}
In the following we write $\tilde{\mathbb{X}}$ for the PPP on $D \times \mathcal{M}$ with intensity measure $\mu \otimes K$.
As stated in Lemma \ref{laplace_unique} it holds that the distribution of a point process in uniquely determined by its Laplace transform.
Hence, it suffices to show that the Laplace transform of the PPP $\tilde{\mathbb{X}}$ is equal to the one of the K-marked PPP $\mathbb{X}_K$.
\medskip
\\ \indent
From Lemma \ref{mark_laplace} we know that the Laplace transform of $\mathbb{X}_K$ is given by
\begin{equation*}
\mathcal{L}_{\mathbb{X}_K} (g ) =
\mathcal{L}_{\mathbb{X}} ( g^*),
\end{equation*}
where
\begin{equation*}
g^* (x) = - \log \Big( \, \, \int\limits_{ \mathcal{M}} \mathrm{e}^{-g(x,y)} K(x, \text{d}y) \Big) \text{ for }x \in D. 
\end{equation*}
Now we apply the Campbell's theorem \ref{Campbell} for the Poisson point process $\mathbb{X}$ and the nonnegative function $g^*$. Together with the definition of $g^*$ we obtain
\begin{align*}
\mathcal{L}_{\mathbb{X}} ( g^*)
=
\exp \Big( \, \int\limits_{D} \Big( \mathrm{e}^{-g^*(x)}-1 \Big) \mu(\text{d}x)\Big) 
=
\exp \Big( \, \int\limits_{D} \Big( \,\, \int\limits_{ \mathcal{M}} \mathrm{e}^{-g(x,y)} K(x, \text{d}y)-1 \Big) \mu(\text{d}x)\Big).
\end{align*}
As $K$ is a probability kernel, it is a probability measure on the mark space. Hence, we know that
\begin{equation*}
\int\limits_{\mathcal{M}} K(x, \text{d} y) =1.
\end{equation*}
Plugging this into the above equation yields
\begin{align*}
\exp \Big( \, \int\limits_{D} \Big(\,\, \int\limits_{ \mathcal{M}} \mathrm{e}^{-g(x,y)} K(x, \text{d}y)-1 \Big) \mu(\text{d}x)\Big)
&=
\exp \Big( \int\limits_{D} \int\limits_{ \mathcal{M}} \Big( \mathrm{e}^{-g(x,y)} -1 \Big) K(x, \text{d}y)\mu(\text{d}x)\Big)\\
&=
\exp \Big( \,\, \int\limits_{D \times \mathcal{M}} \Big( \mathrm{e}^{-g} -1 \Big) \text{d} (\mu \otimes K)\Big).
\end{align*}
We apply again the Campbell's theorem \ref{Campbell}, but this time for the PPP $\tilde{\mathbb{X}}$ and the nonnegative function $g$ and obtain
\begin{equation*}
\exp \Big( \,\, \int\limits_{D \times \mathcal{M}} \Big( \mathrm{e}^{-g} -1 \Big) \text{d} (\mu \otimes K)\Big)
=
\mathcal{L}_{\tilde{\mathbb{X}}} (g).
\end{equation*}
All together, we have $\mathcal{L}_{\mathbb{X}_K}( g) = \mathcal{L}_{\tilde{\mathbb{X}}} (g)$ and hence the assertion is shown.\\
\end{proof}
Note that the measure $\mu \otimes K$ of the marking theorem is defined by
\begin{equation*}
\mu \otimes K(B) =
\int\limits_{\tilde{B}} \mu ( \text{d} x) K(x, \hat{B}_x),
\end{equation*}
where 
$B \subset D \times \mathcal{M}$, 
$\tilde{B} := \{ x \in D \colon \exists \, y \in \mathcal{M} \colon (x,y) \in B\}$
and
$\hat{B}_x := \{y \in \mathcal{M} \colon (x,y) \in B\}$.
\subsection{The Boolean model}
\subsubsection{Model definition}
In this section we describe the Boolean model. 
The following information can be found in \cite[page 23ff.]{JahnelKoenig2018}. 
\medskip
\\ \indent
Let $X^\lambda = \{ X_i \colon i \in I\}$ be a homogeneous PPP in $\mathbb{R}^d$ with intensity $\lambda >0$ and $d \geq 2$. The points represent the users (or devices) in the spatial telecommunication system. 
Each user $X_i$ has a so-called \textit{local communication zone} 
$X_i + \Xi_i$, which describes the area around the user, within which she can send messages to other users. 
Thereby, $\Xi_i \subset \mathbb{R}^d$, $i \in I$, are independent random closed sets that do not depend on the users $X_i$ to which they are attached. The \textit{Boolean model} is then defined as the union of all local communication zones of the Poisson points, which is the random set
\begin{equation*}
\Xi_{BM} = \bigcup\limits_{i \in I} (X_i + \Xi_i).
\end{equation*}
A typical choice for the sets $\Xi_i$ are centered, closed balls. 
The radius of the balls can either be deterministic or random. 
In the next Section 2.2.2, we consider the first case, which
was introduced in 1961 by Gilbert in \cite{Gilbert1961}.
Thus, we set $\Xi_i = B_R(X_i)$ for all $i \in I$, where $R \in (0, \infty)$ is the fixed radius of the balls. 
The corresponding Boolean model is then given by
\begin{equation}\label{bool_const}
\Xi_{BM} = \bigcup\limits_{i \in I} (X_i + B_R(X_i)).
\end{equation}
Then, two users $X_i$ and $X_j$ are connected if their distance is smaller than $R$. The resulting graph is called \textit{Gilbert's graph} and is denoted by
$g_R(X^\lambda)$. Two users are able to communicate along the edges of this graph. 
\medskip
\\ \indent
Afterwards, we consider in Section 2.2.3 the more complex case, where the radius of the balls is random. \medskip
\\ \indent
In both cases, we aim to study under which conditions the Boolean model $\Xi_{BM}$ contains unboundedly large components. If this is the case, we say that the model \textit{percolates} or \textit{percolation occurs}.
As we are in the continuous space $\mathbb{R}^d$, this declaration explains the name of the theory, which is "continuum percolation".
\newpage
\subsubsection{Constant radii}
In this section, we consider the Boolean model with constant radii which we introduced in equation (\ref{bool_const}).
Let us first explain some notations. 
The set $\mathcal{C}_R(x)$ denotes the connected component of $\Xi_{BM}$ that contains the point $x \in \mathbb{R}^d$.  Note that sometimes the expression \textit{cluster} is used as a synonym for 'connected component'. A very important quantity is the \textit{percolation probability}
\begin{equation*}
\theta (\lambda, R) =
\mathbb{P}( \text{Leb}(\mathcal{C}_R( o)) = \infty).
\end{equation*}
As we consider a homogeneous PPP, it holds
\begin{equation*}
\theta (\lambda, R) = \theta (\lambda R^d)
\text{ for } \lambda, R \in (0, \infty),
\end{equation*}
where we use the notation 
$\theta ( \lambda) = \theta (\lambda, 1)$. This means that the percolation probability only depends on the value of $\lambda R^d$ and we do not loose any information when we set $R=1$ in the following. Hence, it suffices to consider $\theta(\lambda)$ when we study continuum percolation in our Boolean model with fixed radii.
Note that the function $\theta(\lambda)$ is increasing in the density $\lambda$. In other words, the larger the density of the Poisson points is, the larger is the probability that there exists an infinite connected component. Hence, we can define the \textit{critical density} as follows
\begin{equation*}
\lambda_{\text{cr}} = \inf \{ \lambda \in (0, \infty) \colon \theta (\lambda) >0 \}
= \sup\{ \lambda \in (0, \infty) \colon \theta (\lambda) =0 \}
\in [0, \infty],
\end{equation*}
where $\sup \emptyset =0$ and $\inf \emptyset = \infty$. 
One of the most important results in the field of continuum percolation is the following theorem, which states that the critical threshold is positive and finite.
\begin{thm}\label{karamel}
For any $d \in \mathbb{N} \setminus \{1 \}$, it holds 
$\lambda_{\mathrm{cr}} \in (0, \infty)$.
\end{thm}
Thus, the theorem states, that for sufficiently small densities, the percolation probability is zero, while for sufficiently large one, it is positive. More precisely, there occurs a phase transition between a so-called \textit{subcritical regime}, where no percolation occurs almost surely, and a \textit{supercritical regime}, where percolation does occur with a positive probability. 
\medskip
\\ \indent
Before we show the proof of Theorem \ref{karamel}, we first need to take a step back and consider the less complex model of discrete bond percolation.
The reason is that it is often very useful to map the continuous percolation problem to a discrete one in order to proof percolation in the continuous space. This is as well done in the case of the previous theorem and hence we first introduce the basic model of the theory of discrete bond percolation. Consider therefore the $d$-dimensional integer lattice $\mathbb{Z}^d$ and put on every edge $e$ between two neighbouring sites in $\mathbb{Z}^d$ a Bernoulli random variable $\zeta_e \in \{0,1\}$ with parameter $p\in [0,1]$. That means
$\mathbb{P}(\zeta_e=1)=p$ and $\mathbb{P}(\zeta_e=0)=1-p$.
Furthermore we assume that the random variables $(\zeta_e)_e$, which form a Bernoulli field, are independent and identically distributed.
Then we call an edge $e$ \textit{open} if $\zeta_e=1$ and \textit{closed} otherwise.
Two sites in $\mathbb{Z}^d$ are called \textit{connected}, if there is a sequence of open edges that form a path between the two sites. The connected component that contains $x \in \mathbb{Z}^d$, is denoted by $C(x)$. 
Note that our model only consists of one parameter $p$, which expresses the probability that an edge is open. Hence, we write 
$\mathbb{P}_p$ and $\mathbb{E}_p$ for the probability and the expectation in this model.
The percolation probability is then given by
\begin{equation*}
\theta(p) = \mathbb{P}_p(\#C(o)=\infty),
\end{equation*}
where $o$ denotes the origin of $\mathbb{R}^d$,
and the critical percolation probability is
\begin{equation*}
p_{\text{cr}}=\inf\{p \in [0,1] \colon \theta(p)>0\}.
\end{equation*}
The following theorem states that the critical threshold in our discrete model is positive and finite. 
\begin{thm}\label{haselnuss}
For any $d \in \mathbb{N} \setminus \{1\}$, it holds $p_{\mathrm{cr}} \in (0,1)$.
\end{thm}
\begin{proof}
First, we prove the existence of a subcritical phase, i.e., that 
\begin{equation*}
p_{\text{cr}} = \inf \{p \in [0,1] \colon \mathbb{P}_p(\#C(o)=\infty)>0\} > 0.
\end{equation*}
Therefore, we denote by $\Psi_n$ the set of self-avoiding $n$-step paths starting at the origin. 
It clearly holds for any $n \in \mathbb{N}$ that if the origin is contained in an infinite connected component, then there exists an $n$-step self-avoiding path of open edges, which starts at the origin.
Hence, we have
\begin{align*}
\mathbb{P}_p(\#C(o)=\infty)
&\leq
\mathbb{P}_p (\text{ there exists } \eta \in \Psi_n \text{ such that } \zeta_e=1 \text{ for all } e \in \eta)\\
&\leq
\sum\limits_{ \eta \in \Psi_n} \mathbb{P}_p (\zeta_e=1 \text{ for all } e \in \eta)\\
&=
\sum\limits_{ \eta \in \Psi_n} p^n
=
|\Psi_n| \cdot p^n
\leq
(2dp)^n.
\end{align*}
For $p \in \big(0, \frac{1}{2d}\big)$, it holds that the expression $(2dp)^n$ tends to zero as $n \rightarrow \infty$, which implies
$\mathbb{P}_p(\#C(o)=\infty) = 0$. This yields that the infimum of the set 
$\{p \in [0,1] \colon \mathbb{P}_p(\#C(o)=\infty)>0\}$
has to be either equal to or larger than the value $1/(2d)$. Hence, it holds $p_{\text{cr}} \geq 1/(2d) >0$, which proves the existence of a subcritical phase. \medskip
\\ \indent
In the next step, we prove the existence of a supercritical phase, i.e., we show that $p_{\text{cr}} < 1$.
As we consider a discrete model, there does clearly exist an infinite connected component in dimension $d > 2$ if there exists one in dimension $d=2$.
Hence, it suffices to prove percolation in $\mathbb{Z}^2$.
Thereby we use the strategy of the famous \textit{Peierls argument}.
\medskip
\\ \indent
Let us consider the shifted lattice $\mathbb{Z}_*^2 = \mathbb{Z}^2 + (1/2,1/2)$.
For any edge $e^*$ in $\mathbb{Z}_*^2$, we say that it is \textit{open} if the unique edge $e$ in our original lattice $\mathbb{Z}^2$, which crosses $e^*$, is closed.
Analogously, an edge $e^*$ in $\mathbb{Z}_*^2$ is called \textit{closed}, if the unique crossing edge $e$ in $\mathbb{Z}^2$ is open.
Note that by this definition, the probability that an edge $e^*$ in the shifted lattice $\mathbb{Z}_*^2$ is open is given by $1-p$ and the probability that it is closed is given by $p$.
Moreover, it follows, that if the origin is not contained in an infinite connected component, there must exist an interface of open edges in $\mathbb{Z}_*^2$ which surrounds the origin and contains the point 
$(n+1/2, 1/2)$ for some $n \in \mathbb{N}$. 
This again implies that there exists a path of $2n+4$ open edges in $\mathbb{Z}_*^2$, which is passing the point $(n+1/2, 1/2)$ for some $n \in \mathbb{N}$. 
Hence, we have
\begin{align*}
1- \theta (p)
&\leq
\sum\limits_{n \in \mathbb{N}} \mathbb{P}_p ( \text{there exists a path of } 2n+4 \text{ open edges in }\mathbb{Z}_*^2 \text{, passing } (n + 1/2, 1/2 ) )
\\
&\leq 
\sum\limits_{n \in \mathbb{N}} (4 \cdot (1-p))^{2n+4}.
\end{align*}
Thereby, the factor $4^{2n+4}$ is the number of paths of length $2n+4$ and 
$(1-p)^{2n+4}$ is the probability that $2n+4$ edges in the shifted lattice are open.
By a simple application of the geometric series, it follows that for $p$ close to $1$, the sum above becomes smaller than $1$. This yields
$1- \theta (p) < 1$ or equivalently $\theta (p) >0$. Hence, we have shown that there exist parameters $p<1$, such that the percolation probability is positive, i.e., $p_{\text{cr}} <1$, what was to be shown.\\
\end{proof}
Furthermore, one can prove that for any $p>p_{\text{cr}}$ an infinite cluster exists almost surely. That means for all parameters $p \in [0,p_{\text{cr}})$ the model does not percolate and for $p \in (p_{\text{cr}},1]$ percolation occurs. \medskip
\\ \indent
In the following we show how the existence of the phase transition in the discrete phase can be used to prove the phase transition in the continuous space. Thereby we do not apply the above Theorem \ref{haselnuss}, but the following version for site percolation on the triangular grid in two dimensions. Hence, we show in this thesis only the phase transition in $\mathbb{R}^2$.
Let $\tilde{p}_{\text{cr}} \in [0,1]$ denote the critical percolation probability of a triangular grid.
\begin{thm}\label{baerenkind}
For $d=2$, it holds $\tilde{p}_{\mathrm{cr}} =1/2$.
\end{thm}
The reason, why we use the triangular lattice is that it has the advantage that two neighbouring hexagons always share an edge and not only a vertex, as it is the case on $\mathbb{Z}^2$.
\\

\noindent\textit{Proof of Theorem \ref{karamel} for d $=$ 2.}
Let us devide the space $\mathbb{R}^2$ into open hexagons with side length $s>0$ (up to bounderies of Lebesgue null sets). Obviously all these hexagons are centered at some points $x \in \mathbb{R}^2$. Hence, we write $A_x^s$ for the hexagons. Without loss of generality, we assume that one of the hexagons is centered at the origin. 
Now the center point $x$ of the hexagon $A_x^s$ is called \textit{open}, if it contains at least one Poisson point $X_i$ from our Poisson point process $X^\lambda = \{X_i \colon i \in I\}$. Otherwise, it is called \textit{closed}. 
Hence, the probability that $x$ is open is given by
\begin{equation*}
p_s = 1- \exp (- \lambda \, \text{Leb} (A_x^s )),
\end{equation*}
where 
$\text{Leb} (A_x^s) = 3 \sqrt{3} s^2 /2$.
As we can see in the following graphic, the center points $x$ form a triangular lattice, which we denote by $\mathcal{T}_s$. 
\begin{figure}[H]
  \centering
    \includegraphics[width=0.5\textwidth]{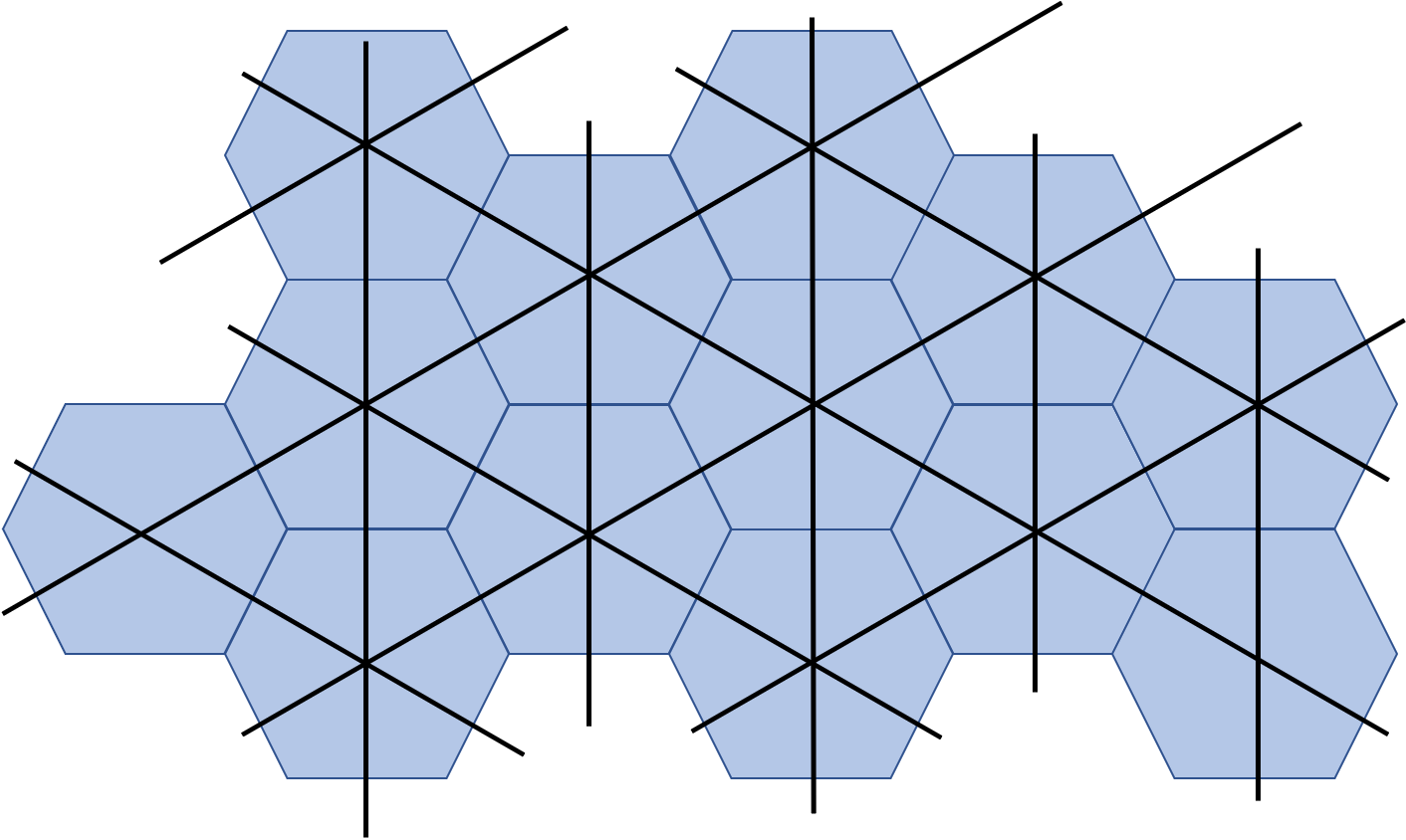}
\end{figure}
Now, we first prove that $\lambda_{\text{cr}} >0$, which means that there exists a subcritical phase.
As explained before, we assume that the fixed radius in our Boolean model is $1$. Hence, if $s>1$, which means the side-length of the hexagons is greater than $1$, it clearly holds that percolation of the Boolean model implies site-percolation on $\mathcal{T}_s$. Equivalently, it holds that no percolation on $\mathcal{T}_s$ implies no percolation in the Boolean model. From Theorem \ref{baerenkind} we know that the critical probability in the triangular lattice is greater than zero, more precisely, $p_{\text{cr}}=1/2$. This means, if we choose $p_s \in (0, 1/2)$, there is no percolation on  $\mathcal{T}_s$, which implies that there exists a subcritical phase in our Boolean model, what we aim to show. 
The following calculation shows, that it is possible to choose such a probability $p_s$.
\begin{align*}
p_s  = 1- \exp(- \lambda \, \text{Leb} (3 \sqrt{3} s^2 /2)) <1/2 
 \,\,\, &\Leftrightarrow \,\,\,
1/2 < \exp(- \lambda \, \text{Leb} (3 \sqrt{3} s^2 /2)) \\
 \,\,\, &\Leftrightarrow \,\,\,
\log (1/2) < - \lambda 3 \sqrt{3} s^2 /2 \\
 \,\,\, &\Leftrightarrow \,\,\,
-\log (2) < - \lambda 3 \sqrt{3} s^2 /2 \\
 \,\,\, &\Leftrightarrow \,\,\,
 \log (2) >\lambda 3 \sqrt{3} s^2 /2 \\
 \,\,\, &\Leftrightarrow \,\,\,
  \frac{2\log (2)}{3 \sqrt{3} s^2 }> \lambda .
\end{align*}
Hence, for all $\lambda$ in the interval $\big(0,2\log (2) (3 \sqrt{3} s^2)^{-1}\big)$, the Boolean model does not percolate.
Moreover, we know, that the critical density in the Boolean model is larger than the value $ 2\log (2) (3 \sqrt{3} s^2)^{-1}$ for any hexagon side-length, which is greater than $1$:
\begin{equation*}
\lambda_{\text{cr}} > 2\log (2) (3 \sqrt{3} s^2)^{-1} \text{ for all } s >1.
\end{equation*}
The smaller the parameter $s$ is, the larger is this lower bound for the critical density and hence it gets more precise. So we choose the smallest possible value $s=1$ and get the following lower bound for the critical density
\begin{equation*}
\lambda_{\text{cr}}  \geq\frac{ 2\log (2) }{ 3 \sqrt{3}}.
\end{equation*}
\indent
In the next step we show the existence of a supercritical phase, i.e., $\lambda_{\text{cr}} <\infty$. Thereby, we use the fact that any two Poisson points in two neighbouring hexagons are at most $\sqrt{13}s$ away from each other.
Moreover, we have that the fixed radius of the ball around each Poisson point is $1$.
Hence, if we choose the side-length of the hexagons so small, that $\sqrt{13}s < 1$, which means $s < 1/ \sqrt{13}$, site-percolation on $\mathcal{T}_s$ clearly implies percolation in the Boolean model.
From Theorem \ref{baerenkind}
we know that the triangular lattice $\mathcal{T}_s$ percolates for $p >\tilde{p}_{\text{cr}} = 1/2$.
Thus, for $p_s \in (1/2, 1)$, the triangular lattice percolates and as explained before, it follows that the Boolean model percolates.
Analogously to the proof of the subcritical phase, we can calculate that it is possible to choose $p_s > 1/2$ and obtain
\begin{align*}
p_s  = 1- \exp(- \lambda \, \text{Leb} (3 \sqrt{3} s^2 /2)) > 1/2 
 \,\,\, &\Leftrightarrow \,\,\,
  \frac{2\log (2)}{3 \sqrt{3} s^2 }< \lambda .
\end{align*}
This means, for all $\lambda$ in the interval $\big(0,2\log (2) (3 \sqrt{3} s^2)^{-1}\big)$, there exists percolation in the Boolean model.
Again, we can derive a bound for the critical density in the Boolean model, but instead of a lower bound, we obtain an upper bound. 
For all $s < 1/ \sqrt{13}$, there exists percolation in our Boolean model for any density, which is larger than the value $2\log (2) (3 \sqrt{3} s^2)^{-1}$ and hence
\begin{equation*}
\lambda_{\text{cr}} < \frac{2\log (2)}{3 \sqrt{3} s^2 } \text{ for all } s < 1/ \sqrt{13}.
\end{equation*}
The larger the parameter $s$ is, the smaller is this upper bound for the critical density, i.e., it is more precise. By choosing the largest possible value $s=1/ \sqrt{13}$ we get the following upper bound for the critical density
\begin{equation*}
\lambda_{\text{cr}}  \leq \frac{ 2\log (2) }{ 3 \sqrt{3} (1 / \sqrt{13})^2}
=
\frac{ 26\log (2) }{ 3 \sqrt{3} }
.
\end{equation*}
\indent
All together, it follows that there exists a phase transition in the Boolean model in dimension $d=2$. Moreover, we have shown that the critical density lies somewhere in the interval
$\big[ 2\log (2) ( 3 \sqrt{3})^{-1} ,  26\log (2) (3 \sqrt{3})^{-1} \big]$.
\begin{flushright}
$\quad\square $
\end{flushright}
\subsubsection{Random radii}
In this section we consider a modification of the above described Boolean model by incorporating a second source of randomness apart from the random point process. 
Remember that the radius within the Boolean model describes the strength of the emitted signal of a user. In the above section we made the simplification that each user has the same constant signal strength. In reality, however, fluctuations are very likely and hence 
we consider in this section the situation where the signal strength becomes random.
Our target is to show under which conditions there still exists a phase transition between a subcritical and a supercritical phase. The information of this section can be found in the book 'Continuum percolation' (Chapter 3.3) from 1996 of Meester and Roy. 
 \medskip
\\ \indent
Let $X^\lambda=(X_i)_{i \in I}$ be a PPP in $\mathbb{R}^d$ with intensity $\lambda >0$ and $d \geq 2$. Each point $X_i$ of $X^\lambda$ is the centre of a closed ball with a random radius $\rho_i$. The Boolean model is hence given by
\begin{equation*}
\Xi_{BM} = \bigcup\limits_{i \in I} (X_i + B_{\rho_i}(X_i)).
\end{equation*}
The radii corresponding to the different points are independent of each other and independent of $X^\lambda$. 
Furthermore, they are identically distributed as the radius random variable $\rho\geq0$. Meester and Roy restrict their study to the case where $\mathbb{E}[\rho^d]<\infty$.
Moreover, we assume that the radius random variable is not zero almost surely, i.e., $\mathbb{P}(\rho=0)<1$, otherwise the model becomes very trivial.
The region covered by the balls is called \textit{occupied region} and the connected components in the occupied region are called \textit{occupied components}.
The Boolean model driven by a PPP $X^\lambda$ with density $\lambda$ and radius random variable $\rho$ is denoted by $(X, \rho,\lambda)$.
In the following, the existence of a subcritical phase is shown as done in Theorem 3.3 in \cite[page 50]{MeesterRoy1996}.
\begin{thm}\label{subcritical_1}
For a Boolean model $(X,\rho,\lambda)$ on $\mathbb{R}^d$ with $d \geq 2$ the following holds: If 
$E[\rho^{2d-1}] < \infty$, then there exists $\lambda_0 >0$ such that for all $\lambda \in (0,\lambda_0)$, there are no infinite occupied components, i.e.
\begin{align*}
\mathbb{P}( \text{number of balls in any occupied component is finite})=1.
\end{align*}
\end{thm}
\begin{proof}
First note that we can suppose that the radius random variable $\rho$ is a positive integer. This is sufficient as for any other $\rho$, $\mathbb{E}[\rho^{2d-1}] < \infty$ implies $\mathbb{E}[\lceil\rho\rceil^{2d-1}] < \infty$, where $\lceil\rho\rceil$ denotes the smallest integer larger than or equal to $\rho$. Furthermore, if we have shown the theorem for $\lceil\rho\rceil$, it clearly holds for the smaller radius random variable $\rho$.\medskip
\\ \indent
Now, the idea of this proof is to fix some $x \in \mathbb{R}^d$ and to construct a branching process starting at this point. 
Let us therefore fix $i>0$ and consider an initial ball $S$ of radius $i$ which is centered at $x$. 
For $j \in \mathbb{N}$, the random number of balls of radius $j$ which intersect $S$ but which are not completely contained in $S$ is denoted by $n_j$.
Since $(X,\rho, \lambda)$ can be thought as the superposition of the independent Boolean models
$
\{(X_j,j,\lambda \, \mathbb{P}(\rho =j));\,\, j =0,1,2,...\},
$
the random variables $n_1,n_2,...$ are an independent sequence where 
each $n_j$ is Poisson distributed with parameter
\begin{align*}
\mu_{i,j} &:=
\mathbb{E}[\text{number of balls with radius } j \text{ centred in } \{z \in \mathbb{R}^d \colon i-j < |z-x|\leq i+j \}]
\\
&=
 \lambda \,\mathbb{P}(\rho =j)\, \pi_d \,[(i+j)^d - (\max(0,i-j))^d ].
\end{align*}
Thereby, the parameter $\pi_d$ denotes the volume of a $d$-dimensional ball with unit radius.
In the next step, we construct a branching process and say that $n_j$ is the number of children of $x$ of type $j$.
We define the infinite matrix $M$, which has $\mu_{i,j}$ as its $(i,j)$-th entry. 
The $n$-th generation of type $j$ is Poisson distributed with parameter $\mu_{i,j}^{(n)}$, which is the $(i,j)$-th entry of the matrix $M^n$. Note that the $0$-th generation of the branching process is taken to be the origin.
The expected number of members of the $n$-th generation is then given by
\begin{align*}
\sum\limits_{j=0}^\infty \mu_{i,j}^{(n)}
\end{align*}
and the expected number of total members in the entire branching process is
\begin{align*}
\mu_i :=\sum\limits_{n=0}^\infty \sum\limits_{j=0}^\infty \mu_{i,j}^{(n)}.
\end{align*}
\indent
\newpage
\noindent
Hence, we aim to show that $\mu_i$ is finite.
In order to find an upper bound for $\mu_i$, we need to consider two different cases for the radius $i$ of the initial ball $S$.
We need to study the case where $i\geq1$ and where $i=0$ separately.
We first consider the case where $i\geq1$. Then it holds for $i>j$
\begin{align*}
(i+j)^d-(\max(0,i-j))^d
 &=
(i+j)^d -(i-j)^d\\
&=
\sum\limits_{k=0}^d \Big(\begin{array}{c} d \\ k \end{array}\Big)i^{d-k}j^k
-
\sum\limits_{k=0}^d (-1)^k \Big(\begin{array}{c} d \\ k \end{array}\Big)i^{d-k}j^k\\
&=
\sum\limits_{k=1}^d \Big(\begin{array}{c} d \\ k \end{array}\Big)\underbrace{(1-(-1)^k)}_{\in\{0,2\}}\underbrace{i^{d-k}}_{\leq i^{d-1}}\underbrace{j^k}_{\leq j^d} \\
&\leq
 2 \,i^{d-1}j^d \underbrace{\sum\limits_{k=1}^d \Big(\begin{array}{c} d \\ k \end{array}\Big)}_{=2^d}\\
&=
 2^{d+1} \,i^{d-1}j^d 
\end{align*}
and for $i \leq j$ we have
\begin{align*}
(i+j)^d-(\max(0,i-j))^d =(i+j)^d \leq 2^dj^d \leq  2^{d+1} \,i^{d-1}j^d.
\end{align*}
Now we consider the case where $i=0$ and observe that
\begin{align*}
(i+j)^d-(\max(0,i-j))^d =j^d.
\end{align*}
All together, this yields 
\begin{align*}
\mu_{i,j} 
&\leq
C \lambda \, \mathbb{P}(\rho =j)i^{d-1}j^d & \text{for } i\geq 1\\
\mu_{i,j} 
&=
C \lambda \, \mathbb{P}(\rho =j)j^d & \text{for } i =0
\end{align*}
where $C:=2^{d+1} \pi_d$ is a constant. By using this, we can calculate an upper bound for the next generation. For $i \geq 1$ it holds
\begin{align*}
\mu_{i,j}^{(2)}
&=
\sum\limits_{l=0}^\infty \mu_{i,l} \cdot \mu_{l,j}\\
&\leq
\sum\limits_{l=0}^\infty C^2 \lambda^2 i^{d-1}j^d\, \mathbb{P}(\rho =j)l^{2d-1}  \, \mathbb{P}(\rho =l)\\
&=
 C^2 \lambda^2 i^{d-1}j^d\, \mathbb{P}(\rho =j) \sum\limits_{l=0}^\infty l^{2d-1}  \, \mathbb{P}(\rho =l)\\
&=
 C^2 \lambda^2 i^{d-1}j^d\, \mathbb{P}(\rho =j) \, \mathbb{E}[\rho^{2d-1}].
\end{align*}
Inductively we can easily see that for any generation $n \geq 1$ it holds that
\begin{align*}
\mu_{i,j}^{(n)}
&\leq
 C^n \lambda^n i^{d-1}j^d\, \mathbb{P}(\rho =j) \, (\mathbb{E}[\rho^{2d-1}])^{n-1}.
\end{align*}
Very similarly, one can show that for $i=0$, it holds
\begin{align*}
\mu_{i,j}^{(n)}
&\leq
 C^n \lambda^n j^d\, \mathbb{P}(\rho =j) \, (\mathbb{E}[\rho^{2d-1}])^{n-1}.
\end{align*}
Furthermore, note that by definition 
\begin{align*}
\mu_{i,j}^{(0)}
=
\begin{cases}
1 & \text{if } i=j\\
0 & \text{elsewhere.}
\end{cases}
\end{align*}
And hence for $i \geq 1$, the total expected number of members of the entire branching process can be bounded as follows:
\begin{align*}
\mu_i
&=
\sum\limits_{n=0}^\infty \sum\limits_{j=0}^\infty \mu_{i,j}^{(n)}\\
&=
1+\sum\limits_{n=1}^\infty \sum\limits_{j=0}^\infty \mu_{i,j}^{(n)}\\
&\leq
1+\sum\limits_{n=1}^\infty \sum\limits_{j=0}^\infty  C^n \lambda^n i^{d-1}j^d\, \mathbb{P}(\rho =j) \, (\mathbb{E}[\rho^{2d-1}])^{n-1}\\
&=
1+i^{d-1} \sum\limits_{j=0}^\infty j^d\, \mathbb{P}(\rho =j)  \sum\limits_{n=1}^\infty   C^n \lambda^n  \, (\mathbb{E}[\rho^{2d-1}])^{n-1} \\
&=
1+i^{d-1} \, \mathbb{E}[\rho^{d}] \sum\limits_{n=1}^\infty   C^n \lambda^n  \, (\mathbb{E}[\rho^{2d-1}])^{n-1}.
\end{align*}
Very similarly we obtain for $i=0$:
\begin{align*}
\mu_i
&\leq
1+ \, \mathbb{E}[\rho^{d}] \sum\limits_{n=1}^\infty   C^n \lambda^n  \, (\mathbb{E}[\rho^{2d-1}])^{n-1}.
\end{align*}
As $\mathbb{E}[\rho^{2d-1}] < \infty$, there exists
$0<\lambda < (C\mathbb{E}[\rho^{2d-1}] )^{-1}$, which yields
\begin{align*}
C \lambda \mathbb{E}[\rho^{2d-1}] < 1.
\end{align*}
\indent
Hence the expected number of total members of the entire branching process is finite, which yields that the number of balls in any occupied component is finite.\\
\end{proof}
In the following the existence of a supercritical phase in our continuous Boolean model is shown. The proof is outlined in the remark in \cite[page 52]{MeesterRoy1996}.
\begin{lem}\label{sonne}
For a Boolean model $(X,\rho,\lambda)$ on $\mathbb{R}^d$ with $d \geq 2$ there exists $\lambda^* >0$ such that there are unbounded components for all $\lambda\geq\lambda^*$.
\end{lem}
\begin{proof}
As mentioned before, the idea is to map the continuum percolation problem to a discrete lattice percolation problem.
In the discrete case, we know for $p>p_{\text{cr}}$, i.e., if the probability that an edge in a lattice is open is larger than $p_{\text{cr}}$, that there exists an infinite cluster of open edges.
Now, we choose $\epsilon>0$ such that $\mathbb{P}(\rho >\epsilon)>\epsilon$ and $1-\epsilon >p_{\text{cr}}$. In the next step, we choose $\delta>0$ so small that if we partition the space into cubes with side length $\delta$, any two points in neighbouring cubes are at distance at most $2\epsilon$. These cubes describe our discrete lattice, on which we map our continuous percolation problem.
Therefore, we define a cube $c$ as \textit{open} if it contains at least one point of our Poisson point process $X^\lambda$ with a ball of radius at least $\epsilon$. Otherwise it is called \textit{closed}. 
Hence, if there is an infinite cluster of open cubes, then there exists an unbounded component in the Boolean model $(X,\rho,\lambda)$, what we aim to show.
Therefore, we choose $N$ so large that
\begin{align*}
(1-(1-\epsilon)^N) \cdot (1-\epsilon) >p_{\text{cr}}.
\end{align*}
\indent
Let $\lambda^*$ be so large that the probability to have at least $N$ Poisson points in a cube with side length $\delta$ is at least $1-\epsilon$.
Then, the probability that a cube is open is larger than $p_{\text{cr}}$.
This can be shown as follows:
The probability that any cube is open is greater than or equal to the probability that the cube contains at least $N$ points, where at least one point has a ball of radius at least $\epsilon$. Then it follows by our parameter choice for any cube $c$
\begin{align*}
\mathbb{P}(c\text{ is open})
&\geq
\mathbb{P}(\exists \, X_1,...,X_N \in X^\lambda \cap c \text{ such that } \exists \, j \in \{1,...,N\} \colon \rho_j \geq\epsilon)\\
&=
\mathbb{P}(\exists\, j \in \{1,...,N\} \colon \rho_j \geq \epsilon) \cdot \mathbb{P}(\exists \, X_1,...,X_N \in X^\lambda \cap c)\\
&=
(1-\mathbb{P}(\rho_1\leq \epsilon, ...,\rho_N\leq \epsilon ) ) \cdot \mathbb{P}(\exists \, X_1,...,X_N \in X^\lambda \cap c)\\
&=
(1-(\mathbb{P}(\rho\leq \epsilon) )^N ) \cdot \mathbb{P}(\exists \, X_1,...,X_N \in X^\lambda \cap c)\\
&>
(1-(1-\epsilon)^N) \cdot (1-\epsilon)\\
&>
p_{\text{cr}}.
\end{align*} 
Note that
\begin{align*}
\mathbb{P}(\rho\leq \epsilon) 
=
1-\mathbb{P}(\rho>\epsilon) 
<
1-\epsilon.
\end{align*}
Furthermore note that distinct cubes are independently open or closed.  As for any $p > p_{\text{cr}}$, there exists an infinite cluster and we have shown that $\mathbb{P}(c\text{ is open})$ is greater than $p_{\text{cr}}$, we know that there exists an infinite cluster of open cubes for any density $\lambda \geq \lambda^*$. As explained before, this implies the existence of an unbounded component in the Boolean model $(X, \rho, \lambda)$.
\\
\end{proof}

\section{SINR percolation for Cox point processes}
After we have studied the Boolean model in the previous chapter, we now introduce another, more advanced model in the area of telecommunication: The SINR model.
The term \textit{SINR} means \textit{Signal to Interference and Noise Ratio} and indicates the crucial difference to the Boolean model: the message transmission between two points is interfered by the surrounding points and a fixed noise. \medskip
\\ \indent
Besides the above described adaption of the message transmission between points, we furthermore consider in this chapter a more sophisticated way to model the location of the points.
As in the previous chapter, we consider a random point process, where the different points represent the devices (or users), which like to send and receive messages. But instead of a Poisson point process (PPP), we consider a so-called \textit{Cox point process (CPP)}. 
Remember that within a PPP, we do not have any information about the environment of the users. 
Loosely speaking, a CPP incorporates a random environment, which makes the model more realistic.
Before we start to introduce the SINR model formally, we give some background information about CPPs in the next section.
\medskip
\\ \indent
If not stated differently, the content of this chapter can be found in the second part of the thesis 'Message routeing and percolation in interference limited multihop networks' from 2019 by Tóbiás, which is an extended version of his previous paper 'Signal to interference ratio percolation for Cox point processes' from 2018. 
\subsection{The Cox point process}
In this section we define what a CPP is and give some examples, which motivate why it is interesting to consider a CPP instead of a PPP even if it is more complex.
The information and figures in this section can be found in \cite[page 15ff.]{JahnelKoenig2018}.\medskip
\\ \indent
In order to give a formal definition of a CPP, we first need to explain what a random measure is.
\begin{defi}
Let $D \subset \mathbb{R}^d$ be a measurable set.
A random element $\Lambda$ of the space of all $\sigma$-finite measures $\mathcal{M}(D)$ on $D$, equipped with the smallest $\sigma$-algebra, such that all evaluation mappings
\begin{equation*}
\mathcal{M} (D) \rightarrow [0, \infty) \colon \mu \mapsto \mu(B)
\end{equation*}
are measurable for all measurable sets $B \subset D$, is called a \textit{random measure on $D$}.
\end{defi}
\begin{defi}
Let $\Lambda$ be a random measure on $D$. Then, the PPP $\mathbb{X}$ with random intensity measure $\Lambda$ is called a \textit{Cox point process directed by $\Lambda$}.
\end{defi} 
In other words, given a random measure $\Lambda$ on $D$, a CPP is simply a PPP which is conditioned on $\Lambda$. \medskip
\\ \indent
Let us now introduce two important properties of a CPP. First, remember from Lemma \ref{laplace_unique} that the distribution of a point process $\mathbb{X}$ is characterized by the Laplace transform. In the case of a CPP it is given by
\begin{equation*}
\mathcal{L}_{\mathbb{X}} (f) =
\mathbb{E} [ \mathrm{e}^{- S_f ( \mathbb{X})} ] =
\mathbb{E} \Big[ \exp \Big( \int_D \big( \mathrm{e}^{-f(x)} - 1 \big) \Lambda( \text{d}x) \Big) \Big],
\end{equation*}
where $f$ is a measurable function from $D$ to $[0, \infty)$.
Another interesting property is the expected number of points in a measurable volume $A \subset D$. In the case of a CPP, this is simply the expected intensity of $A$, which means
\begin{equation*}
\mathbb{E} [ N_{\mathbb{X}} (A)]
=
\mathbb{E} [ \mathbb{E} [ N_{\mathbb{X}} (A) | \Lambda]]
=
\mathbb{E} [ \Lambda (A) ].
\end{equation*}
\indent
Using CPPs enables us to model more realistic spatial telecommunication systems, where the users are located in a random environment.
For example it is possible to model an urban area environment, such as random street systems, through the random measure $\Lambda$. There exist several possibilities to do this, the following two examples introduce the main approaches.
\begin{exa}[Poisson-Voronoi tessellation and Poisson-Delaunay tessellation]\label{pvt}
In order to model a random street system, let us consider the case where $d=2$. Then, we define the random environment $\Lambda$ as the restriction of the Lebesgue measure to a random segment process $S$ in $\mathbb{R}^2$.
A natural choice for the measure on $S$ is the one-dimensional Hausdorff measure, which we denote by $\nu_1$. It attaches a finite and positive value to each bounded non-trivial line segment, more precisely it attaches its length.
By putting $\Lambda(\text{d}x) = \nu_1(S \cap \text{d}x) $, we get a random measure on $\mathbb{R}^2$ that is concentrated on $S$.\medskip
\\ \indent
There are several possible choices for $S$. As we would like to model random street systems, we focus on tessellations in this example. In the following we introduce a few interesting tessellations.
We start with the \textit{Poisson-Voronoi tessellation (PVT)}. Let us therefore consider a PPP $\mathbb{X} = (X_i)_{i \in I}$ in the communication area $D$. Now, we attach to each point $X_i$ a so-called \textit{Voronoi cell} $z(X_i)$, which is defined as follows:
\begin{equation*}
z(X_i) = \big\{ x \in D \colon || x-X_i || \leq \inf\limits_{j \in I} || x- X_j || \big\}.
\end{equation*}
This means, the cell $z(X_i)$ contains all points in $D$ that are closer to $X_i$ than any other point in $\mathbb{X}$.
When we draw the boundaries of all the cells which arise from our point process $\mathbb{X}$, we obtain a pattern, which looks like tessellations and hence explains the naming of $S$. In Figure \ref{fig:tessellation} an illustration of a PVT is given.
\medskip
\\ \indent
Another example of a tessellation is the \textit{Poisson-Delaunay tessellation (PDT)}, which is defined via a given PVT as follows: 
Any two points $X_i$ and $X_j$ of $\mathbb{X}$ are connected by a line if and only if there exists exactly one crossing of this line and one cell boundary of the given PVT.
The following figure is an illustration of a PVT and its corresponding PDT.
\begin{figure}[H]
  \centering
    \includegraphics[width=0.9\textwidth]{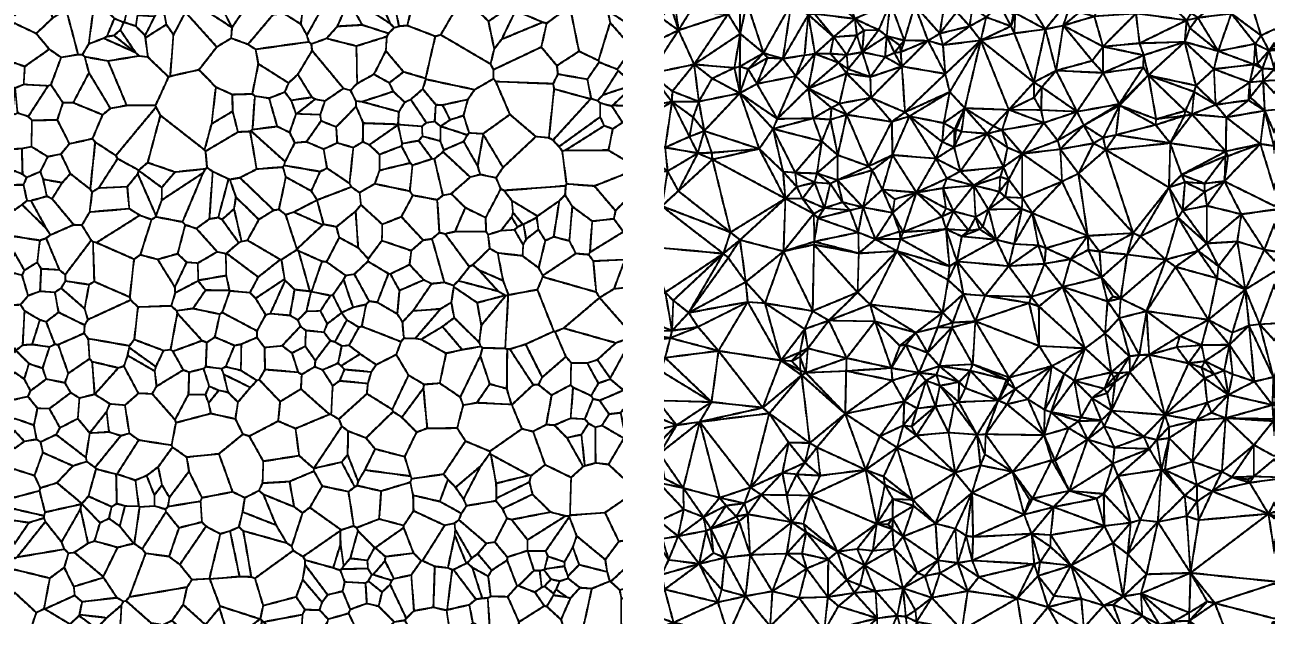}
  \caption{Realization of a PVT (left) and its corresponding PDT (right)}
  \label{fig:tessellation}
\end{figure}
\indent
\end{exa}
\begin{exa}[Manhatten grids]
Another example of a tessellation process are \textit{Manhattan grids (MG)}.
Let us therefore consider a stationary renewal process. By drawing perpendicular lines through the points of this process, the grid arises. As the name indicates, it represents for example the street system in Manhattan, New York. 
The reason why we consider a stationary renewal process as the underlying point process and not a simple PPP as before, is that the latter one only allows us to model a street system, where the distance between two consecutive streets is exponentially distributed. This would not be a realistic setting for a street system.\medskip
\\ \indent
It is possible to improve the MG further by putting additional rectangular lines inside the boxes given by the MG. 
The arising graph is called \textit{nested Manhattan grid (NMG)}.
See the following figures for illustrations of the above described processes.
\begin{figure}[H]
  \centering
    \includegraphics[width=0.9\textwidth]{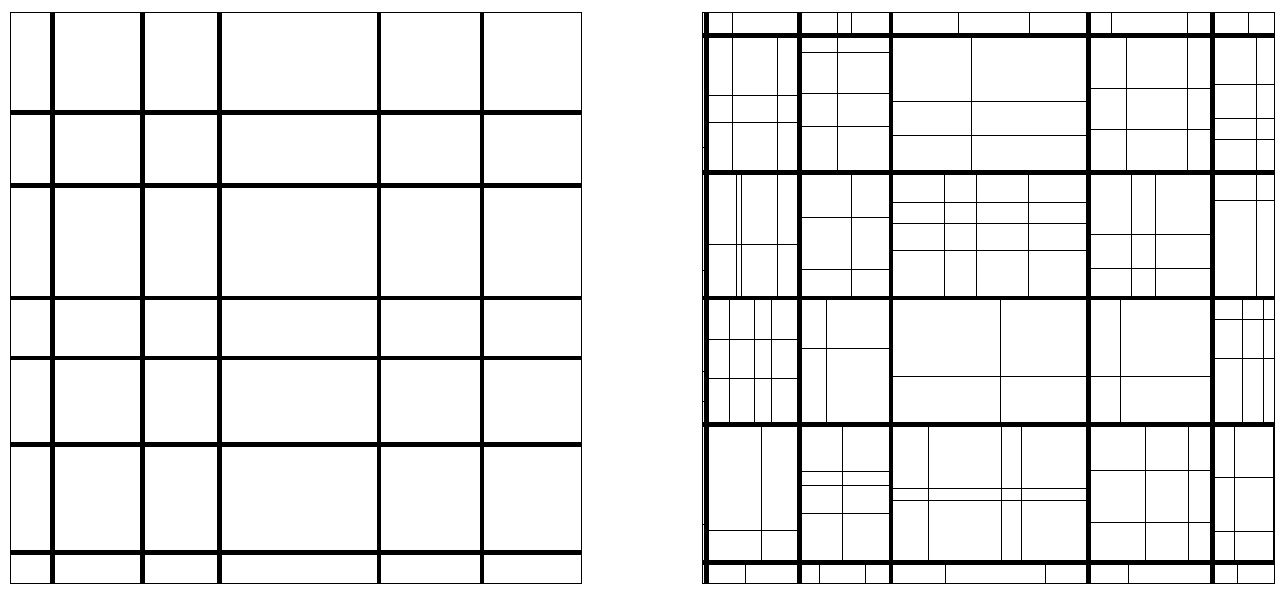}
  \caption{Realization of a MG (left) and a corresponding NMG (right)}
\end{figure}
\end{exa}
\subsection{Model definition and existence of a subcritical phase}
Let us first introduce the SINR model.
The information of this and the following section can be found in \cite[Chapter 4]{Tobias2019}.
Throughout the whole section, let $X^\lambda=(X_i)_{i \in I}$ be a CPP on $\mathbb{R}^d$ with a stationary intensity measure $\lambda \cdot \Lambda$, where $\lambda>0$ and $\mathbb{E}[\Lambda([0,1]^d)]=1$. Hence, $X^\lambda$ is a PPP with intensity measure $\lambda \cdot \Lambda$ conditioned on $\Lambda$. The random element $\Lambda$ in the space $\mathbb{M}$ of Borel measures on $\mathbb{R}^d$ describes the random environment. As it is stationary, $\Lambda(\cdot)$ equals $\Lambda(\cdot +x)$ in distribution for all $x\in \mathbb{R}^d$. \medskip
\\ \indent 
Remember that the points of the point process $X^\lambda$ represent the users in a telecommunication system.
In the following we describe how the message transmission between the users works.
Consider therefore two points $X_i$ and $X_j$ of $X^\lambda$, where $X_i$ wants to communicate with $X_j$. 
Similar to the Boolean model, which we introduced in the previous chapter, a successfull message transmission depends on the distance between the two points, which want to communicate. It makes sense to assume that the larger this distance is, the more difficult is the communication. In our model, this feature is expressed via a so-called \textit{path-loss function} $\ell$. There are a few important conditions for this function, which are stated later-on. Moreover, we assume that every point has a fixed signal power $P>0$. 
So far, we have explained that a successfull message transmission between the two points $X_i$ and $X_j$ depends on the expression
$P \cdot \ell(|X_i-X_j|)$. Now, we make another assumption, which determines the main difference between the Boolean model and the SINR model: We incorporate the interference from all other points that also transmit a signal that reaches $X_j$, which is 
\begin{equation*}
 \sum\limits_{k \neq i,j} P \cdot \ell(|X_k-X_j|)
\end{equation*}
 and some constant environment noise $N_0>0$. Moreover, we multiply the interference from the other points with a factor $\gamma>0$, which is a constant that describes the technology which is used within the system. The smaller it is, the less do the surrounding nodes interfere with the message transmission. 
All together, a successfull message transmission from point $X_i$ to $X_j$ depends on the value of the SINR, which is defined as:
\begin{equation*}\label{SINR}
\text{SINR}(X_i, X_j, X^\lambda, P, N_0, \gamma):=
\frac{P \cdot \ell(|X_i-X_j|)}{N_0+\gamma \sum\limits_{k \neq i,j} P \cdot \ell(|X_k-X_j|)}.
\end{equation*}
Remember that 'SINR' is the abbreviation for 'Signal  to Interference and Noise Ratio'.
In particular, we assume that within our telecommunication system there exists a technical factor $\tau$, which determines how sensitive the system is. The point $X_i$ can send messages successfully to the point $X_j$ if its SINR exceeds $\tau$. \medskip
\\ \indent
We still need to specify the properties of the path-loss function in a mathematical way.
The path-loss function $\ell \colon [0,\infty) \rightarrow  [0,\infty)$ satisfies the following conditions:
\begin{enumerate}
\item[a)] $\ell$ is continuous, constant on $[0,d_0]$ for some $d_0 \geq 0$ and strictly decreasing on $[d_0,\infty) \cap$ supp $\ell$,
\item[b)] $1 \geq \ell(0) > \tau N_0/P$,
\item[c)] $\int\limits_{\mathbb{R}^d}\ell (|x|)dx < \infty$.
\end{enumerate}
The last condition is needed in order to ensure that the denominator within the SINR stays finite. Otherwise, a successfull message transmission would never be possible. \medskip
\\ \indent
In our model, we do not want to analyse the case where a point $X_i$ can send a message to $X_j$, but $X_j$ can not send a message to $X_i$. Instead, we are only interested to study the situation where both points can send messages to each other.
Hence, we draw an edge between two points $X_i$ and $X_j$ if the SINR exceeds $\tau$ in both cases, which means it holds
$\text{SINR}(X_i, X_j, X^\lambda, P, N_0, \gamma) > \tau$ and 
$\text{SINR}(X_j, X_i, X^\lambda, P, N_0, \gamma) > \tau$.
By doing this we get the so called \textit{SINR graph}, which is denoted by $g_{(\gamma, N_0, \tau,P)}(X^\lambda)$. Note that for $\gamma =0$, the SINR graph equals the Gilbert graph $g_{r_\text{B}}(X^\lambda)$ with radius
\begin{equation}\label{radius}
r_\text{B} := \ell^{-1}\Big(\frac{\tau N_0}{P}\Big),
\end{equation}
which is strictly positive as $\ell$ is strictly decreasing on $[d_0, \infty)$ and hence 
$r_\text{B} >0$ is equivalent to 
$\tau N_0 / P < \ell(0)$, which holds true by our second condition on the path-loss function.
Note that we have $r_\text{B} > d_0$  by the first and second condition on $\ell$. \medskip
\\ \indent
After we have introduced the SINR model, we aim to study the existence of a phase transition. This means, we would like to analyse under which conditions there exists a critical density such that for larger densities there exists an infinite cluster almost surely and for smaller densities there exists no infinite cluster almost surely. Let us therefor fix the parameters $N_0$, $\tau$ and $P$.
The critical density is then defined by
\begin{align*}
\lambda_{N_0,\tau,P}:=
\inf\{\lambda>0 \colon \gamma^*(\lambda')>0 \text{ for all } \lambda' \geq \lambda\},
\end{align*}
where
\begin{align*}
\gamma^*(\lambda)=\gamma^*(\lambda,N_0,\tau,P):=
\sup\{\gamma >0 \colon \mathbb{P}(g_{(\gamma,N_0,\tau,P)}(X^\lambda) \text{ percolates})>0\}.
\end{align*}
Note that in most cases, it will hold that 
$\lambda_{N_0,\tau,P}=\inf\{\lambda>0 \colon \gamma^*(\lambda)>0\}$. \medskip
\\ \indent
In order to proof the phase transition, we need more conditions on the the random measure $\Lambda$.
Thereby, we need to understand how strong and how far reaching spatial stochastic dependencies are within our model. The concept of stability helps us to do this.
In the following we first introduce the definition of a stabilizing random measure, which is a sufficient criterium for the existence of a subcritical phase. 
Thereby, we denote the cube with side length $n>0$ and center $x\in \mathbb{R}^d$ by 
\begin{equation*}
\text{Q}_n(x):=x+\Big[-\frac{n}{2},\frac{n}{2}\Big]^d.
\end{equation*}
For the cube at the origin $o$ we simply write $ \text{Q}_n:=\text{Q}_n(o)$.
\newpage
\begin{defi}
The random measure $\Lambda$ is \textit{stabilizing}, if there exists a random field of stabilization radii $R=\{R_x\}_{x \in \mathbb{R}^d}$ defined on the same probability space as $\Lambda$ such that
\begin{enumerate}
\item 
$(\Lambda,R)$ are jointly stationary,
\item
$\lim\limits_{n \rightarrow \infty} \mathbb{P}(\sup\limits_{y \in \text{Q}_n\cap\, \mathbb{Q}^d}R_y <n)
=
\lim\limits_{n \rightarrow \infty} \mathbb{P}(R(\text{Q}_n)<n)=1$ and
\item 
for all $ n \geq 1$ and for any measurable function $f \colon \mathbb{M}\rightarrow [0,\infty)$ and finite $\phi \subset \mathbb{R}^d$ with $\text{dist}(x,\phi\setminus\{x\})>3n$ for all $x \in \phi$ the following random variables are independent
\begin{equation*}
f(\Lambda_{\text{Q}_n(x)}) \cdot \mathds{1}\{R( \text{Q}_n(x))<n\}, \,\, x \in \phi.
\end{equation*}
\end{enumerate}
Note that we used the following notation
\begin{equation*}
R(\text{Q}_n(x)):=\sup\limits_{y \in \text{Q}_n (x)\cap\, \mathbb{Q}^d}R_y.
\end{equation*}
\indent
A strong form of stabilization is called \textit{b-dependence}: For $b >0$, the random measure $\Lambda$ is \textit{b-dependent}, if $\Lambda_A$ and $\Lambda_B$ are independent whenever $\text{dist}(A,B) >b$. Thereby, $\Lambda_A$ indicates the restriction of the random measure $\Lambda$ to the set $A \subset \mathbb{R}^d$, respectively $\Lambda_B$ to the set $B \subset \mathbb{R}^d$.
Moreover, 
$\text{dist}(\phi, \psi) = \inf \{ ||x-y||_2 \colon x \in \phi, y \in \psi\}
$
denotes the $\ell^{2}$-distance between two sets $\phi, \psi \subset \mathbb{R}^d$.
\end{defi}
In order to be able to prove the existence of a supercritical phase, it is necessary, but not sufficient that $\Lambda$ is stabilizing. 
The reason is the possible scenario, where the environment puts no positive weight on a connected area and hence, there is no chance to obtain an infinite cluster even for a very large density $\lambda$.
Thus, we need a stronger stability assumption on $\Lambda$, which ensures enough connectivity. 
Therefore, we introduce the definition of an asymptotically essentially connected random measure in the following.
Note thereby that the \textit{support} of a measure $\mu$ is given by
$
\text{supp}(\mu):=\{x \in \mathbb{R}^d \colon \mu(\text{Q}_\epsilon(x)) >0 \,\,\,\forall \epsilon >0\}.
$
\begin{defi}
The stabilizing random measure $\Lambda$ with stabilizing radii $R$ is \textit{asymptotically essentially connected} if for all $n \geq 1$, whenever $R(\text{Q}_{2n}) <\frac{n}{2}$, it holds that
\begin{enumerate}
\item
$\text{supp}(\Lambda_{\text{Q}_n}) \neq \emptyset$ and
\item
there exists a connected component $\mathcal{C}$ such that
$
\text{supp}(\Lambda_{\text{Q}_n}) \subset \mathcal{C} \subset \text{supp}(\Lambda_{\text{Q}_{2n}}).
$
\end{enumerate}
\end{defi}
\begin{exa}
The stationary PVT on $\mathbb{R}^d$, which we defined in Example \ref{pvt} in the previous section, is asymptotically essentially connected (and hence as well stabilizing). The proof relies on the definition of the corresponding stabilization radius as
$R_x = \inf \{ || X_i - x || \colon X_i \in X^\lambda \}$.
See \cite[page 9]{HirschJahnelCali2018} for detailed information.
\end{exa}
Before we start to consider the phase transition in the SINR model, we state some results in the Boolean model.
Let us further on consider a CPP $X^\lambda$. Remember that in the Boolean model, two points $X_i$ and $X_j$ are connected by an edge if their distance is less than a fixed connection threshold $r>0$. The arising graph is denoted by $g_r(X^\lambda)$. 
In this model, the so called \textit{critical intensity} is given by 
\begin{equation*}
\lambda_{\text{c}}(r) = \inf\{\lambda \colon \mathbb{P}(g_r(X^\lambda) \text{ percolates})>0\}.
\end{equation*}
\indent
The following theorem is proven by Hirsch, Jahnel and Cali in \cite[page 3f. and page 11f.]{HirschJahnelCali2018} and used by Tóbiás in order to prove the existence of a subcritical phase in his SINR model.
\begin{thm}\label{hase}
Let $r>0$.
\begin{enumerate}
\item[a)] If $\Lambda$ is stabilizing, then $\lambda_{\mathrm{c}}(r)>0$.
\item[b)] If $\Lambda$ is asymptotically essentially connected, then $\lambda_{\mathrm{c}}(r)<\infty$.
\end{enumerate}
\end{thm}
The proof of both parts of the above theorem is based on a renormalization argument. 
As we make the condition that $\Lambda$ is stabilizing, it is possible to define an auxiliary percolation process, which can be studied applying the techniques from \cite{Liggett1997}, more precisely, the authors use Theorem 0.0.
In order to prove the existence of a subcritical phase, we only need to prove that there exist large areas without any points. This is easier than the proof of the existence of a supercritical phase, where we need to show that it is possible to create suitable connected components. Thereby the stabilization condition is not sufficient. Consider for example the stabilizing measure $\Lambda \equiv 0$, for which holds that $\lambda_{\text{c}}(r) = \infty$ for every $r>0$. This shows that we need in addition to stability as well a strong local connectivity of the intensity measure. 
Hence, a supercritical phase only exists under the condition that $\Lambda$ is asymptotically essentially connected.
\medskip
\\ \indent
As already mentioned before, the following lemma states the existence of a subcritical phase in the SINR model. As its proof is based on the above Theorem \ref{hase} it becomes very short. Note that the proof of the existence of a supercritical phase is much more complex.
\begin{lem}
Let $N_0,\tau, P>0$.
It holds that
$\lambda_{N_0,\tau,P}\geq \lambda_{\mathrm{c}}(r_{\mathrm{B}})$. Furthermore, if $\Lambda$ is stabilizing, then $\lambda_{N_0,\tau,P}>0$.
\end{lem}
\begin{proof}
For $N_0, \tau, P, \gamma>0$ it is clear that 
$g_{(\gamma, N_0, \tau,P)}(X^\lambda) \preceq g_{(0, N_0, \tau,P)}(X^\lambda)$, i.e., all connections in the left graph
$g_{(\gamma, N_0, \tau,P)}(X^\lambda)$ exist as well in the right graph
$g_{(0, N_0, \tau,P)}(X^\lambda)$. That means that if the right graph does not percolate, the left graph does not percolate either. As shown before, the right graph equals a Gilbert graph with critical density $\lambda_{\text{c}}(r_{\text{B}})$ and hence
$\lambda_{N_0,\tau,P}\geq \lambda_{\text{c}}(r_{\text{B}})$.
By Theorem \ref{hase} it follows that for a Gilbert graph with positive radius $r$, it holds that, if $\Lambda$ is stabilizing, then $\lambda_{\text{c}}(r)>0$. Hence, we can conclude $\lambda_{N_0,\tau,P}>0$.\\
\end{proof}
\newpage
\subsection{Existence of a supercritical phase}
In this section, we finally give the main theorem of this chapter that states under which conditions there exists a supercritical phase in the SINR model. 
\begin{thm}\label{thm_phase_transition}
Let $N_0,\tau, P>0$.
If $\Lambda$ is asymptotically essentially connected, then $\lambda_{N_0,\tau,P}<\infty$ holds if at least one of the following conditions is satisfied:
\begin{enumerate}
\item[(a)]
$\ell$ has compact support,
\item[(b)]
$\Lambda(\text{Q}_1)$ is almost surely bounded,
\item[(c)]
$\mathbb{E}[\exp(\alpha\Lambda(\text{Q}_1))]<\infty$ for some $\alpha >0$ and $\int\limits_x^\infty r^{d-1}\ell(r) dr = \mathcal{O}\big(\frac{1}{x}\big)$ as $x \rightarrow \infty$.
\end{enumerate}
\end{thm}
The proof of the above theorem comes naturally in three steps.
First, we map the continuous percolation problem on a discrete lattice. Secondly, we show that there exists a supercritical phase in the lattice. In the last step, we show that percolation in the lattice implies percolation in our original continuous model. We start by fixing the parameters $N_0, \tau, P >0$ and simplifying the notation for $\gamma\geq 0$ and $\lambda >0$ by
\begin{equation*}
g_{(\gamma)}(X^\lambda)= g_{(\gamma, N_0, \tau,P)}(X^\lambda).
\end{equation*}
\ 
\subsubsection{Step 1: Mapping on a lattice}
In order to map the continuous percolation problem on a discrete lattice, we need to give some definitions.
\begin{defi}\label{def_shot}
For $a \geq0$, we define a shifted version of the path-loss function $\ell$ by
\begin{equation*}
\ell_a \colon [0,\infty) \rightarrow [0,\infty), \,\, r \mapsto \ell(0) \cdot \mathds{1}\Big\{r < \frac{a \sqrt{d}}{2}\Big\} + \ell\Big(r-\frac{a \sqrt{d}}{2}\Big) \cdot \mathds{1}\Big\{r \geq \frac{a \sqrt{d}}{2}\Big\}
\end{equation*}
and the \textit{shot noise processes} are given by
\begin{equation*}
I(z)= \sum\limits_{X_i \in X^\lambda}\ell(|z-X_i|)
\text{ and } 
I_a(z)= \sum\limits_{X_i \in X^\lambda}\ell_a(|z-X_i|) \text{, } z \in \mathbb{R}^d.
\end{equation*}
\end{defi}
\medskip
Let us now choose some $r \in (d_0,r_{\text{B}})$. This is possible as assumptions (a)-(b) on $\ell$ imply that
$\ell(d_0) = \ell(0) >\tau N_0/P$ and hence
$d_0 < \ell^{-1}(\tau N_0/P)$. Together with equation (\ref{radius}) it follows that $d_0 <r_{\text{B}}$.
We continue with this given fixed parameter $r$ and introduce another important definition.
\newpage
\begin{defi}
For $n \in \mathbb{N}$ we define a site $z \in \mathbb{Z}^d$ as \textit{n-good} if
\begin{enumerate}
\item
$R(\text{Q}_n(nz)) < n/2$ and
\item
$X^\lambda \cap \text{Q}_n(nz)\neq \emptyset$ and
\item
every $X_i$, $X_j \in X^\lambda \cap \text{Q}_{3n}(nz)$ are connected by a path in $g_r(X^\lambda)\cap \text{Q}_{6n}(nz)$.
\end{enumerate}
Otherwise, it is called \textit{n-bad}. 
\end{defi}
Remember that $\text{Q}_n(nz)$ is the cube with side length $n$, which is centered at $nz$. Moreover,
$R(\text{Q}_n(nz))$ is the supremum over all stabilization radii of the points in $\text{Q}_n(nz) \cap \, \mathbb{Q}^d$, which exist as $\Lambda$ is stabilizing by assumption.
So, loosely speaking, a site $z$ of our lattice $\mathbb{Z}^d$ is $n$-good, if the stabilization radii near the point $nz$ are not too large, there exists a Cox point near the point $nz$ and any two Cox points near the point $nz$ are connected by a path of at most $r$-distant Cox points.
\medskip
\\ \indent
Using the above definitions, we are now able to map the continuous percolation problem on a discrete lattice.
Thereby, we define for $z \in \mathbb{Z}^d$, $n \in \mathbb{N}$ and $M>0$: 
\medskip
\\ \indent
$A_n(z)=\mathds{1}\{z \text{ is $n$-good}\}$, $B_{n,M}(z)=\mathds{1}\{I_{6n}(nz) \leq M\}$ and $C_{n,M}(z)=A_n(z) \cdot B_{n,M}(z)$. \medskip
\\ 
Then, a site $z$ of the lattice $\mathbb{Z}^d$ is \textit{open} if $C_{n,M}(z)$ equals $1$, otherwise, it is \textit{closed}.
So we separate the percolation problem in two parts. The first random variable $A_n(\cdot)$ takes care of the connectedness and the second random variable $B_{n,M}(\cdot )$ treats the interference. The random variable $C_{n,M}(\cdot)$ brings both aspects together.
\\ \
\subsubsection{Step 2: Percolation in the lattice}
In the second step we show that there exists percolation in the discrete lattice. 
The proof is structured in three parts accordingly to the above described mapping on the lattice. 
First, we show Lemma \ref{lem_A}, 
which covers the connectedness problem and states that for sufficiently large $n \in \mathbb{N}$ and $\lambda >0$ the probability that any pairwise distinct sites in the lattice are $n$-bad can be bounded arbitrarily small.
Afterwards, Proposition \ref{prop_B} states that for sufficiently large $n \in \mathbb{N}$ and $M >0$ the probability that the interference at any pairwise distinct sites in the lattice is greater than $M$ can be bounded arbitrarily small.
As the proof of the latter one is very comprehensive, it is shown separatly in Section \ref{sec_B}. The results of the lemma and the proposition are combined in the third part, which is Proposition \ref{prop_C}. The percolation of the lattice follows immediatly from the latter proposition using the Peierls argument.
\begin{lem}\label{lem_A}
For all sufficiently large $n \in \mathbb{N}$ and $\lambda >0$, there exists a constant $q_\mathrm{A} <1$ such that for any $N \in\mathbb{N}$ and pairwise distinct sites $z_1,...,z_N \in \mathbb{Z}^d$ it holds
\begin{equation*}
\mathbb{P}(A_n(z_1)=0,...,A_n(z_N)=0)\leq q_{\mathrm{A}}^N.
\end{equation*}
Moreover, for any $\epsilon >0$ and for sufficiently large $\lambda$, one can choose $n$ so large that $q_{\mathrm{A}} \leq \epsilon$.
\end{lem}
\begin{proof}
From \cite[page 11]{HirschJahnelCali2018} follows for asymptotically essentially connected $\Lambda$, that any process of $n$-good sites is $7$-dependent and that percolation of $n$-good sites implies percolation of the Boolean model $g_r(X^\lambda)$.
Furthermore, in \cite[page 12]{HirschJahnelCali2018} it is shown that for all $z \in \mathbb{Z}^d$: 
\begin{equation}\label{lim_equ}
\lim\limits_{n \rightarrow \infty}\lim\limits_{\lambda \rightarrow \infty}\mathbb{P}(A_n(z)=0)=0,
\end{equation}
where the convergence is uniform in $z \in \mathbb{Z}^d$. \medskip
\\ \indent
Now, let $N \in \mathbb{N}$ and $z_1,...,z_N \in \mathbb{Z}^d$. Then it follows by 7-dependence: There exists $m \geq 1$ and a subset $\{k_j\}_{j=1}^m$ of $[N]=\{1,....,N\}$ such that $A_n(z_{k_1}),...,A_n(z_{k_m})$ are independent and $m \geq N/8^d$.
Loosely speaking, this means that if we consider a large enough number of points, we will find independent ones. Hence, we obtain
\begin{align*}
\mathbb{P}(A_n(z_1)=0,...,A_n(z_N)=0)
&\leq
\mathbb{P}(A_n(z_{k_1})=0,...,A_n(z_{k_m})=0)\\
&=
\mathbb{P}(A_n(o)=0)^m\\
&\leq
\mathbb{P}(A_n(o)=0)^{\frac{N}{8^d}}.
\end{align*}
Note that the last  inequality clearly holds as $\mathbb{P}( \cdot) \in [0,1]$.
Then, we define 
\begin{equation*}
q_{\text{A}}:=\limsup\limits_{n \rightarrow \infty}\mathbb{P}(A_n(o)=0)^{\frac{1}{8^d}}
\end{equation*} 
and by equation (\ref{lim_equ}), the parameter $q_{\text{A}}$ tends to $0$ as the density $\lambda$ tends to infinity.
All together, it follows
\begin{equation*}
\mathbb{P}(A_n(z_1)=0,...,A_n(z_N)=0)
\leq
\mathbb{P}(A_n(o)=0)^{\frac{N}{8^d}}
\leq 
q_{\text{A}}^N.
\end{equation*}
Note that by definition it holds
\begin{equation*}
q_{\text{A}} = 
\limsup\limits_{n \rightarrow \infty} \, \mathbb{P}(\text{ The origin is } n\text{-bad.})^{\frac{1}{8^d}}
,
\end{equation*}
 where the probability that the origin is $n$-bad means, that either 
$R(\text{Q}_n(o)) \geq n$ or the intersection $X^\lambda \cap \text{Q}_n(o)$ is empty
or every $X_i, X_j \in X^\lambda \cap \, \text{Q}_{3n}(o)$ are connected by a path in $g_r(X^\lambda)\cap \text{Q}_{6n}(o)$.
Hence, the parameter $q_{\text{A}}$ clearly can be made arbitrarily small by choosing a larg enough parameter $n$.\\
\end{proof}
\begin{prop}\label{prop_B}
Under assumption (a), (b) or (c) in Theorem \ref{thm_phase_transition} the following holds:\\
For all $\lambda >0$ and for all sufficiently large $n \in \mathbb{N}$ and $M >0$, there exists a constant $q_{\mathrm{B}} <1$ such that for any $N \in\mathbb{N}$ and pairwise distinct sites $z_1,...,z_N \in \mathbb{Z}^d$ it holds
\begin{equation*}
\mathbb{P}(B_{n,M}(z_1)=0,...,B_{n,M}(z_N)=0)\leq q_{\mathrm{B}}^N.
\end{equation*}
Moreover, for any $\epsilon >0$, $\lambda >0$ and large enough $n \in \mathbb{N}$, one can choose $M$ so large that $q_{\mathrm{B}} \leq \epsilon$.
\end{prop}
The proof of the above proposition is very extensive and hence it is given in the separate Section \ref{sec_B} in the end of this chapter.
\begin{prop}\label{prop_C}
Under assumption (a), (b) or (c) in Theorem \ref{thm_phase_transition} the following holds:\\
For all sufficiently large $\lambda >0$, $n \in \mathbb{N}$ and $M >0$, there exists a constant $q_C <1$ such that for all $N \in\mathbb{N}$ and pairwise distinct sites $z_1,...,z_N \in \mathbb{Z}^d$ it holds
\begin{equation*}
\mathbb{P}(C_{n,M}(z_1)=0,...,C_{n,M}(z_N)=0)\leq q_C^N.
\end{equation*}
Moreover, for any $\epsilon >0$, there exist $\lambda$, $n$, $M$ large enough so that $q_C \leq \epsilon$.
\end{prop}
\begin{proof}
Let $N \in  \mathbb{N}$ and $z_1,...,z_N\in \mathbb{Z}^d$ be pairwise distinct. We show the proof analogously as in \cite[page 559f.]{Dousse2006}, which is a bit different to the approach in \cite[page 86]{Tobias2019} but leads to the same result.
\medskip
\\ \indent
In order to simplify the notation, we write
$\tilde{A}_n(z_i):=1-A_n(z_i)$ and $\tilde{B}_{n,M}(z_i):=1-B_{n,M}(z_i)$ for all $i$. As $A_n(\cdot)$ and $B_{n,M}(\cdot)$ only take the value $0$ or $1$, we can easily see that the following holds for all $i$
\begin{equation*}
1-C_{n,M}(z_i)
= 1- A_n(z_i) B_{n,M}(z_i)
\leq
(1-A_n(z_i))+(1-B_{n,M}(z_i))
=
\tilde{A}_n(z_i)+\tilde{B}_{n,M}(z_i). 
\end{equation*}
\indent
Let $(k_i)_{i=1}^N$ be a binary sequence, i.e., $k_i \in \{0,1\}$ and denote by $K$ the set of the $2^N$ such sequences.
Then it follows
\begin{align*}
\mathbb{P}(C_{n,M}(z_1)=0,...,C_{n,M}(z_N)=0)
&=
\mathbb{P}(1-C_{n,M}(z_1)=1,...,1-C_{n,M}(z_N)=1)\\
&=
\mathbb{P}\Big(\prod\limits_{i=1}^N (1-C_{n,M}(z_i))=1\Big)\\
&=
\mathbb{E}\Big(\prod\limits_{i=1}^N (1-C_{n,M}(z_i))\Big)\\
&\leq
\mathbb{E}\Big(\prod\limits_{i=1}^N (\tilde{A}_n(z_i)+\tilde{B}_{n,M}(z_i))\Big)\\
&=
\sum\limits_{(k_i)_{i=1}^N\in K} \mathbb{E}\Big( \prod\limits_{\{i \colon k_i=0\}}\tilde{A}_n(z_i)\prod\limits_{\{i \colon k_i=1\}}\tilde{B}_{n,M}(z_i) \Big)\\
&\leq
\sum\limits_{(k_i)_{i=1}^N\in K}\sqrt{ \mathbb{E}\Big( \prod\limits_{\{i \colon k_i=0\}}\tilde{A}^2_n(z_i)\prod\limits_{\{i \colon k_i=1\}}\tilde{B}^2_{n,M}(z_i) \Big)}\\
&=
\sum\limits_{(k_i)_{i=1}^N\in K}\sqrt{ \mathbb{E}\Big( \prod\limits_{\{i \colon k_i=0\}}\tilde{A}_n(z_i)\prod\limits_{\{i \colon k_i=1\}}\tilde{B}_{n,M}(z_i) \Big)}\\
&=
\sum\limits_{(k_i)_{i=1}^N\in K}\sqrt{ \mathbb{E}\Big( \prod\limits_{\{i \colon k_i=0\}}\tilde{A}_n(z_i)\Big) \mathbb{E}\Big(\prod\limits_{\{i \colon k_i=1\}}\tilde{B}_{n,M}(z_i) \Big)}
.
\end{align*}
\indent
Thereby we used the Schwarz's inequality, the fact that $\tilde{A}_n(\cdot),\tilde{B}_{n,M}(\cdot) \in \{0,1\}$ and that the random variables $A_n(\cdot)$ and $B_{n,M}(\cdot)$ are independent, and hence as well  $\tilde{A}_n(\cdot)$ and $\tilde{B}_{n,M}(\cdot)$. 
All together with Lemma \ref{lem_A} and Proposition \ref{prop_B}, it follows for sufficiently large $n$ and $M$:
\begin{align*}
\mathbb{P}(C_{n,M}(z_1)=0,...,C_{n,M}(z_N)=0)
&\leq
\sum\limits_{(k_i)_{i=1}^N\in K}\sqrt{  \prod\limits_{\{i \colon k_i=0\}}q_\mathrm{A}\prod\limits_{\{i \colon k_i=1\}}q_\mathrm{B}}\\
&=
\sum\limits_{(k_i)_{i=1}^N\in K}  \prod\limits_{\{i \colon k_i=0\}}\sqrt{q_\mathrm{A}}\prod\limits_{\{i \colon k_i=1\}}\sqrt{q_\mathrm{B}}\\
&=
(\sqrt{q_\mathrm{A}}+\sqrt{q_\mathrm{B}})^N\\
&:=
q_\mathrm{C}^N.
\end{align*}
Note that $q_\mathrm{C} <1$ exists, as Proposition \ref{prop_B} allows us to choose a constant $q_\mathrm{B}$ with $q_\mathrm{B} < (1-\sqrt{q_\mathrm{A}})^2$ for large enough $n$ and $M$. \\
\end{proof}
As we have shown Proposition \ref{prop_C}, the percolation of the lattice follows from the usual Peierls argument. Consider the proof of the supercritical phase within the proof of Theorem \ref{haselnuss} for further details. 
\\ \
\subsubsection{Step 3: Percolation in the SINR graph}
In this step we show that an infinite connected component in our lattice implies an infinite connected component in the SINR model. 
Therefore, we need the following corollary.
\begin{cor}\label{cor_shot_noise}
For $a \geq 0$ it holds that 
$
I(x) \leq I_a(z)
$
for all $x \in \mathbb{R}^d$ and $z\in \text{Q}_a(x)$.
\end{cor}
\begin{proof}
First, note that for $z\in \text{Q}_a(x)$, it holds $|z-x| \leq \frac{a \sqrt{d}}{2}$. In other words, the maximal distance between any point in the cube $\text{Q}_{a}(x)$ and its center $x$ is given by $a\sqrt{d}/2$. 
Using this and the triangle inequility yields
\begin{align*}
|z-X_i| - \frac{a \sqrt{d}}{2}
&=
|z-x+x-X_i| - \frac{a \sqrt{d}}{2}\\
&\leq
|z-x|+|x-X_i| - \frac{a \sqrt{d}}{2}\\
&\leq
|x-X_i|.
\end{align*}
Hence it follows
\begin{align*}
I(x)
&=
\sum\limits_{X_i \in X^\lambda}\ell(|x-X_i|)\\
&=
\sum\limits_{X_i \in X^\lambda}
\Big(
\underbrace{\ell(|x-X_i|)}_{\leq \ell(0)} \cdot \mathds{1}\Big\{|z-X_i| < \frac{a \sqrt{d}}{2}\Big\}
+ 
\underbrace{\ell(|x-X_i|)}_{\leq \ell(|z-X_i|)} \cdot \mathds{1}\Big\{|z-X_i| \geq \frac{a \sqrt{d}}{2}\Big\}
\Big)\\
&\leq
\sum\limits_{X_i \in X^\lambda} \ell_a(|z-X_i|)
=
I_a(z).
\end{align*}
\end{proof}
Let us now choose $\lambda$, $n$ and $M$ such that there exists an infinite connected component $\mathcal{C}$ of sites $z$ in the lattice $\mathbb{Z}^d$. This means that for $z \in \mathcal{C}$, it holds $C_{n,M}(z) =1$ and hence by definition $A_n(z)=B_{n,M}(z)=1$. 
Remember that we have shown in Step 2 that it is possible to obtain such a infinite connected component.
\medskip
\\ \indent
Let $z$, $z' \in \mathcal{C}$ be two neighbouring sites in the lattice $\mathbb{Z}^d$, i.e., $|z-z'|=1$. First, we use that $z$ and $z'$ are $n$-good, as $A_n(z)=A_n(z')=1$. Hence, there exist by definition $X_i \in \text{Q}_n(nz)$ and $X_j \in \text{Q}_n(nz')$, which implies as well $X_i, X_j \in \text{Q}_{3n}(nz)$. By using this and again the definition of $n$-goodness, it follows that $X_i$ and $X_j$ are connected by a path in $g_r(X^\lambda) \cap\text{Q}_{6n}(nz)$. 
In other words, for each $z$ in the infinite connected component of our lattice, there exist two Cox points $X_i$ and $X_j$, which are at most $r$ away from each other.\medskip
\\ \indent
Moreover, we know that $B_{n,M}(z)=1$ for $z \in \mathcal{C}$, which means $I_{6n}(nz)\leq M$ by definition.
From Corollary \ref{cor_shot_noise} follows
$I(x) \leq I_{6n}(nz)$ for all $x \in \mathbb{R}$ and $nz \in \text{Q}_{6n}(x)$, and hence 
$I(x) \leq I_{6n}(nz)$ for all $x \in \text{Q}_{6n}(nz)$.
All together, it follows that 
$I(x) \leq M$ for all  $x \in \text{Q}_{6n}(nz)$. We use this upper bound for $x = X_i, X_j$.
Note thereby that 
\begin{equation*}
\sum\limits_{k \neq i,j}\ell(|X_k-x|)
= 
I(x)- \ell(|X_j-x|)-\ell(|X_j-x|)
<
I(x)
\leq M.
\end{equation*} 
Using all this, it holds that the points $X_i$ and $X_j$ are connected in $g_{(\gamma)}(X^\lambda)$, where $\gamma \in (0,\gamma')$ with
\begin{equation*}
\gamma'=
\frac{N_0}{PM}\Big(\frac{\ell(r)}{\ell(r_\mathrm{B})}-1\Big),
\end{equation*}
as the SINR is greater than the threshold $\tau$. This can be seen in the following equation, which holds for any $X_i$, $X_j \in X^\lambda \cap \text{Q}_{6n}(nz)$ with $|X_i-X_j| \leq r$
\begin{align*}
\frac{P \ell(|X_i-X_j|)}{N_0+ \gamma \sum\limits_{k \neq i,j}P \ell(|X_k-X_j|)}
>
\frac{P \ell(r)}{N_0+ \gamma P M}
>
\frac{P \ell(r)}{N_0+ \gamma' P M}
=
\frac{P \ell(r_\mathrm{B})}{N_0}
=
\tau.
\end{align*}
Note that the parameter $\gamma' $ is strictly positive as $d_0<r<r_\mathrm{B}$ and thus $\ell(r) >\ell(r_\mathrm{B})$.
Hence, the points $X_i$ and $X_j$ are connected in $g_{(\gamma)}(X^\lambda)$.
\medskip
\\ \indent
All together, we can conclude that the intersection $g_{(\gamma)}(X^\lambda)\cap \Big(\bigcup\limits_{z \in \mathcal{C}}\text{Q}_{6n}(nz)\Big)$ contains an infinite path and this implies that $g_{(\gamma)}(X^\lambda)$ percolates.
\\ \
\subsubsection{Interference control: Proof of Proposition \ref{prop_B}}\label{sec_B}
Let us split the interference term for any $x \in \mathbb{R}^d$ and $ n\geq1$ in two parts. The first one covers the interference at the point $nx$, which is generated from the Cox points located \textit{inside} the cube $\text{Q}_{12n\sqrt{d}}(nx)$. The second part represents the interference at $nx$, which comes from all the other Cox points \textit{outside} of the same cube. Hence, we obtain
\begin{equation*}
I_{6n}^{\text{in}}(nx) = \sum\limits_{X_i\in X^\lambda \cap \text{Q}_{12n\sqrt{d}}(nx)} \ell_{6n}(|X_i-nx|)
\,\,
\text{ and }
\,\,
I_{6n}^{\text{out}}(nx) = \sum\limits_{X_i\in X^\lambda \setminus \text{Q}_{12n\sqrt{d}}(nx)} \ell_{6n}(|X_i-nx|).
\end{equation*}
Note that  $I_{6n}(nx)= I_{6n}^{\text{in}}(nx) +I_{6n}^{\text{out}}(nx)$.
Then we define $B_{n,M}^{\text{in}}(z):=\mathds{1}\{I_{6n}^{\text{in}}(nz)\leq M\}$ and $B_{n,M}^{\text{out}}(z):=\mathds{1}\{I_{6n}^{\text{out}}(nz)\leq M\}$ and show Proposition \ref{prop_B} replacing all $B_{n,M}(\cdot)$ once by $B_{n,M}^{\text{in}}(\cdot)$ and once by $B_{n,M}^{\text{out}}(\cdot)$.
Thereby, the proof is structered in six steps. This is illustrated in the following figure:
\begin{center}
\includegraphics[width=0.6\textwidth]{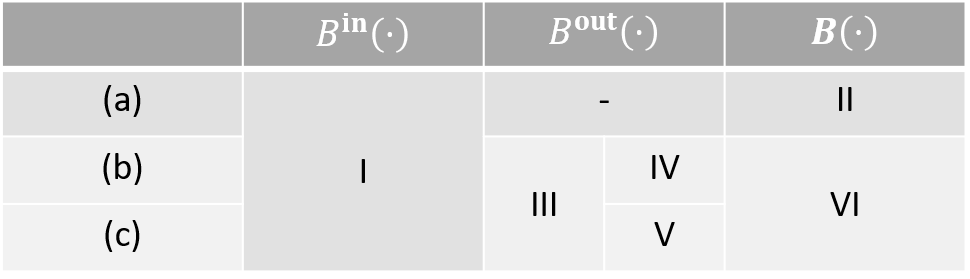}
\end{center}

The first left column describes the three different conditions of Theorem \ref{thm_phase_transition}. Remember that we aim to show that the SINR model percolates if one of the conditions holds true.
We start by proving the proposition for $B^{\text{in}}(\cdot)$. Within this proof, none of the conditions (a)-(c) are used and hence the proof holds under all three conditions. In the second step we conclude from Step I that the original proposition already holds under the condition of (a). Afterwards we proof the proposition for $B^{\text{out}}(\cdot)$. In the beginning, we consider the conditions (c) and (d) together - that is Step III - later on we need to consider them separatly. This is done in Step IV and V. In the last step we finish the proof by using Step I and III-V.\\

\noindent \textbf{Step I}\\
In order to show the proposition with $B_{n,M}^{\text{in}}(\cdot)$, we construct a renormalized percolation process and say a site $z \in \mathbb{Z}^d$ is \textit{n-tame} if
\begin{equation*}
R(\text{Q}_{12n\sqrt{d}}(nz))<n/2
\,\,\,\text{ and }\,\,\,
I_{6n}^{\text{in}}(nz)\leq M.
\end{equation*}
Otherwise, it is called \textit{n-wild}. 
\medskip
\\ \indent
From the definition of stabilization follows that the $n$-tame site process is $\lceil 12n\sqrt{d}+1\rceil$-dependent. By using dependent percolation theory (see \cite[Theorem 0.0]{Liggett1997}), it suffices to show that for all $\lambda>0$, the probability that $z$ is $n$-wild can be made arbitrarily close to zero uniformly in $z \in \mathbb{Z}^d$ by choosing first $n$ sufficiently large and then $M$ large enough accordingly. We have
\begin{align*}
\mathbb{P}(z \text{ is $n$-wild})
&=
\mathbb{P}(\text{Q}_{12n\sqrt{d}}(nz)\geq n/2 \text{ or }I_{6n}^{\text{in}}(nz)>M)\\
&\leq
\mathbb{P}(\text{Q}_{12n\sqrt{d}}(nz)\geq n/2)+\,\mathbb{P}(I_{6n}^{\text{in}}(nz)>M).
\end{align*} 
From the definition of stabilization it follows that $\mathbb{P}(R(\text{Q}_n)\geq n)$ tends to zero as $n \rightarrow \infty$. Hence the first term
$\mathbb{P}(\text{Q}_{12n\sqrt{d}}(nz)\geq n/2)$
can be made arbitrarily small by choosing $n$ large enough. Furthermore, for all $z \in \mathbb{Z}^d$ it holds
\begin{equation*}
I_{6n}^{\text{in}}(nz)
=
 \sum\limits_{X_i\in X^\lambda \cap \text{Q}_{12n\sqrt{d}}(nz)} \underbrace{\ell_{6n}(|X_i-nz|)}_{\leq \ell (0)}
\leq 
\ell(0) \,\#\{X^\lambda \cap \text{Q}_{12n\sqrt{d}}(nz)\}
\end{equation*}
and hence
\begin{align*}
\mathbb{E}[I_{6n}^{\text{in}}(nz)]
\leq
\ell(0) \,\lambda \,\mathbb{E}[\Lambda(\text{Q}_{12n\sqrt{d}}(nz))]
=
\ell(0) \,\lambda\, (12n\sqrt{d})^d
<\infty.
\end{align*}
Thus, for $n\geq1$, we can uniformly in $z \in \mathbb{Z}^d$ make $\mathbb{P}(I_{6n}^{\text{in}}(nz)>M)$ arbitrarily small by choosing $M$ large enough.
Hence, we have shown Proposition \ref{prop_B} for $B_{n,M}(\cdot)$ replaced by $B_{n,M}^{\text{in}}(\cdot)$.\\

\noindent \textbf{Step II}\\
From Step I follows easily that the original Proposition \ref{prop_B} with no replacements of $B_{n,M}(\cdot)$ holds under condition (a) of Theorem \ref{thm_phase_transition}, where $\ell$ has compact support. This implies that the support of $\ell$ is bounded, which means that $\sup(\text{supp } \ell )$ is finite. Hence we can conclude on the one hand that
$\ell(r)=0$ for $r>\sup(\text{supp } \ell )$ and by definition of the shifted version of $\ell$ it follows that
$\ell_{6n}(r)=0$ for  $r>\sup(\text{supp } \ell ) + 3n\sqrt{d}$.
On the other hand, we can conclude that $\sup(\text{supp } \ell ) + 3n\sqrt{d} \leq 12n\sqrt{d}$ for large enough $n$. This yields
\begin{align*}
I_{6n}^{\text{in}}(nz)
&\leq
I_{6n}(nz)\\
&=
\sum\limits_{X_i\in X^\lambda} \ell_{6n}(|X_i-nz|)
\\
&=
\sum\limits_{X_i\in X^\lambda \cap \text{Q}_{\sup(\text{supp } \ell ) + 3n\sqrt{d}}(nz)} \ell_{6n}(|X_i-nz|)
\\
&\leq
\sum\limits_{X_i\in X^\lambda \cap \text{Q}_{12n\sqrt{d}}(nz)} \ell_{6n}(|X_i-nz|)
\\
&=
I_{6n}^{\text{in}}(nz)
\end{align*}
for sufficiently large $n$ and hence the proposition follows under condition (a) of Theorem \ref{thm_phase_transition}.\\

\noindent \textbf{Step III}\\
We still need to show Proposition \ref{prop_B} for $B_{n,M}^{\text{out}}(\cdot)$ considering the conditions (b) or (c) of Theorem \ref{thm_phase_transition}. In the following, we assume without loss of generality that $\text{supp }\ell = [0,\infty)$.\\
First, we apply a general estimation:
\begin{align*}
\mathbb{P}(B_{n,M}^{\text{out}}(z_1)=0,...,B_{n,M}^{\text{out}}(z_N)=0)
&=
\mathbb{P}(I_{6n}^{\text{out}}(nz_1)>M,...,I_{6n}^{\text{out}}(nz_N)>M)\\
&\leq
\mathbb{P}\Big(\sum\limits_{i=1}^N I_{6n}^{\text{out}}(nz_i)>NM \Big).
\end{align*}
In the next step, we use the Marcov's inquality and obtain for any $s>0$
\\
$\mathbb{P}\Big(\sum\limits_{i=1}^N I_{6n}^{\text{out}}(nz_i)>NM \Big)$
\begin{equation}\label{exp_0}
\,\,\,\,\,\,\,\,\,\,\,\,\,
\leq
\exp(-sNM) \underbrace{ \mathbb{E} \Big[ \exp \Big( s \sum\limits_{i=1}^N \sum\limits_{X_k \in X^\lambda \setminus \text{Q}_{12n\sqrt{d}(nz_i)}} \ell_{6n}(|nz_i-X_k|) \Big) \Big]}_{(*)}.
\end{equation}
In order to estimate the expression $(*)$ further, we apply the form of the Laplace functional of a CPP to the function
$
f(x)= s \sum\limits_{i=1}^N\ell_{6n}(|nz_i-x|) \cdot \mathds{1}\{x\in \mathbb{R}^d \setminus \text{Q}_{12n\sqrt{d}}(nz_i)\} 
$
and obtain
\begin{equation}\label{exp_1}
(*)
=
\mathbb{E} \Big[ \exp \Big( \lambda \int\limits_{\mathbb{R}^d} \Big( \exp \Big (s \sum\limits_{i=1}^N\ell_{6n}(|nz_i-x|) \mathds{1}\{x\in \mathbb{R}^d \setminus \text{Q}_{12n\sqrt{d}}(nz_i)\} \Big)-1 \Big) \Lambda(dx)\Big) \Big].
\end{equation}
For further estimations of the expression (\ref{exp_1}) we use the following inequality
\begin{equation*}
\exp(y)-1 \leq 2y \,\,\, \text{ for all } \,\,\, y \leq 1.
\end{equation*}
Hence, we need to show
\begin{equation}\label{inequ}
s \sum\limits_{i=1}^N\ell_{6n}(|nz_i-x|) \cdot \mathds{1}\{x\in \mathbb{R}^d \setminus \text{Q}_{12n\sqrt{d}}(nz_i)\} \leq 1.
\end{equation} 
As the points $z_1,...,z_N$ are pairwise distinct by assumption, we can apply the following approximation
\begin{equation*}
 \sum\limits_{i=1}^N\ell_{6n}(|nz_i-x|)\cdot \mathds{1}\{x\in \mathbb{R}^d \setminus \text{Q}_{12n\sqrt{d}}(nz_i)\} 
 \leq  \sum\limits_{z\in n\mathbb{Z}^d}\ell_{6n}(|z-x|).
\end{equation*}
\indent
Furthermore, we denote the cube in $n\mathbb{Z}^d$ which contains $x\in \mathbb{R}^d$ by $P_x$.
Now the idea is to classify the $z \in n\mathbb{Z}^d$ according to their distance to $P_x$ in order to determine their maximal contribution to the above sum. More precisely we have
\begin{equation*}
n \mathbb{Z}^d = \bigcup\limits_{i=0}^\infty \{z \in n \mathbb{Z}^d \colon i \leq \text{dist}_\infty(z,P_x) < (i+1)\},
\end{equation*}
where 
$\text{dist}_p(\phi, \psi) = \inf \{ ||x-y||_p \colon x \in \phi, y \in \psi\}
$
denotes the $\ell^{p}$-distance between two sets $\phi, \psi \subset \mathbb{R}^d$ for $p \in [1,\infty]$.
Then, we can easily see that for any $i \in \mathbb{N}$ and $\tilde{z} \in \{z \in n \mathbb{Z}^d \colon i \leq \text{dist}_\infty(z,P_x) < (i+1)\}$ it holds
$|\tilde{z}-x| \geq ni$ and hence
$\ell_{6n}(|\tilde{z}-x|) \leq \ell_{6n}(ni)$.
This yields
\begin{align*}
 \sum\limits_{z \in n\mathbb{Z}^d}\ell_{6n}(|z-x|) 
&\leq
 \sum\limits_{i=0}^\infty \#\{z\in n\mathbb{Z}^d \colon i\leq \text{dist}_\infty(z,P_x)<(i+1)\}\,\cdot \, \ell_{6n}(ni)\\
&=
 \sum\limits_{i=0}^\infty ((2i+2)^d-(2i)^d)\,\cdot \, \ell_{6n}(ni)\\
&:=K(n) \in (0,\infty].
\end{align*}
Let us now have a closer look at the expression $K(n)$.
It holds for any $n \in \mathbb{N}$ and $i \geq 6\sqrt{d}/2$ clearly that $ni \geq 6n\sqrt{d}/2$ and hence
\begin{equation*}
\ell_{6n}(ni) 
= 
\ell(ni-6n\sqrt{d}/2)
=
\ell(n(i-6\sqrt{d}/2))
\leq
\ell(i-6\sqrt{d}/2)
,
\end{equation*}
as $\ell(\cdot)$ is monotone decreasing. Furthermore, by the assumptions (i)-(ii) on $\ell$ it is clear that $\ell_{6n}(\cdot)\leq 1$ and hence
\begin{equation*}
K(n)
\leq
 \sum\limits_{i=0}^{\lceil 6\sqrt{d}/2 \rceil} ((2i+2)^d-(2i)^d) + \sum\limits_{i=\lceil 6\sqrt{d}/2 \rceil}^\infty ((2i+2)^d-(2i)^d)\cdot \ell(i-6\sqrt{d}/2)
:=K_0.
\end{equation*}
We can conclude
\begin{equation}\label{sommer}
 \sum\limits_{z \in n\mathbb{Z}^d}\ell_{6n}(|z-x|) 
\leq
K_0.
\end{equation}
By assumption (c) on $\ell$ it follows that $K_0$ is finite and hence inequality (\ref{inequ}) holds for $s\leq 1/K_0$ and we have \\

$ 
\exp \Big (s \sum\limits_{i=1}^N\ell_{6n}(|nz_i-x|) \cdot \mathds{1}\{x\in \mathbb{R}^d \setminus \text{Q}_{12n\sqrt{d}}(nz_i)\} \Big)-1 
$
\begin{equation*}
\,\,\, \, \,\,\, \,\,\,\, \,
\leq
2 \cdot s \sum\limits_{i=1}^N\ell_{6n}(|nz_i-x|) \cdot \mathds{1}\{x\in \mathbb{R}^d \setminus \text{Q}_{12n\sqrt{d}}(nz_i)\}
.
\end{equation*}
By plugging this into equation (\ref{exp_1}), we obtain
\begin{align*}
(*)
&\leq
\mathbb{E}\Big[\exp \Big(2 \lambda s \sum\limits_{i=1}^N \int\limits_{\mathbb{R}^d \setminus \text{Q}_{12n\sqrt{d}}(nz_i)} \ell_{6n}(|nz_i-x|) \Lambda(dx)\Big)\Big]\\
&=
\mathbb{E}\Big[ \prod\limits_{i=1}^N \exp \Big(2 \lambda s\int\limits_{\mathbb{R}^d \setminus \text{Q}_{12n\sqrt{d}}(nz_i)} \ell_{6n}(|nz_i-x|) \Lambda(dx)\Big)\Big]\\
&=
\mathbb{E}\Big[ \prod\limits_{i=1}^N Y^{1/N}_i\Big]
,
\end{align*}
\\
where we introduced the random variables
\begin{equation*}
Y_i := \exp \Big(2 \lambda sN\int\limits_{\mathbb{R}^d \setminus \text{Q}_{12n\sqrt{d}}(nz_i)} \ell_{6n}(|nz_i-x|) \Lambda(dx)\Big)
\,\,\, \text{ for } i=1,...,N.
\end{equation*}
As these random variables are identically distributed, we can apply the following extended version of Hölder's inequality
\begin{equation*}
\mathbb{E} \Big[ \prod\limits_{i=1}^\infty Y_i^{p_i}\Big]
\leq
\mathbb{E}[Y_1], \,\,\, \,\,\, \text{ where } p_i\geq 0 \,\,\, \text{ for all } \, i \in \mathbb{N}\,\,\, \text{ and } \sum\limits_{i=1}^\infty p_i=1
\end{equation*}
for $p_i=1/N$ for $i=1,...,N$ and $p_i=0$ elsewhere. Note that for $i>N$, we can simply define $Y_i=1$, which yields
\begin{equation*}
\mathbb{E} \Big[ \prod\limits_{i=1}^N Y_i^{1/N}\Big]
=
\mathbb{E} \Big[ \prod\limits_{i=1}^\infty Y_i^{1/N}\Big].
\end{equation*}
All together we have the following further approximation
\begin{equation*}
(*)
\leq
\mathbb{E}\Big[\exp \Big(2 \lambda sN\int\limits_{\mathbb{R}^d \setminus \text{Q}_{12n\sqrt{d}}} \ell_{6n}(|x|) \Lambda(dx)\Big)\Big]
\end{equation*}
and hence\\

$\mathbb{P}(B_{n,M}^{\text{out}}(z_1)=0,...,B_{n,M}^{\text{out}}(z_N)=0)$
\begin{equation}\label{juhu}
\,\,\, \,
\leq
\exp(-NMs) \,
\mathbb{E}\Big[\exp \Big(2 \lambda sN\int\limits_{\mathbb{R}^d \setminus \text{Q}_{12n\sqrt{d}}} \ell_{6n}(|x|) \Lambda(dx)\Big)\Big].
\end{equation}
\\
\textbf{Step IV}\\
Now, we consider the cases (b) and (c)  in Theorem \ref{thm_phase_transition} separatly. First, we consider case (b), where $\Lambda(\text{Q}_1)$ is almost surely bounded. Therefore we can easily see that inequality (\ref{juhu}) implies\\

$\mathbb{P}(B_{n,M}^{\text{out}}(z_1)=0,...,B_{n,M}^{\text{out}}(z_N)=0)$
\begin{equation}
\,\,\, \,
\leq
\exp(-NMs) \,
\mathbb{E}\Big[\exp \Big(2 \lambda sN\int\limits_{\mathbb{R}^d} \ell_{6n}(|x|) \Lambda(dx)\Big)\Big].
\end{equation}
In the following we claim that there exists $W\geq0$ such that
\begin{equation*}
\limsup\limits_{N \rightarrow \infty} \frac{1}{N}\log \mathbb{E}\Big[\exp \Big(2 \lambda sN\int\limits_{\mathbb{R}^d} \ell_{6n}(|x|) \Lambda(dx)\Big)\Big]
\leq
2\lambda s W.
\end{equation*}
Let therefore be $\{\text{Q}^i\}_{i=1}^\infty$ a subdivision of $\mathbb{R}^d$ into congruent copies of $\text{Q}_1$ (up to the boundaries). Remember that  $\text{Q}_1$ is the cube with side length $1$, centered at the origin. Then we can estimate
\begin{align*}
\mathbb{E}\Big[\exp \Big(2 \lambda sN\int\limits_{\mathbb{R}^d} \ell_{6n}(|x|) \Lambda(dx)\Big)\Big]
&\leq
\mathbb{E}\Big[\exp \Big(2 \lambda sN\sum\limits_{i=1}^\infty \max\limits_{x\in \text{Q}^i} \ell_{6n}(|x|) \Lambda(\text{Q}^i)\Big)\Big]\\
&=
\mathbb{E}\Big[\prod\limits_{i=1}^\infty\exp \Big(2 \lambda sN \max\limits_{x\in \text{Q}^i} \ell_{6n}(|x|) \Lambda(\text{Q}^i)\Big)\Big]\\
&=
\mathbb{E}\Big[\prod\limits_{i=1}^\infty\exp \Big(2 \lambda sN \Big(\sum\limits_{j=1}^\infty\max\limits_{x\in \text{Q}^j} \ell_{6n}(|x|)\Big) \Lambda(\text{Q}^i)\Big)^{p_i}\Big],
\end{align*}
where we used $p_i=\max\limits_{x\in \text{Q}^i} \ell_{6n}(|x|) / \sum\limits_{j=1}^\infty\max\limits_{x\in \text{Q}^j} \ell_{6n}(|x|)$. Note that $\sum\limits_{i=1}^\infty\max\limits_{x\in \text{Q}^i} \ell_{6n}(|x|)$ is finite by assumption on $\ell$. As $p_i \geq 0$ and 
$\sum\limits_{i=1}^\infty p_i=1$, we can again apply the extended Hölder's inequality from above and obtain
\begin{align*}
\mathbb{E}\Big[\exp \Big(2 \lambda sN\int\limits_{\mathbb{R}^d} \ell_{6n}(|x|) \Lambda(dx)\Big)\Big]
&\leq
\mathbb{E}\Big[\exp \Big(2 \lambda sN  \Lambda(\text{Q}_1)\sum\limits_{j=1}^\infty \max\limits_{x\in \text{Q}^j} \ell_{6n}(|x|)\Big)\Big].
\end{align*}
By definition it holds
$\text{esssup } \Lambda(\text{Q}_1) = \inf \{t>0 \colon \mathbb{P} (\Lambda(\text{Q}_1)<t)=1\}$
and hence $ \Lambda(\text{Q}_1) \leq\text{esssup } \Lambda(\text{Q}_1)$ almost surely. This yields
\begin{align*}
\mathbb{E}\Big[\exp \Big(2 \lambda sN\int\limits_{\mathbb{R}^d} \ell_{6n}(|x|) \Lambda(dx)\Big)\Big]
&\leq
\exp \Big(2 \lambda sN  \text{esssup} \Lambda(\text{Q}_1)\sum\limits_{j=1}^\infty \max\limits_{x\in \text{Q}^j} \ell_{6n}(|x|)\Big).
\end{align*}
In the next step, we take $\lim \sup\limits_{N\rightarrow \infty} \frac{1}{N} \log$ on both sides and have
\begin{align*}
\limsup\limits_{N \rightarrow \infty} \frac{1}{N}\log 
\Big(
\mathbb{E}\Big[\exp \Big(2 \lambda sN\int\limits_{\mathbb{R}^d} \ell_{6n}(|x|) \Lambda(dx)\Big)\Big]
\Big)
&\leq
2 \lambda s \text{ esssup} \Lambda(\text{Q}_1) \sum\limits_{j=1}^\infty \max\limits_{x\in \text{Q}^j} \ell_{6n}(|x|) .
\end{align*}
Hence, by choosing
$W= \text{ esssup} \Lambda(\text{Q}_1) \sum\limits_{j=1}^\infty \max\limits_{x\in \text{Q}^j} \ell_{6n}(|x|) $, which is finite as $\Lambda(\text{Q}_1)$ is bounded, the claim holds. All in all we get
\begin{equation*}
\mathbb{E}\Big[\exp \Big(2 \lambda sN\int\limits_{\mathbb{R}^d} \ell_{6n}(|x|) \Lambda(dx)\Big)\Big]
\leq
\exp(2\lambda s N (W+ \mathcal{O}(1)))
\end{equation*}
and hence
\begin{align*}
\mathbb{P}(B_{n,M}^{\text{out}}(z_1)=0,...,B_{n,M}^{\text{out}}(z_N)=0)
&\leq
\exp(2\lambda s N (W+ \mathcal{O}(1))) \cdot \exp(-NMs) 
\\
&=
(\exp(2\lambda s (W+ \mathcal{O}(1))-Ms))^N.
\end{align*}
\indent
Thus, we have found
$q_\mathrm{B}:=\exp(3\lambda s W-Ms)$ for which the inequality in Proposition \ref{prop_B} holds in case (b) of Theorem  \ref{thm_phase_transition}. Furthermore, for any $\lambda >0$, $q_\mathrm{B}$ can be made arbitrarily small by choosing $M>0$ large enough. \\

\noindent\textbf{Step V}\\
We still need to show that the proposition holds in case (c) of Theorem \ref{thm_phase_transition}, which means in case that we have
$\mathbb{E}[\exp(\alpha\Lambda(\text{Q}_1))]<\infty$ for some $\alpha >0$ and $\int\limits_x^\infty r^{d-1}\ell(r) dr = \mathcal{O}\big(\frac{1}{x}\big)$ as $x \rightarrow \infty$. 
Therefore, we need the following lemma.
\begin{lem}\label{lem_shot_noise}
For $n \geq 1$, $a \geq n$, $k \in \mathbb{N}$, $z \in k \cdot \mathbb{Z}^d$ the following inequality holds almost surely for all 
$y \in \text{Q}_k(z) \cap \mathbb{Z}^d$ and measurable $S \subset \mathbb{R}^d$:
\begin{equation*}
\sum\limits_{X_j \in X^\lambda \cap S} \ell_{ka}(|nz-X_j|)
\geq
\sum\limits_{X_j \in X^\lambda \cap S} \ell_{a}(|ny-X_j|).
\end{equation*}
\end{lem}
\begin{proof}
By definition it holds
\begin{equation*}
\ell_{ka}(|nz-x|)
=
\begin{cases}
\ell(0) 							& \text{if } |nz-x|<ka\sqrt{d}/2\\
\ell(|nz-x|-ka\sqrt{d}/2)	& \text{elsewhere}
\end{cases}
\end{equation*}
and 
\begin{equation*}
\ell_{a}(|ny-x|)
=
\begin{cases}
\ell(0) & \text{if } |ny-x|<a\sqrt{d}/2\\
\ell(|ny-x|-a\sqrt{d}/2)	& \text{elsewhere.}
\end{cases}
\end{equation*}
\indent
As the path-loss function $\ell$ is decreasing, we first need to show that 
$|ny-x|<a\sqrt{d}/2$  implies 
$|nz-x|<ka\sqrt{d}/2$. Note that
$|z-y|\leq(k-1)\sqrt{d}/2$ as $y$ is an integer number and $\text{Q}_k(z)$ is the cube without boundaries.
Hence, by the triangle inequality follows
\begin{equation*}
|nz-x|
\leq
|nz-ny|+|ny-x|
\leq
a \cdot|z-y|+|ny-x|
<
a(k-1)\sqrt{d}/2+a\sqrt{d}/2
=
ka\sqrt{d}/2.
\end{equation*}
Thus, it is clear that 
$\ell_{ka}(|nz-x|)\geq \ell_{a}(|ny-x|)$
as long as $|ny-x|<a\sqrt{d}/2$. \medskip
\\ \indent
Now we consider the case where $|ny-x|\geq a\sqrt{d}/2$. As $\ell$ is decreasing, the inequality holds as well as long as $|nz-x|<ka\sqrt{d}/2$. Hence, we still need to consider the case $|nz-x|\geq ka\sqrt{d}/2$, which implies $|ny-x|\geq a\sqrt{d}/2$, and need to show that 
\begin{equation}\label{sunrise_lem}
\ell(|nz-x|-ka\sqrt{d}/2) \geq\ell(|ny-x|-a\sqrt{d}/2)
\end{equation}
holds true. Again, as $\ell$ is a decreasing function, equation (\ref{sunrise_lem}) is equivalent to
$|nz-x|-ka\sqrt{d}/2 \leq |ny-x|-a\sqrt{d}/2$. This can be shown in a similar way as above:
\begin{align*}
|nz-x|-ka\sqrt{d}/2
&\leq
|nz-ny|+|ny-x|-ka\sqrt{d}/2\\
&\leq
a \cdot (k-1)\sqrt{d}/2+|ny-x|-ka\sqrt{d}/2\\
&=
|ny-x|-a\sqrt{d}/2.
\end{align*}
All together we have shown for any $X_j$ in $X^\lambda \cap S$ that
$\ell_{ka}(|nz-X_j|)
\geq
\ell_{a}(|ny-X_j|)$
and hence the assertion follows.\\
\end{proof}
Let us now continue with equation (\ref{juhu}) from before:
\\

$\mathbb{P}(B_{n,M}^{\text{out}}(z_1)=0,...,B_{n,M}^{\text{out}}(z_N)=0)$
\begin{equation}\label{star_equ}
\,\,\,
\leq
\exp(-NMs)
\underbrace{\mathbb{E}\Big[\exp \Big(2 \lambda sN\int\limits_{\mathbb{R}^d \setminus \text{Q}_{12n\sqrt{d}}} \ell_{6n}(|x|) \Lambda(dx)\Big)\Big]}_{(*)}.
\end{equation}
Again, we aim to estimate $(*)$ by an upper bound. We start by extending the integration domain from
$\mathbb{R}^d \setminus \text{Q}_{12n\sqrt{d}}$ to $\mathbb{R}^d \setminus \text{Q}_{\lfloor12n\sqrt{d}\rfloor}$ and subdividing it into the sets
$A^i=\text{Q}_{\lfloor12n\sqrt{d}\rfloor +2i+2} \setminus \text{Q}_{\lfloor12n\sqrt{d}\rfloor+2i}$ for $i \in \mathbb{N}_0$. This means
\begin{equation*}
\mathbb{R}^d \setminus \text{Q}_{\lfloor12n\sqrt{d}\rfloor}
=
\bigcup\limits_{i=0}^\infty A^i.
\end{equation*}
Each $A^i$ is covered by the union of 
$(\lfloor12n\sqrt{d}\rfloor +2i+2)^d -( \lfloor12n\sqrt{d}\rfloor +2i)^d$ 
congruent copies of $\Lambda(\text{Q}_1)$ and the number of copies equals $\text{Leb}(A^i)$.
Moreover, for any $i \in \mathbb{N}_0$ and for all $x \in A^i$ it holds
$
|x| 
\geq
\lfloor12n\sqrt{d}\rfloor +2i
$
and hence 
\begin{equation*}
\ell_{6n}(|x|) 
\leq
\ell_{6n}\big(\lfloor12n\sqrt{d}\rfloor +2i\big) 
=
\ell\Big(\lfloor12n\sqrt{d}\rfloor +2i - \frac{6n\sqrt{d}}{2}\Big) 
=
\ell\big(\lfloor12n\sqrt{d}\rfloor +2i- 3n\sqrt{d}\big) 
.
\end{equation*}
Furthermore, we have
$\lfloor12n\sqrt{d}\rfloor + i \geq 12n\sqrt{d}$ as $i \geq 1$. This yields
\begin{equation*}
\lfloor12n\sqrt{d}\rfloor +2i- 3n\sqrt{d}
\geq
12n\sqrt{d} +i-3n\sqrt{d}
=
9n\sqrt{d} +i
\geq
2n\sqrt{d} +i,
\end{equation*}
and hence it holds that
$\ell_{6n}(|x|)  \leq \ell(2n\sqrt{d} +i)$.
We obtain
\begin{align*}
(*)
&\leq
\mathbb{E}\Big[\exp \Big(2 \lambda sN\int\limits_{\mathbb{R}^d \setminus \text{Q}_{\lfloor12n\sqrt{d}\rfloor}} \ell_{6n}(|x|) \Lambda(dx)\Big)\Big]\\
&=
\mathbb{E}\Big[\exp \Big(2 \lambda sN\sum\limits_{i=0}^\infty \int\limits_{A^i} \ell_{6n}(|x|) \Lambda(dx)\Big)\Big]
\\
&\leq
\mathbb{E}\Big[\exp \Big(2 \lambda sN\sum\limits_{i=0}^\infty \int\limits_{A^i}\ell(2n\sqrt{d} +i) \Lambda(dx)\Big)\Big]
\\
&=
\mathbb{E}\Big[\exp \Big(2 \lambda sN\sum\limits_{i=0}^\infty \ell(2n\sqrt{d} +i) \Lambda(A^i)\Big)\Big]
\\
&=
\mathbb{E}\Big[\prod\limits_{i=0}^\infty\exp \Big(2 \lambda sN \Lambda(A^i) \sum\limits_{j=1}^\infty\ell(2n\sqrt{d} +j) \text{Leb}(A^j) \Big)^{p_i}\Big],
\end{align*}
where
$p_i = \ell(2n\sqrt{d} +i) / \sum\limits_{j=0}^\infty\ell(2n\sqrt{d} +j)\text{Leb}(A^j)$. Thus, we can again apply the extended Hölder's inequality from above and get
\begin{align*}
(*)
&\leq
\mathbb{E}\Big[\exp \Big(2 \lambda sN \Lambda(A^1) \sum\limits_{j=0}^\infty\ell(2n\sqrt{d} +j)  \text{Leb}(A^j)\Big)\Big]
\\
&=
\mathbb{E}\Big[\exp \Big(2 \lambda sN  \sum\limits_{j=0}^\infty\ell(2n\sqrt{d} +j)  \text{Leb}(A^j)\Lambda(\text{Q}_1)\Big)\Big].
\end{align*}
\newpage
\noindent
For $j \in \mathbb{N}_0$, we have
$\text{Leb}(A^j )
\leq
d(\lfloor 12n\sqrt{d} \rfloor+2j+2)^{d-1}
\leq
d(14n\sqrt{d}+2j+2 )^{d-1}
\leq
d \, 7^{d-1}(2n\sqrt{d}+j)^{d-1}
$
and hence
\begin{equation*}
(*)
\leq
\mathbb{E}\Big[\exp \Big(2 \lambda sN d \, 7^{d-1}\sum\limits_{j=0}^\infty\ell(2n\sqrt{d} +j) (2n\sqrt{d}+j)^{d-1}\Lambda(\text{Q}_1)\Big)\Big].
\end{equation*}
By the constraint on $\ell$ in case (c) of Theorem \ref{thm_phase_transition} we know that the sum 
\begin{equation*}
\sum\limits_{j=0}^\infty\ell(2n\sqrt{d} +j) (2n\sqrt{d}+j)^{d-1}
=\sum\limits_{j=2n\sqrt{d}}^\infty\ell(j) j^{d-1}
\end{equation*}
tends to $\mathcal{O}(1/n)$ as $n \rightarrow \infty$.
Together with equation (\ref{star_equ}), it follows
\begin{equation}\label{superestimate}
\mathbb{P}(B_{n,M}^{\text{out}}(z_1)=0,...,B_{n,M}^{\text{out}}(z_N)=0)
\leq
\exp(-NMs)
\mathbb{E}\Big[ \exp\Big(Cs\frac{N}{n}\Lambda(\text{Q}_1)\Big) \Big]
\in [0,\infty]
\end{equation}
for a suitable $C>0$, given that $s \leq 1/K_0$.\medskip
\\ \indent
In the next step of our proof we set $N=kn$ for $k, n \in\mathbb{N}$.
Furthermore, for any $z' \in\mathbb{Z}^d$ and $k\in\mathbb{N}$ we define $z:=z(z',k)$ by the point for which holds that it is contained in $k\mathbb{Z}^d$ and 
$z' \in \text{Q}_k(z)$. 
Then, we apply Lemma \ref{lem_shot_noise} for each $i \in \{1,...,kn\}$. The lemma has a few conditions, which are all fulfilled as $n \geq 1$, $6n\geq n$, $k \in \mathbb{N}$, $z(z_i,k) \in k \mathbb{Z}^d$, $z_i \in \text{Q}_k(z(z_i,k)) \cap \mathbb{Z}^d$ and $\mathbb{R}^d \setminus \text{Q}_{12n\sqrt{d}}(nz_i) \subset \mathbb{R}^d$ is measurable. Hence, it follows
\begin{equation*}
\sum\limits_{X_j \in X^\lambda \setminus \text{Q}_{12n\sqrt{d}}(nz_i) }
\ell_{6kn}(|nz(z_i,k)-X_j|) 
\geq
\sum\limits_{X_j \in X^\lambda \setminus \text{Q}_{12n\sqrt{d}}(nz_i) }
\ell_{6n}(|nz_i-X_j|).
\end{equation*}
By our definition, this is equivalent to
$
I^{\text{out}}_{6kn}(nz(z_i,k))
\geq
I^{\text{out}}_{6n}(nz_i)
$. It follows
\begin{equation*}
\mathbb{P}(I^{\text{out}}_{6n}(nz_1)>M,...,I^{\text{out}}_{6n}(nz_{kn})>M)
\leq
\mathbb{P}(I^{\text{out}}_{6kn}(nz(z_1,k))>M,...,I^{\text{out}}_{6kn}(nz(z_{kn},k))>M),
\end{equation*}
respectively
\\

$\mathbb{P}(B_{n,M}^{\text{out}}(z_1)=0,...,B_{n,M}^{\text{out}}(z_N)=0)$
\begin{equation}\label{herbst}
\,\,\,
\leq
\mathbb{P}(B_{n,M}^{\text{out}}(z(z_1,k))=0,...,B_{n,M}^{\text{out}}(z(z_{kn},k))=0).
\end{equation}
\indent
We now aim to apply equation (\ref{superestimate}) for the points $z(z_1,k),...,z(z_{kn},k)$. But as they are not necessarily disjoint, this is not directly possible. In the proof of the bound (\ref{superestimate}), the only place where we used the disjointness of the points $z_1,...,z_N$ was the estimate (\ref{sommer}). In order to replace this estimate, we show the following lemma.
Hence, the crucial point of the lemma is, that the points $z_1,...,z_n \in \mathbb{Z}^d$ do not necessarily have to be disjoint.
\newpage
\begin{lem}\label{winter}
If the conditions on the path-loss function conditions $\ell$ are satisfied, then there exists a constant $C_1>0$ such that for all $n\geq 1$, $x \in \mathbb{R}^d$ and $z_1,...,z_n \in \mathbb{Z}^d$ with $nz_1,...,nz_n \in \mathbb{R}^d \setminus \text{Q}_{12n\sqrt{d}}(x)$ it holds that
\begin{equation*}
\sum\limits_{i=1}^n \ell_{6n}(|x-nz_i|) \leq C_1.
\end{equation*}
\end{lem}
\begin{proof}
Let $n \geq 1$, $x \in \mathbb{R}^d$ and $z_1,...,z_n \in \mathbb{Z}^d$. For all $i \in\{1,...,n\}$ and $nz_i \in \mathbb{R}^d \setminus \text{Q}_{12n\sqrt{d}}(x)$, it clearly holds that the point $nz_i$ is at least $6n\sqrt{d}$ distant from $x$, that means 
$|x-nz_i| \geq 6n\sqrt{d}$. Moroever, the path-loss function $\ell$ is decreasing, and hence we obtain
\begin{equation*}
\ell_{6n}(|x-nz_i|) \leq \ell_{6n}(6n\sqrt{d}) =\ell(6n\sqrt{d}-6n\sqrt{d}/2)= \ell(3n\sqrt{d}).
\end{equation*}
Furthermore, it holds $3n\sqrt{d}\geq n$ as $d\geq2$ and hence $\ell(3n\sqrt{d})\leq \ell(n)$.
This yields
\begin{equation}\label{equation1}
\sum\limits_{i=1}^n \ell_{6n}(|x-nz_i|)
\leq
n\ell(n).
\end{equation}
By the constraint on $\ell$ in case (c) of Theorem \ref{thm_phase_transition} we know that
$n \sum\limits_{i=n}^\infty i^{d-1}\ell(i) \leq \mathcal{O}(1)$. Moreover, it holds $n^{1-d} \leq1$ as $d\geq2$. Hence we obtain
\begin{equation}\label{equation2}
n\ell(n)
=
n^{1-d}\, n  \, n^{d-1}\ell(n) 
\leq
n^{1-d}\,n \sum\limits_{i=n}^\infty i^{d-1}\ell(i)
\leq
\mathcal{O}(1).
\end{equation}
Combining equation (\ref{equation1}) and (\ref{equation2}) shows the assertion.\\
\end{proof}
Let $C_1>0$ be such that Lemma \ref{winter} holds. Then we can conclude that equation (\ref{superestimate}) holds for the not necessarily disjoint points $z(z_1,k),...,z(z_{kn},k)$ for some $C>0$ if $s\leq 1/C_1$. 
Together with equation (\ref{herbst}), it follows
\begin{align*}
\mathbb{P}(B_{n,M}^{\text{out}}(z_1)=0,...,B_{n,M}^{\text{out}}(z_N)=0)
&\leq
\exp(-knMs)
\mathbb{E}\Big[ \exp\Big(Cs\frac{kn}{n}\Lambda(\text{Q}_1)\Big) \Big]\\
&=
\Big(\exp(-Ms)\mathbb{E}\Big[ \exp\Big(\frac{Cs}{n}\Lambda(\text{Q}_1)\Big) \Big]\Big)^{kn}.
\end{align*}
Furthermore, as we consider case (c) of Theorem \ref{thm_phase_transition}, we know that the parameter
$\alpha^*=\sup\{\alpha>0 \colon \mathbb{E}[\exp(\alpha\Lambda(\text{Q}_1)))]<\infty\}$  
is positive. It follows for
$s \leq \min\{1/C_1,\alpha^*/C\}$ that
\begin{equation*}
\mathbb{E}\Big[ \exp\Big(\frac{Cs}{n}\Lambda(\text{Q}_1)\Big) \Big]
\leq
\mathbb{E}\Big[ \exp\Big(\frac{\alpha^*}{n}\Lambda(\text{Q}_1)\Big) \Big]
\longrightarrow
\exp(1)
\end{equation*}
as $n$ tends to infinity. This yields for $N=kn$
\begin{align*}
\mathbb{P}(B_{n,M}^{\text{out}}(z_1)=0,...,B_{n,M}^{\text{out}}(z_N)=0)
&\leq
(\exp(-Ms+1))^{kn}
:=q_\mathrm{B}'^{\,\,kn}
=q_\mathrm{B}'^N.
\end{align*}
Hence for $kn \leq N \leq (k+1)n$ and pairwise distinct $z_1,...,z_N \in \mathbb{Z}^d$ it holds
\begin{align*}
\mathbb{P}(B_{n,M}^{\text{out}}(z_1)=0,...,B_{n,M}^{\text{out}}(z_N)=0)
&\leq
\mathbb{P}(B_{n,M}^{\text{out}}(z_1)=0,...,B_{n,M}^{\text{out}}(z_{kn})=0)\\
&\leq
q_\mathrm{B}'^{\,\,kn}
=
q_\mathrm{B}^{2kn}
\leq
q_\mathrm{B}^N,
\end{align*}
where we defined $q_\mathrm{B}:=\sqrt{q_\mathrm{B}'}$. The last inequality holds as 
$2kn \geq (k+1)n \geq N$ for $k\geq 1$.
\\

\noindent \textbf{Step VI}\\
Remember that we have shown in Step IV that Proposition \ref{prop_B} holds for some large enough $n_1 \in \mathbb{N}$, $M_1>0$  and $q^{\text{in}}_{\mathrm{B}}<1$ if we replace all $B_{n,M}(\cdot)$ by $B_{n,M}^{\text{in}}(\cdot)$. Moreover, we have shown in the previous Step V that Proposition \ref{prop_B} holds for some large enough $n_2 \in \mathbb{N}$, $M_2>0$ and $q^{\text{out}}_{\mathrm{B}}<1$ if we replace all $B_{n,M}(\cdot)$ by $B_{n,M}^{\text{out}}(\cdot)$. 
Note that we can choose the parameters $M_1$ and $M_2$ so large that $q^{\text{in}}_{\mathrm{B}}$ respectively $q^{\text{out}}_{\mathrm{B}}$ are arbitrarily small, in specific such that $\sqrt{q^{\text{in}}_{\mathrm{B}}}+\sqrt{q^{\text{out}}_{\mathrm{B}}}<1$. 
Now, we bring both assertions together and show the original Proposition \ref{prop_B} without any replacements.
Let therefore be $n=\max(n_1,n_2)$ and $M=2 \cdot\max(M_1,M_2)$. Then it follows
\begin{align*}
\mathbb{P}(B_{n,M}(z_1)=0,...,B_{n,M}(z_N)=0)
&=
\mathbb{P}(I_{6n}(nz_1)>M,...,I_{6n}(nz_N)>M)\\
&\leq
\mathbb{P}(I^{\text{in}}_{6n}(nz_1)>M/2 \text{ or } I^{\text{in}}_{6n}(nz_1)>M/2,...)\\
&=
\mathbb{P}(I^{\text{in}}_{6n}(nz_1)>\max(M_1,M_2) \text{ or } I^{\text{in}}_{6n}(nz_1)>\max(M_1,M_2),...)\\
&\leq
\mathbb{P}(I^{\text{in}}_{6n}(nz_1)>M_1 \text{ or } I^{\text{in}}_{6n}(nz_1)>M_2,...)\\
&=
\mathbb{P}(B^{\text{in}}_{n,M}(z_1)B^{\text{out}}_{n,M}(z_1)=0,...,B^{\text{in}}_{n,M}(z_N)B^{\text{out}}_{n,M}(z_N)=0).
\end{align*}
In order to simplify the notation, we write
\medskip
\\ \indent 
$\hat{B}_{n,M}(z_i):=B^{\text{in}}_{n,M}(z_i) \cdot B^{\text{out}}_{n,M}(z_i)$, 
$\tilde{B}^{\text{in}}_{n,M}(z_i):=1-B^{\text{in}}_{n,M}(z_i)$
and
$\tilde{B}^{\text{out}}_{n,M}(z_i):=1-B^{\text{out}}_{n,M}(z_i)$. 
\medskip
\\ 
This yields
\begin{equation}\label{giraffe}
\mathbb{P}(B_{n,M}(z_1)=0,...,B_{n,M}(z_N)=0)
\leq
\mathbb{P}(\hat{B}_{n,M}(z_1)=0,...,\hat{B}_{n,M}(z_N)=0).
\end{equation}
Furthermore, with the new notation the following estimation holds
\begin{equation*}
1-\hat{B}_{n,M}(z_i)
=
1-B^{\text{in}}_{n,M}(z_i)B^{\text{out}}_{n,M}(z_i)
\leq
1-B^{\text{in}}_{n,M}(z_i)+1-B^{\text{out}}_{n,M}(z_i)
=
\tilde{B}^{\text{in}}_{n,M}(z_i)+\tilde{B}^{\text{out}}_{n,M}(z_i).
\end{equation*}
In order to estimate equation (\ref{giraffe}) further, let $(k_i)_{i=1}^N$ be a binary sequence, i.e., $k_i \in \{0,1\}$ and let $K$ denote the set of the $2^N$ such binary sequences. 
Then it follows
\begin{align*}
\mathbb{P}(\hat{B}_{n,M}(z_1)=0,...,\hat{B}_{n,M}(z_N)=0)
&=
\mathbb{P}(1-\hat{B}_{n,M}(z_1)=1,...,1-\hat{B}_{n,M}(z_N)=1)\\
&=
\mathbb{P}\Big(\prod\limits_{i=1}^N (1-\hat{B}_{n,M}(z_i))=1\Big)\\
&=
\mathbb{E}\Big(\prod\limits_{i=1}^N (1-\hat{B}_{n,M}(z_i))\Big)\\
&\leq
\mathbb{E}\Big(\prod\limits_{i=1}^N \Big(\tilde{B}^{\text{in}}_{n,M}(z_i)+\tilde{B}^{\text{out}}_{n,M}(z_i)\Big)\Big)\\
&=
\sum\limits_{(k_i)_{i=1}^N\in K} \mathbb{E}\Big( \prod\limits_{\{i \colon k_i=0\}}\tilde{B}^{\text{in}}_{n,M}(z_i)\prod\limits_{\{i \colon k_i=1\}}\tilde{B}^{\text{out}}_{n,M}(z_i) \Big)\\
&\leq
\sum\limits_{(k_i)_{i=1}^N\in K}\sqrt{ \mathbb{E}\Big( \prod\limits_{\{i \colon k_i=0\}}\Big(\tilde{B}^{\text{in}}_{n,M}(z_i)\Big)^2\prod\limits_{\{i \colon k_i=1\}} \Big(\tilde{B}^{\text{out}}_{n,M}(z_i) \Big)^2\Big)}\\
&=
\sum\limits_{(k_i)_{i=1}^N\in K}\sqrt{ \mathbb{E}\Big( \prod\limits_{\{i \colon k_i=0\}}\tilde{B}^{\text{in}}_{n,M}(z_i)\prod\limits_{\{i \colon k_i=1\}} \tilde{B}^{\text{out}}_{n,M}(z_i) \Big)}\\
&=
\sum\limits_{(k_i)_{i=1}^N\in K}\sqrt{ \mathbb{E}\Big( \prod\limits_{\{i \colon k_i=0\}}\tilde{B}^{\text{in}}_{n,M}(z_i)\Big)
\mathbb{E}\Big(
\prod\limits_{\{i \colon k_i=1\}} \tilde{B}^{\text{out}}_{n,M}(z_i) \Big)}
.
\end{align*}
Thereby, we used the Schwarz's inequality, the fact that $\tilde{B}^{\text{in}}_{n,M}(\cdot),\tilde{B}^{\text{out}}_{n,M}(\cdot) \in \{0,1\}$ and the independence of the random variables $\tilde{B}^{\text{in}}_{n,M}(\cdot)$ and $\tilde{B}^{\text{out}}_{n,M}(\cdot)$.
Together with equation (\ref{giraffe}) it follows
\begin{equation*}
\mathbb{P}(B_{n,M}(z_1)=0,...,B_{n,M}(z_N)=0)
\leq
\sum\limits_{(k_i)_{i=1}^N\in K}\sqrt{ \mathbb{E}\Big( \prod\limits_{\{i \colon k_i=0\}}\tilde{B}^{\text{in}}_{n,M}(z_i)\Big)
\mathbb{E}\Big(
\prod\limits_{\{i \colon k_i=1\}} \tilde{B}^{\text{out}}_{n,M}(z_i) \Big)}
.
\end{equation*}
\indent
In the next step we use our results from Step I and III-V. Remember that $q^{\text{in}}_\mathrm{B}$ denotes the parameter for which the proposition holds if we replace $B_{n,M}(\cdot)$ by $B_{n,M}^{\text{in}}(\cdot)$ and $q^{\text{out}}_\mathrm{B}$ is the parameter for which the proposition holds if we replace $B_{n,M}(\cdot)$ by $B_{n,M}^{\text{out}}(\cdot)$. Then it follows for sufficiently large $n$ and $M$:
\begin{align*}
\mathbb{P}(B_{n,M}(z_1)=0,...,B_{n,M}(z_N)=0)
&\leq
\sum\limits_{(k_i)_{i=1}^N\in K}\sqrt{  \prod\limits_{\{i \colon k_i=0\}}q^{\text{in}}_\mathrm{B} \prod\limits_{\{i \colon k_i=1\}}q^{\text{out}}_\mathrm{B}}\\
&=
\sum\limits_{(k_i)_{i=1}^N\in K}  \prod\limits_{\{i \colon k_i=0\}}\sqrt{q^{\text{in}}_\mathrm{B}}\prod\limits_{\{i \colon k_i=1\}}\sqrt{q^{\text{out}}_\mathrm{B}}\\
&=
\Big(\sqrt{q^{\text{in}}_\mathrm{B}}+\sqrt{q^{\text{out}}_\mathrm{B}} \,\Big)^N\\
&=:
q_\mathrm{B}^N.
\end{align*}
By our choice of the parameters $n$ and $M$ in the beginning of this step, it holds that $q_B<1$. Moreover, remember that it is possible to choose  $q^{\text{in}}_\mathrm{B}$ and $q^{\text{out}}_\mathrm{B}$ arbitrarily small and hence the parameter $q_B$ arbitrarily small.

\section{SINR with random powers percolation for Poisson point processes}
In this chapter we combine different features of the models which have been presented so far in order to model a more realistic situation. 
First, we include the interference from other users, that means we consider a SINR model, and second, we let the signal power be a random variable. 
As before, we aim to understand how far messages can travel through the telecommunication system.
Before we start to study this continuum percolation problem, we define in the following section the new model which incorporates the above described features.
\subsection{Model definition}
Let $X^\lambda = (X_i)_{i \in I}$ be a homogeneous Poisson point process(PPP) in $\mathbb{R}^d$ with intensity $\lambda >0$ and $d \geq 2$.
As before, the points of this point process represent the users or devices of a spatial telecommunication system.
Note that the PPP was formally introduced in Chapter 2.1.
In the next step, we define how a message can be transmitted successfully between these users. 
This is done very similar to the SINR model, which was introduced in Chapter 3.2, and hence we focus in this chapter on the new aspects and encourage the reader to consider Chapter 3.2 for further explanations.
\medskip
\\ \indent
Let us consider two points $X_i$ and $X_j$ of $X^\lambda$, where $X_i$ wants to send a message to $X_j$. 
As in the previous chapter, it depends on the signal to interference and noise ratio (SINR), if this message transmission is successfull. 
We define the SINR as 
\begin{equation*}
\text{SINR}((X_i, \rho_i), (X_j, \rho_j), X^\lambda, \rho, N_0, \gamma):=
\frac{\rho_i \cdot \ell(|X_i-X_j|)}{N_0+\gamma \sum\limits_{k \neq i,j} \rho_k \cdot \ell(|X_k-X_j|)}.
\end{equation*}
Analogously to the SINR model introduced in the previous chapter, we can see that the message transmission is interfered by some environment noise $N_0 >0$ and all other nodes that also transmit a signal. This strength of the interference caused by the other users is described by the factor $\gamma >0$. 
The smaller it is, the less is the message transmission interfered by surrounding users.
Furthermore, the SINR incorporates a path-loss function $\ell \colon [0,\infty) \rightarrow [0,\infty)$, which fulfills the following conditions:
\begin{enumerate}
\item $\ell$ is continuous and strictly decreasing on supp $\ell$,
\item $1 \geq \ell(0) $,
\item $\int\limits_{\mathbb{R}^d}\ell (|x|)\text{d}x < \infty$.
\end{enumerate}
\vspace{0.3cm} \
\indent
The only difference to the previous SINR model is, that we assume that the strength of the emitted signal is not a positive constant. Instead,it is represented by a random point process $\rho =(\rho_i)_{i \in I}$. Thereby, the idea is to attach to each user $X_i$ a random signal power $\rho_i$.
Remember that we can interpret the process $((X_i, \rho_i))_{i \in I}$ as a marked PPP, which we introduced in the second chapter. 
In our case, the set of marks are all nonnegative numbers, which means $\mathcal{M} = [0, \infty)$ and the probability kernel is given by $K \colon \mathbb{R}^d \times [0, \infty) \rightarrow [0,1]$.
The signal power of the different points are independent of each other and independent of $X^\lambda$. Furthermore, they are identically distributed as the signal power random variable $\rho$, which fulfills the following conditions:
\begin{enumerate}
\item $\mathbb{P}(\rho<0)=0$,
\item $\mathbb{P}(\rho=0)<1$,
\item $\mathbb{E}[\rho]<\infty$.
\end{enumerate}
These conditions are not very strong as they reflect a realistic and non-trivial model setting:
We assume that the signal power is nonnegative and that there is a positive probability that it is not constantly zero. Furthermore, we assume that the expectation of the random signal power is bounded.
\medskip
\\ \indent
As done in the previous chapter we say that two points $X_i$ and $X_j$ of $X^\lambda$ are able to communicate if their SINR exceeds a technical threshold $\tau>0$ in both cases. This means that it holds
$\text{SINR}((X_i, \rho_i), (X_j, \rho_j), X^\lambda, \rho, N_0, \gamma) > \tau$ and \\
$\text{SINR}((X_j, \rho_j), (X_i, \rho_i), X^\lambda, \rho, N_0, \gamma)> \tau$.
The graph which arises when we draw an edge between all communicating points, is denoted by $g_{(\gamma,N_0,\tau)}(X^\lambda, \rho)$. 
\medskip
\\ \indent
Note that the just described model has been studied by \cite{KongYeh2007} in the two-dimensional case and under 
stronger conditions on the random signal power. They assumed that it is bounded by a strictly positive minimal and a finite maximal power value, which both have a positive probability.
In the following, we consider the more general case where the random signal power can be unbounded and allow $d \geq 3$. Nevertheless, we briefly discuss in the remarks of this chapter the (easier) situation of a bounded random signal power.
\medskip
\\ \indent
After we have defined our new SINR model, we aim to study under which conditions a phase transition exists.
Let us therefore define for fixed parameters $N_0$ and $\tau$ the criticial density by
\begin{align*}
\lambda_{N_0,\tau}:=
\inf\{\lambda>0 \colon \gamma^*(\lambda')>0 \text{ for all } \lambda' \leq \lambda\},
\end{align*}
where
\begin{align*}
\gamma^*(\lambda)=\gamma^*(\lambda,N_0,\tau):=
\sup\{\gamma >0 \colon \,\mathbb{P}(g_{(\gamma,N_0,\tau)}(X^\lambda, \rho) \text{ percolates})>0\}.
\end{align*}
Remember that we say that a graph percolates if there exists an infinite connected component of users which are all able to send messages to each other.
\medskip
\\
\indent 
In the following Chapter 4.2, we show that there exists a subcritical phase in the SINR model with random signal powers. Afterwards we prove the existence of a supercritical phase in Chapter 4.3.
\newpage
\subsection{Existence of a subcritical phase}
Let us first consider the special case where we have no interference, which means $\gamma=0$.
Then, we define the critical density as
\begin{equation*}
\lambda_{\mathrm{c}}:=\inf\{\lambda>0 \colon g_{(0,N_0,\tau)}(X^\lambda,\rho) \text{ percolates}\}.
\end{equation*} 
The following Proposition \ref{kuchen} states that under certain decay and probability conditions on the path-loss function, there exists a subcritical phase for $\gamma =0$.
The proof is based on Theorem \ref{subcritical_1}, which states the existence of a phase transition in a Boolean model with random radii under certain conditions on the random radii.
After we have shown the proposition, we use it in order to prove the existence of a subcritical phase in our SINR model with interference, i.e., where $\gamma>0$. 
\begin{prop}\label{kuchen}
Let $N_0, \tau>0$ and $\gamma=0$. There exists a subcritical phase, which means $\lambda_{\mathrm{c}} >0$, if there exist parameters $r^* \in (0,\infty)$ and $\alpha, \beta >0$ with $\alpha \beta > 2d-1$ such that $\ell(r) \leq r^{-\beta}$ and $\mathbb{P}(\rho > r) \leq r^{-\alpha}$ for $r \geq r^*$.
\end{prop}
\begin{proof}
As $\gamma=0$, the message transmission between two points is not interfered by other surrounding nodes. In particular, 
two points $X_i$ and $X_j$ can communicate if
\begin{equation*}
\rho_i \cdot \ell(|X_i-X_j|)/N_0 > \tau
\text{   and   } 
\rho_j \cdot \ell(|X_i-X_j|)/N_0 > \tau.
\end{equation*} 
With the notation $R_k:=\ell^{-1}(\tau N_0/\rho_k)$ for all $k \in I$ and the convention that $R_k=0$ if $\tau N_0 / \rho_k > \ell (0)$, this is equivalent to
$
|X_i-X_j| < R_i \text{ and } |X_i-X_j| < R_j$.
In other words, the points $X_i$ and $X_j$ are only able to send messages to each other if
\begin{equation*}
|X_i-X_j| <\min\{R_i,R_j\}.
\end{equation*} 
The arising graph is denoted by $g_{(0,N_0,\tau)}(X^\lambda,\rho)$.
\medskip
\\ \indent
Let us now consider two other points $X_i$ and $X_j$ which are only able to communicate if 
$|X_i-X_j| <R_i+R_j$. 
The arising graph can be described by $g_{(0,N_0,\tau)}(X^\lambda,\mu)$, where the random variable $\mu$ is distributed as 
$\ell^{-1}(\tau N_0 / \rho)$.
It is clear that if there is no infinite cluster in $g_{(0,N_0,\tau)}(X^\lambda,\mu)$ there is no infinite cluster in $g_{(0,N_0,\tau)}(X^\lambda,\rho)$ as 
$\min\{R_i,R_j\} \leq R_i+R_j$.
If
$\mathbb{E}[\mu^{2d-1}] < \infty$
it follows by Theorem \ref{subcritical_1} that there exists a subcritical phase in $g_{(0,N_0,\tau)}(X^\lambda,\mu)$ and hence there exist positive densities such that no percolation occurs. 
\medskip
\\ \indent
In the following we show that under our given conditions on the path-loss function it holds that $\mathbb{E}[\mu^{2d-1}]$ is finite. 
Thereby, we use the fact that for any positive and continuous random variable $\mathbb{Z}$ we have
\begin{equation*}
\mathbb{E} [\mathbb{Z}] = \int\limits_0^\infty \mathbb{P}(\mathbb{Z}>t) \text{d}t.
\end{equation*}
In our case, this means
\begin{equation*}
\mathbb{E}[\mu^{2d-1}]
=
\mathbb{E}\big[\big( \ell^{-1}(\tau N_0 / \rho) \big)^{2d-1}\big]
=
\int\limits_0^\infty \mathbb{P} \big( \big(\ell^{-1}(\tau N_0 / \rho) \big)^{2d-1} >t \big) \text{d}t.
\end{equation*}
As the path-loss function $\ell$ is decreasing, it holds
\begin{align*}
\int\limits_0^\infty \mathbb{P} \big( \big(\ell^{-1}(\tau N_0 / \rho) \big)^{2d-1} >t \big) \text{d}t
&=
\int\limits_0^\infty \mathbb{P}  \big(\tau N_0 / \rho < \ell \big(t^{\frac{1}{2d-1}}\big) \big) \text{d}t\\
&=
\int\limits_0^\infty \mathbb{P}  \Big( \rho > \frac{\tau N_0 }{ \ell \big(t^{\frac{1}{2d-1}}\big)} \Big) \text{d}t.
\end{align*}
In the next step, we split the above term in order to use our assumption on the path-loss function $\ell$, which says that $\ell(r) \leq r^{-\beta}$ for large enough $r$, more precisely for $r \geq r^*$. 
Hence it holds for 
$t^{\frac{1}{2d-1}} \geq r^*$, that
$\ell \big(t^{\frac{1}{2d-1}}\big) \leq t^{\frac{-\beta}{2d-1}}$,
which yields
\begin{equation*}
\frac{\tau N_0 }{ \ell \big(t^{\frac{1}{2d-1}}\big)}  \geq \tau N_0 \, t^{\frac{\beta}{2d-1}}.
\end{equation*}
Thus, the inequality
\begin{equation}\label{yes}
\mathbb{P}  \Big( \rho > \frac{\tau N_0 }{ \ell \big(t^{\frac{1}{2d-1}}\big)} \Big) 
\leq
\mathbb{P}  \big( \rho > \tau N_0 \, t^{\frac{\beta}{2d-1}} \big) 
\end{equation}
holds for $t \geq (r^*)^{2d-1}$.
Moreover, we use our assumption on the random signal power, which means that
$\mathbb{P}(\rho > r) \leq r ^{-\alpha}$ for $r \geq r^*$.
In our case this implies
\begin{equation}\label{yesyes}
\mathbb{P}  \big( \rho > \tau N_0 \, t^{\frac{\beta}{2d-1}} \big) 
\leq
 \big(\tau N_0 \, t^{\frac{\beta}{2d-1}}\big)^{-\alpha}
\end{equation}
for 
$\tau N_0 \, t^{\frac{\beta}{2d-1}} \geq r^*$, or equivalently for 
$t \geq ( r^*/\tau N_0 )^{\frac{2d-1}{\beta}}$.
All together it follows that the equations (\ref{yes}) and (\ref{yesyes}) hold for 
$t \geq c:=\max\{(r^*)^{2d-1},( r^*/\tau N_0 )^{\frac{2d-1}{\beta}}\}$
and hence we obtain
\begin{align*}
\mathbb{E} [\mu^{2d-1}]
&=
\int\limits_0^\infty \mathbb{P}  \Big( \rho > \frac{\tau N_0 }{ \ell \big(t^{\frac{1}{2d-1}}\big)} \Big) \text{d}t\\
&=
\int\limits_0^c \mathbb{P}  \Big( \rho > \frac{\tau N_0 }{ \ell \big(t^{\frac{1}{2d-1}}\big)} \Big) \text{d}t
+
\int\limits_ c^\infty \mathbb{P}  \Big( \rho > \frac{\tau N_0 }{ \ell \big(t^{\frac{1}{2d-1}}\big)} \Big) \text{d}t\\
&\leq
\int\limits_0^ c \mathbb{P}  \Big( \rho > \frac{\tau N_0 }{ \ell \big(t^{\frac{1}{2d-1}}\big)} \Big) \text{d}t
+
\int\limits_ c^\infty \mathbb{P}  \big( \rho > \tau N_0 \, t^{\frac{\beta}{2d-1}} \big) \text{d}t\\
&\leq
\int\limits_0^ c \mathbb{P}  \Big( \rho > \frac{\tau N_0 }{ \ell \big(t^{\frac{1}{2d-1}}\big)} \Big) \text{d}t
+
\int\limits_ c^\infty  \big(\tau N_0 \, t^{\frac{\beta}{2d-1}}\big)^{-\alpha} \text{d}t.
\end{align*}
The first integral is clearly finite as $\mathbb{P}(\cdot) \in [0,1]$ and its integral bounds are finite. 
By using our constraint that $\alpha \beta  > 2d-1$, we can show that the second integral is as well finite. Let us have a closer look at the second integral.
\begin{align*}
\int\limits_ c^\infty  \big(\tau N_0 \, t^{\frac{\beta}{2d-1}}\big)^{-\alpha} \text{d}t
&=
(\tau N_0)^{-\alpha} \Big(- \frac{\alpha\beta} {2d-1} +1 \Big)^{-1}  \Big[ t^{- \frac{\alpha\beta} {2d-1} +1 }\Big]_c^\infty\\
&=
(\tau N_0)^{-\alpha} \Big(- \frac{\alpha\beta} {2d-1} +1 \Big)^{-1}  \Big( \lim\limits_{t \rightarrow \infty} t^{- \frac{\alpha\beta} {2d-1} +1 } -  c^{- \frac{\alpha\beta} {2d-1} +1 }\Big).
\end{align*}
Hence, the integral is finite, if and only if the exponent
$- \frac{\alpha\beta} {2d-1} +1 $ is smaller than zero. This is clearly the case:
\begin{equation*}
- \frac{\alpha\beta} {2d-1} +1
<
- \frac{2d-1} {2d-1} +1
=
-1+1
=0.
\end{equation*}
All together, it follows that $\mathbb{E} [\mu^{2d-1}]$ is finite, what was to be shown.
\\
\end{proof}
Note that common path-loss functions which fulfill the conditions of the above proposition are for example $\ell(r) = (1+r)^{-p}$ and $\ell(r) = \min\{1,r^{-p}\}$, where $p \geq 1$.
\begin{rem}
The above Proposition \ref{kuchen} holds in particular for any path-loss functions with bounded support. 
In the following we briefly show how the above proof becomes even easier in this situation.
Remember that our target is to show that $\mathbb{E}[\mu^{2d-1}]$ is finite in order to apply Theorem \ref{subcritical_1}. Thereby $\mu$ is a random variable which is distributed as $\ell^{-1}(\tau N_0 / \rho)$.
If the path-loss function $\ell$ has bounded support, it follows that
\begin{equation*}
r_{\max}:=\sup\{r \in [0,\infty) \colon \ell(r) \neq 0\} < \infty
\end{equation*}
 and hence for any $r \in [0,\infty)$ such that $\ell(r) \neq 0$ it holds $\ell^{-1}(r) \leq r_{\max}$.
In particular it holds
$\ell^{-1}(\tau N_0 / \rho) \leq r_{\max}$ and thus
$\mathbb{E}[\mu^{2d-1}] =\mathbb{E}\big[( \ell^{-1}(\tau N_0 / \rho) )^{2d-1}\big] \leq r_{\max}^{\,2d-1} < \infty$, what was to be shown.
\end{rem}
\begin{rem}
It is easy to see that Proposition \ref{kuchen} holds as well under the condition of a bounded random signal power as studied by \cite{KongYeh2007}.
In this case, we assume that there exists some $\rho_{\max} < \infty$ such that $\rho \leq \rho_{\max}$. It follows that 
$\tau N_0 /\rho  \geq \tau N_0 / \rho_{\max}>0$ and hence
$\ell^{-1}(\tau N_0 /\rho)  \leq \ell^{-1}( \tau N_0 / \rho_{\max}) < \infty$.
This yields $\mathbb{E}[\mu^{2d-1}] < \infty$, what was needed to apply Theorem \ref{subcritical_1} as done in the proof of the above proposition.
\end{rem}
From Proposition \ref{kuchen} follows directly the existence of a subcritical phase in our original model with interference. This is shown in the following corollary.
\begin{cor}\label{easypeacy}
Let $N_0, \tau>0$ and $\gamma >0$.
There exists a subcritical phase, which means $\lambda_{N_0,\tau}>0$, if there exist parameters $r^* \in (0,\infty)$ and $\alpha, \beta >0$ with $\alpha \beta > 2d-1$ such that $\ell(r) \leq r^{-\beta}$ and $\mathbb{P}(\rho > r) \leq r^{-\alpha}$ for $r \geq r^*$.
\end{cor}
\begin{proof}
We first show that all the connections of the graph $g_{(\gamma,N_0,\tau)}(X^\lambda,\rho)$ do as well exist in the graph $g_{(0,N_0,\tau)}(X^\lambda,\rho)$, where we set $\gamma =0$, i.e., we cancel the interference from all other points.
Let us therefore consider two points $X_i$ and $X_j$ of the point process $X^\lambda$ which can send messages to each other in the model with $\gamma >0$. This means that 
\begin{equation*}
\frac{\rho_i \cdot \ell(|X_i-X_j|)}{N_0+\gamma \sum\limits_{k \neq i,j} \rho_k \cdot \ell(|X_k-X_j|)}> \tau
\text{   and    }
\frac{\rho_j \cdot \ell(|X_j-X_i|)}{N_0+\gamma \sum\limits_{k \neq j,i} \rho_k \cdot \ell(|X_k-X_i|)}> \tau,
\end{equation*}
i.e., the SINR is larger than the threshold $\tau$ in both directions.
As the interference is a positive value, it follows that 
\begin{equation*}
\frac{\rho_i \cdot \ell(|X_i-X_j|)}{N_0}> \tau
\text{   and    }
\frac{\rho_j \cdot \ell(|X_j-X_i|)}{N_0}> \tau,
\end{equation*}
which is exactly the condition of the model with $\gamma =0$ that guarantees a successfull message transmission between the two points $X_i$ and $X_j$.
Hence, there does clearly not exist an infinite cluster in the graph $g_{(\gamma,N_0,\tau)}(X^\lambda,\rho)$, if there is no one in $g_{(0,N_0,\tau)}(X^\lambda,\rho)$. This means our assertion follows if there exists a subcritical phase in $g_{(0,N_0,\tau)}(X^\lambda,\rho)$, what was shown in Proposition \ref{kuchen}.\\
\end{proof}
\subsection{Existence of a supercritical phase}
\noindent
In this section we show that there exists a supercritical phase if the random power exceeds a certain large enough constant with positive probability and some exponential moments are finite.
\begin{thm}\label{mymainresult}
Let $N_0, \tau>0$ and $\gamma>0$. For large enough $\lambda$ there exists a $\gamma^*(\lambda)>0$ such that the graph $g_{(\gamma,N_0,\tau)}(X^\lambda, \rho)$ percolates for any $\gamma \leq \gamma^*(\lambda)$, 
if the following conditions hold:
\begin{enumerate}
\item[a)] There exists a parameter
 $r \geq N_0 \tau /\ell (0)$ such that $\mathbb{P}(\rho > r) >0$.
\item[b)] There exists a parameter $\alpha >0$ such that $\mathbb{E} [ \exp( \alpha \rho)] = \int\limits_{[0,\infty)} \exp(\alpha \, m) K(\text{d}m) $ is finite.
\end{enumerate}
\end{thm}
\noindent
Note that the last integral describes the moment-generating function, which is defined for any random variable $\mathbb{Y}$ as
$
M_{\mathbb{Y}} (t) := \mathbb{E} [ \exp(t \mathbb{Y})]$, where $ t \in \mathbb{R}
$.
\begin{proof}
Analogously to the proof of the existence of a supercritical phase in Theorem \ref{thm_phase_transition} in the previous chapter, the proof of this theorem consists of three steps: First, the continuous percolation problem is mapped on a discrete percolation problem. In the second step, we show the existence of a supercritical phase in the discrete case. And finally we show in the third step that discrete percolation implies continuum percolation.\\

\noindent \textbf{Step 1: Mapping on a lattice}\\
We start with two definitions which we need in order to map the continuous percolation problem on a discrete lattice.
\newpage
\begin{defi}
Let us define a site $z \in \mathbb{Z}^d$ as \textit{good} if
\begin{enumerate}
\item
there exists $X_k \in X^\lambda \cap \text{Q}_1(z)$ with signal power $\rho_k > r$ and
\item
every $X_i$, $X_j \in X^\lambda \cap \text{Q}_{3}(z)$ are connected by a path in $g_\delta(X^\lambda)\cap \text{Q}_{6}(z)$,
\end{enumerate}
where $\delta := \ell^{-1} (N_0 \tau / r) / 2$. 
Otherwise, it is called \textit{bad}. 
\end{defi}
\noindent
Remember that the constant $r\geq N_0 \tau/\ell(0)$ exists by the first condition in Theorem \ref{mymainresult} and it hence clearly holds that the parameter $\delta$ is strictly positive.
Moreover, the cube with side length $a \geq 0$ and center $z$ is denoted by $\text{Q}_a(z)$ and $g_\delta(X^\lambda)$ describes the Gilbert's graph which was introduced in Chapter 2.
\medskip
\\ \indent
In the following we define the shot noise processes. They are an adapted version of the ones we introduced in the previous chapter in Definition \ref{def_shot}. In comparison, we now include the random signal power.
\begin{defi}
For $a \geq 0$ and $z \in \mathbb{R}^d$ we define the shot noise processes as
\begin{equation*}
\tilde{I}(z) = \sum\limits_{(X_i, \rho_i) \in (X^\lambda, \rho)} \rho_i  \,\ell(|z-X_i|)
\text{   and   }
\tilde{I}_a(z) = \sum\limits_{(X_i, \rho_i) \in (X^\lambda, \rho)} \rho_i  \,\ell_a(|z-X_i|).
\end{equation*}
\end{defi}
Now, we map the continuous percolation problem to a discrete lattice percolation problem.
Thereby, the idea is to separate the percolation problem in two parts: The connectedness problem and the interference problem.
Let us define for $z \in \mathbb{Z}^d$ and $M>0$: \medskip
\\ \indent
$A(z)=\mathds{1}\{z \text{ is good}\}$, $B_{M}(z)=\mathds{1}\{\tilde{I}_{6}(z) \leq M\}$ and $C_{M}(z)=A(z) \cdot B_{M}(z)$.
\medskip
\\ 
A site $z$ in the lattice $\mathbb{Z}^d$ is then called \textit{open}, if $C_{M}(z)=1$, otherwise it is called \textit{closed}.
\\

\noindent\textbf{Step 2: Percolation in the lattice}\\
In this step we aim to show that the lattice $\mathbb{Z}^d$ percolates. 
The proof consists of three parts, each of it considers one of the random variables, which we defined in the previous step. 
First, we show Proposition \ref{klar}, 
which takes care of the connectedness problem and states that for sufficiently large $\lambda >0$ the probability that any pairwise distinct sites in the lattice are bad can be bounded arbitrarily small.
Afterwards, Proposition \ref{spass} states that for sufficiently large $M >0$ the probability that the interference at any pairwise distinct sites in the lattice is greater than $M$ can be bounded arbitrarily small.
In the third part, which is Proposition \ref{name}, we combine these two propositions.
Then, the percolation of the lattice follows immediatly with the Peierls argument (consider the proof of Theorem \ref{haselnuss} for more information).
\begin{prop}\label{klar}
Under the conditions a) and b) of Theorem \ref{mymainresult}, for all sufficiently large $\lambda >0$, there exists a constant $q_\mathrm{A} <1$ such that for any $N \in\mathbb{N}$ and pairwise distinct sites $z_1,...,z_N \in \mathbb{Z}^d$ it holds
\begin{equation*}
\mathbb{P}(A(z_1)=0,...,A(z_N)=0)\leq q_\mathrm{A}^N.
\end{equation*}
Moreover, for any $\epsilon >0$ and for sufficiently large $\lambda$, it holds $q_A \leq \epsilon$.
\end{prop}
\begin{proof}
By our previous definition of $A(\cdot)$ we know that any random variable $A(z_i)$ only depends on the setting within a certain area around the point $z_i$. 
More precisely, for any $z_i$ and $z_j$ which are at least $7^d$ distant from each other, the random variables $A(z_i)$ and $A(z_j)$ are independent.
Hence, there exists $m \geq 1$ and a subset $\{k_j\}_{j=1}^m$ of $[N]=\{1,...,N\}$ such that $A(z_{k_1}),...,A(z_{k_m})$ are independent and $m \geq N/7^d$.
Then, we obtain
\begin{align*}
\mathbb{P}(A(z_1)=0,...,A(z_N)=0)
&\leq
\mathbb{P}(A(z_{k_1})=0,...,A(z_{k_m})=0)\\
&=
\mathbb{P}(A(o)=0)^m\\
&\leq
\mathbb{P}(A(o)=0)^N.
\end{align*}
Furthermore, we know by our definition that
\begin{equation*}
\lim\limits_{\lambda \rightarrow \infty}\mathbb{P}(A(o)=0)=0.
\end{equation*}
All together, the proposition is shown.\\
\end{proof}
\begin{prop}\label{spass}
Under the conditions a) and b) of Theorem \ref{mymainresult},
for all $\lambda >0$ and for all sufficiently large $M >0$, there exists a constant $q_B <1$ such that for any $N \in\mathbb{N}$ and pairwise distinct sites $z_1,...,z_N \in \mathbb{Z}^d$ it holds
\begin{equation*}
\mathbb{P}(B_{M}(z_1)=0,...,B_{M}(z_N)=0)\leq q_\mathrm{B}^N.
\end{equation*}
Moreover, for any $\epsilon >0$, $\lambda >0$, one can choose $M$ so large that $q_\mathrm{B} \leq \epsilon$.
\end{prop}
\begin{proof}
First, we make a simple approximation:
\begin{align*}
\mathbb{P}(B_{M}(z_1)=0,...,B_{M}(z_N)=0)
&=
\mathbb{P}(\tilde{I}_{6}(z_1) >M,..., \tilde{I}_{6}(z_N) >M) \\
&\leq
\mathbb{P}\Big(\sum\limits_{i=1}^N \tilde{I}_{6}(z_i) > N M \Big).
\end{align*}
By the Markov's inequality follows for any $s \geq 0$
\begin{equation}\label{markov_equation}
\mathbb{P}\Big(\sum\limits_{i=1}^N \tilde{I}_{6}(z_i) > N M \Big)
\leq
\exp(-sNM) \cdot \underbrace{\mathbb{E} \Big[  \exp \Big( s \sum\limits_{i=1}^N  \tilde{I}_{6}(z_i) \Big) \Big]}_{(*)}.
\end{equation}
In the following, we approximate the term $(*)$ further.
Note that by our definition, it holds
\begin{equation*}
\sum\limits_{i=1}^N  \tilde{I}_{6}(z_i)
=
\sum\limits_{i=1}^N   \sum\limits_{(X_k, \rho_k) \in (X^\lambda, \rho)} \rho_k  \,\ell_{6}(|z_i-X_k|)
=
\sum\limits_{(X_k, \rho_k) \in (X^\lambda, \rho)} \sum\limits_{i=1}^N    \rho_k  \,\ell_{6}(|z_i-X_k|)
\end{equation*}
and hence
\begin{equation*}
(*) =
\mathbb{E} \Big[  \exp \Big( s \sum\limits_{(X_k, \rho_k) \in (X^\lambda, \rho)} \sum\limits_{i=1}^N    \rho_k  \,\ell_{6}(|z_i-X_k|)\Big) \Big].
\end{equation*}
\indent
By Theorem \ref{mark_thm}, the \textit{Marking Theorem}, we know that the point process $((X_k, \rho_k))_{k \in I}$ is in distribution equal to the PPP on $\mathbb{R}^d \times [0, \infty)$ with intensity $\mu \otimes K$, where $\mu = \lambda \cdot \text{Leb}$, as we consider a stationary PPP.
We now apply the following Campbell's theorem
\begin{equation*}
\mathbb{E} \Big[ \exp \Big( s \sum\limits_k f(X_k, \rho_k) \Big) \Big]
=
\exp \Big( \,\, \int\limits_{\mathbb{R}^d \times [0, \infty)} \Big( \exp (s \cdot f(x,m))-1 \Big) \, \text{d} (\lambda \text{Leb} \otimes K)\Big)
\end{equation*}
for the measurable function
\begin{equation*}
f(x,m) = \sum\limits_{i =1}^N m \cdot \ell_{6} (|z_i -x|),
\end{equation*}
where $x \in \mathbb{R}^d$ and $m \in [0, \infty)$.
This yields
\begin{align*}
(*)
&=
\exp \Big( \,\, \int\limits_{\mathbb{R}^d \times [0, \infty)} \Big( \exp \Big(s \cdot  \sum\limits_{i =1}^N m \cdot \ell_{6} (|z_i -x|) \Big)-1 \Big) \, \text{d} (\lambda \text{Leb} \otimes K)\Big)\\
&=
\exp \Big( \, \,\int\limits_{\mathbb{R}^d}  \int\limits_{[0,\infty)} \Big( \exp \Big(m \, s \, \sum\limits_{i =1}^N \ell_{6} (|z_i -x|) \Big)-1 \Big) \, K(\text{d}m) \,  \lambda \text{Leb}(\text{d}x)\Big).
\end{align*}
Note that we have assumed that the marks are independent of the point process.
In the following, we have a closer look at the integral
\begin{equation*}
\int\limits_{[0,\infty)} \exp \Big(m \, s \,  \sum\limits_{i =1}^N  \ell_{6} (|z_i -x|) \Big) \, K(\text{d}m),
\end{equation*}
for any $x \in \mathbb{R}^d$.
Remember that the moment-generating function of the random signal power $\rho$ is given by
\begin{equation*}
M_{\rho} (\alpha) 
= \int\limits_{[0,\infty)} \exp( \alpha \cdot m) \, K( \text{d}m), \,\, \alpha \in \mathbb{R}.
\end{equation*}
Furthermore, the first moment of $\rho$ is given by
\begin{equation*}
\mathbb{E}(\rho) = \int\limits_{[0,\infty)} m \, K(\text{d}m)
\end{equation*}
and it is finite by assumption.
Moreover, we have assumed that there exists some $\alpha >0$ such that the integral 
$ \int\limits_0^\infty \exp( \alpha \cdot m) \, K( \text{d}m)$ is finite.
Hence, we have that the integral converges:
\begin{equation}\label{int_conv}
\lim\limits_{\alpha \downarrow 0} \int\limits_{[0,\infty)} \exp( \alpha \cdot m) \, K( \text{d}m)  =1.
\end{equation}
\indent
Now we use the fact that the first derivative of the moment-generating function of a random variable at zero equals the first moment of the random variable. This means for our signal power random variable $\rho$
\begin{equation*}
\frac{\partial}{\partial \alpha} M_{\rho} (\alpha) \Big \vert_{\alpha =0}
=
\frac{\partial}{\partial \alpha} \Big( \int\limits_{[0,\infty)} \exp( \alpha \cdot m) \, K( \text{d}m) \Big) \Big \vert_{\alpha =0}
=
\int\limits_{ [0, \infty)} m   \,K( \text{d}m).
\end{equation*}
As the moment-generating function $M_{\rho} (\alpha)$ is continuously differentiable at $\alpha =0$, we know that for small $\alpha >0$ the slope of $M_{\rho} (\alpha)$ is still close to the slope at $\alpha =0$.
Together with equation (\ref{int_conv}) it follows that for small enough $ \alpha >0$, the integral
$\int\limits_{[0,\infty)} \exp( \alpha \cdot m) \, K( \text{d}m)$ is close to one and its slope is close to the value
$\int\limits_{ [0, \infty)} m   \,K( \text{d}m)$.
We can conclude that for any $C \in (1,\infty)$ there exists some small enough $\alpha>0$ such that
\begin{equation*}
\int\limits_{ [0, \infty)} ( \exp ( \alpha \, m  ) )  \,K( \text{d}m)
\leq
C \, \alpha
\int\limits_{ [0, \infty)} m   \,K( \text{d}m) + 1.
\end{equation*}
Furthermore, we know that the integral
$
\int\limits_{[0,\infty)} K( \text{d}m) $
equals one
as $K$ is a probability kernel and hence a probability measure on the mark space. This yields for any $C \in (1,\infty)$ and accordingly chosen small enough parameter $\alpha >0$
\begin{equation}\label{tanzen}
\int\limits_{ [0, \infty)} ( \exp ( \alpha \, m ) -1)  \,K( \text{d}m)
\leq
\int\limits_{ [0, \infty)} C \, \alpha \, m   \,K( \text{d}m).
\end{equation}
\indent
From \cite[page 88]{Tobias2019} we know that there exists a constant $K_0 < \infty$, such that for any $x \in \mathbb{R}^d$, $n, N \in \mathbb{N}$ and distinct points $z_i \in \mathbb{R}^d$, we have
\begin{equation*}
\sum\limits_{i = 1}^N \ell _{6n} ( |nz_i - x | )
\leq
\sum\limits_{z \in \mathbb{Z}^d} \ell _{6n} ( |nz - x | ) \leq K_0.
\end{equation*}
Consider the proof of percolation in the SINR model in the previous chapter for a detailed calculation.
In particular, it holds for $n=1$, which we use in the following.\medskip
\\ \indent
Let us now fix $C \in (1, \infty)$.
Then we know that equation (\ref{tanzen}) holds for some small enough value $\alpha >0$ and hence as well for a smaller value
$\alpha \cdot \sum\limits_{i=1}^N \ell_{6}( |z_i - x | )  /K_0 \leq \alpha$. 
We obtain for any $x \in \mathbb{R}^d$
\begin{equation*}
\int\limits_{ [0, \infty)} \Big( \exp \Big(m \, \frac{\alpha}{K_0} \, \sum\limits_{i=1}^N \ell_{6}( |z_i - x | ) \Big) -1 \Big) 
\,K(\text{d}m)
\leq
\int\limits_{[0, \infty)} C \, m \, \frac{\alpha}{K_0} \, \sum\limits_{i=1}^N \ell_{6}( |z_i - x | )   \, K( \text{d}m).
\end{equation*}
We can conclude for $s = \alpha / K_0 >0$ that
\begin{align*}
(*)
&\leq
\exp \Big( \,\, \int\limits_{\mathbb{R}^d \times [0, \infty)} C \, m \, \frac{\alpha}{K_0} \, \sum\limits_{i=1}^N \ell_{6}( |z_i - x | )  \, \text{d} (\lambda \text{Leb} \otimes K)\Big)\\
&=
\exp \Big( C \, \frac{\alpha}{K_0} \,   \int\limits_{[0, \infty)} m   \, K( \text{d}m)  \int\limits_{\mathbb{R}^d}\,  \sum\limits_{i=1}^N \ell_{6}( |z_i - x | ) \, \lambda \text{Leb} (\text{d}x)\Big)\\
&=
\exp \Big( C \, \frac{\alpha}{K_0} \,   \int\limits_{[0, \infty)} m   \, K( \text{d}m) \,  \sum\limits_{i=1}^N  \int\limits_{\mathbb{R}^d}\ell_{6}( |z_i - x | ) \, \lambda \text{Leb} (\text{d}x)\Big)\\
&=
\exp \Big( C \, \frac{\alpha}{K_0} \,   \int\limits_{[0, \infty)} m   \, K( \text{d}m) \,  \sum\limits_{i=1}^N  \int\limits_{\mathbb{R}^d}\ell_{6}( | x | ) \, \lambda \text{Leb} (\text{d}x)\Big)\\
&=
\exp \Big( C \, \frac{\alpha}{K_0} \,   \int\limits_{[0, \infty)} m   \, K( \text{d}m) \,  N  \int\limits_{\mathbb{R}^d}\ell_{6}( | x | ) \, \lambda \text{Leb} (\text{d}x)\Big).
\end{align*}
Thereby we used that the Lebesgue measure is invariant under shifts. Note that the integrals
\begin{equation*}
\int\limits_{[0, \infty)} m   \, K( \text{d}m)
\text{ and }
\int\limits_{\mathbb{R}^d}\ell_{6}( | x | ) \, \lambda \text{Leb} (\text{d}x)
\end{equation*}
are finite by assumption. 
All together, it follows with equation (\ref{markov_equation}) for $s = \alpha / K_0 >0$ that\\

$
\mathbb{P}(B_{M}(z_1)=0,...,B_{M}(z_N)=0)
$
\begin{align*}
&\leq
\exp\Big(- \frac{\alpha}{K_0} NM\Big)
\exp \Big( C \, \frac{\alpha}{K_0}    \int\limits_{[0, \infty)} m   \, K( \text{d}m) \,  N  \int\limits_{\mathbb{R}^d}\ell_{6}( | x | ) \, \lambda \text{Leb} (\text{d}x)\Big)\\
&\leq
\Big[\exp \Big(-\frac{\alpha}{K_0} M +C \, \frac{\alpha}{K_0} \,   \int\limits_{[0, \infty)} m   \, K( \text{d}m) \, \int\limits_{\mathbb{R}^d}\ell_{6}( | x | ) \, \lambda \text{Leb} (\text{d}x)\Big)\Big]^N\\
&=:q_\mathrm{B}^N.
\end{align*}
As the choice of the parameters $C$ and $\alpha$ is independent of the parameter $M$, it is possible for any $\epsilon > 0$ to choose $M$ so large, that $q_\mathrm{B} \leq \epsilon$. Hence, we have shown the proposition.\\
\end{proof}
\begin{prop}\label{name}
Under the conditions a) and b) of Theorem \ref{mymainresult},
for all $\lambda >0$, there exists a constant $q_C <1$ such that for all $N \in\mathbb{N}$ and pairwise distinct sites $z_1,...,z_N \in \mathbb{Z}^d$ it holds
\begin{equation*}
\mathbb{P}(C_{M}(z_1)=0,...,C_{M}(z_N)=0)\leq q_\mathrm{C}^N.
\end{equation*}
Moreover, for any $\epsilon >0$, there exist $\lambda$, $M$ large enough so that $q_\mathrm{C} \leq \epsilon$.
\newpage
\end{prop}
\begin{proof}
The proof can be done analogously to the proof of Proposition \ref{prop_C} in the previous chapter.\\
\end{proof}
The percolation of the lattice follows from the usual Peierls argument, see for example the proof of Theorem \ref{haselnuss}.\\

\noindent\textbf{Step 3: Percolation in the SINR graph with random signal power}\\
In this last step, we proof that the existence of percolation in the lattice implies percolation in our original model with interference and random signal powers. We procede very similar to the third step of the proof of Theorem \ref{thm_phase_transition}.\medskip
\\ \indent
Let us first give a simple corollary, which we need in the following.
\begin{cor}\label{teatime}
For $a \geq 0$ it holds that 
$
\tilde{I}(x) \leq \tilde{I}_a(z)
$
for all $x \in \mathbb{R}^d$ and $z\in \text{Q}_a(x)$.
\end{cor}
\begin{proof}
The proof can be done analgously to the proof of Corollary \ref{cor_shot_noise}.
\end{proof}
In the previous Step 2, we have shown that there exist parameters $\lambda$ and $M$ such that an infinite connected component arises in the lattice, which we denote by $\mathcal{C}$. 
This means by our definition that for any site $z$ in $\mathcal{C}$, it holds  $A(z) =1$ and $ B_{M}(z) = 1$. 
Let
$z, z' \in \mathcal{C}$ be two neighbouring sites in the lattice. Then we know $A(z)=A(z')=1$, i.e., the sites are good.
This implies that there exist $X_i \in \text{Q}_1(z)$ and $X_j \in \text{Q}_1(z')$, which are connected by a path in $g_\delta(X^\lambda) \cap \text{Q}_{6}(z)$. In particular, it follows that $|X_i -X_j| \leq \delta$.
Furthermore, we know by our definition that the random signal power of the points $X_i$ and $X_j$ exceeds a certain level, more precisely it holds $\rho_i , \rho_j > r$.\medskip
\\ \indent
Now, we use that for any $z \in \mathcal{C}$, it holds $B_{M}(z) = 1$, i.e., the interference at the point $z$ is bounded by the constant $M$: $\tilde{I}_{6}(z) \leq M$.
From Corollary \ref{teatime} follows $\tilde{I}(x) \leq \tilde{I}_{6}(z)$ for all $x \in \mathbb{R}$ and $z \in \text{Q}_{6}(x)$, and hence $\tilde{I}(x)\leq \tilde{I}_{6}(z)$ for all $x \in \text{Q}_{6}(z)$.  In particular for $x = X_i$, respectively $x = X_j$, it follows
$\tilde{I}(X_i) \leq M$, respectively $\tilde{I}(X_j) \leq M$.\medskip
\\ \indent
Using all this, we show that $X_i$ and $X_j$ are as well connected in $g_{(\gamma,N_0,\tau)}(X^\lambda, \rho)$, where $\gamma  \in (0,\gamma^*)$ with
\begin{equation*}
\gamma^* = \frac{N_0}{M} \Big( \frac{\ell(\delta)}{\ell(2\delta)}-1 \Big) >0.
\end{equation*}
Note that the parameter $\gamma^*$ is strictly positive as we assumed that the path-loss function $\ell$ is strictly decreasing and hence it holds
$\ell(\delta)> \ell(2 \delta)$.
All together, we have for any points $X_i$ and $X_j$ with $|X_i-X_j| \leq \delta$
\begin{align*}
\frac{\rho_i \ell(|X_i-X_j|)}{N_0+ \gamma \sum\limits_{k \neq i,j} \rho_k \ell(|X_k-X_j|)}
>
\frac{\epsilon \ell(\delta )}{N_0+ \gamma  M}
>
\frac{\epsilon \ell(\delta)}{N_0+ \gamma^*  M}
=
\frac{\epsilon \ell(2\delta)}{N_0}
=
\tau
,
\end{align*}
and thus our model percolates.\\
\end{proof}
We finish this section with the special case of $\gamma =0$, i.e., where no interference is considered.
The following proposition shows that model percolation is possible under some conditions on the random signal power.
The proof can be carried out very similar to the one of Lemma \ref{sonne}, which states the existence of a supercritical phase in a Boolean model with random radii for large enough densities.
Remember that $\lambda_{\text{c}}$ is defined in the beginning of Chapter 4.2 as the critical density in the case where $\gamma =0$ and $p_{\text{cr}}$ describes the critical percolation probability in a discrete lattice (see Chapter 2.2.2).
\begin{prop}\label{pflaume}
Let $N_0, \tau>0$ and $\gamma=0$. There exists a supercritical phase, which means $\lambda_{\mathrm{c}} < \infty$, 
if 
there exists $\epsilon \geq N_0 \tau / \ell(0)$ such that $\mathbb{P}(\rho>\epsilon) > \epsilon$ and $1- \epsilon >p_{\mathrm{cr}}$.
\end{prop}
\begin{proof}
Let us consider two points $X_i$ and $X_j$ of the PPP $X^\lambda$ which like to communicate. As explained in the beginning of the proof of Proposition \ref{kuchen}, this is only possible if their distance is smaller than $\min\{R_i,R_j\}$, i.e., if $|X_i-X_j|< \min\{R_i,R_j\}$. 
\medskip
\\ \indent
Now, the proof can be done analogously to the proof of Lemma \ref{sonne}. 
The only difference is the side length of the cubes $\delta>0$, which we choose so small that any two points in neighbouring cubes are at distance at most $\ell^{-1}(N_0 \tau /  \epsilon)$ (and not $2 \epsilon$ as in Lemma \ref{sonne}). This choice of $\delta$ provides that an infinite cluster of open cubes implies an unbounded component in our model. This can be shown as follows.
\medskip
\\ \indent
For any two neighbouring open cubes $c_i$ and $c_j$, there exist $X_i \in c_i$ and $X_j \in c_j$ with signal power $\rho_i >\epsilon $, respectively $\rho_j >\epsilon$. 
It follows
$N_0 \tau / \rho_i < N_0 \tau / \epsilon$ and as $\ell^{-1}$ is decreasing it holds
$R_i =\ell^{-1}(N_0 \tau / \rho_i) > \ell^{-1} (N_0 \tau / \epsilon)$, respectively $R_j> \ell^{-1} (N_0 \tau / \epsilon)$. 
In other words, we have
\begin{equation*}
\ell^{-1} (N_0 \tau / \epsilon)<\min\{R_i, R_j\}.
\end{equation*}
By our choice of the side length of the cubes we know that $|X_i -X_j| < \ell^{-1} (N_0 \tau / \epsilon)$ and thus
$|X_i -X_j| < \min \{R_i,R_j\}$. This means that the points $X_i$ and $X_j$ are connected in our model.
Hence, if there is an infinite cluster of open cubes, there exists an infinite component in our model.
This shows the assertion.\\
\end{proof}
\begin{rem}
As before, it is easy to see that the above proposition holds if the random signal power is bounded from below by a large enough value as studied by \cite{KongYeh2007}. More precisely, we assume that there exists $\rho_{\min} > N_0 \tau / \ell(0)$ such that $\rho \geq \rho_{\min}$.
Then it follows 
$N_0 \tau / \rho \leq N_0 \tau / \rho_{\min}$ and hence 
$\ell^{-1} (N_0 \tau / \rho_{\min}) \leq \ell^{-1} (N_0 \tau / \rho)$. 
The last expressions are well defined as 
$\rho_{\min} > N_0 \tau / \ell(0)$ and hence
$\ell(0) > N_0 \tau / \rho_{\min} $.
Accordingly to the above introduced notation, we write $R_{\min}:=\ell^{-1} (N_0 \tau / \rho_{\min})$. 
As 
$R_{\min} \leq \min\{R_i,R_j\}$, 
all connections in the graph $g_{(0,N_0,\tau)}(X^\lambda,\rho)$ do as well exist in the graph $g_{(0,N_0,\tau)}(X^\lambda,\rho_{\min})$, where the radius is constant. This latter graph is a Boolean model with constant radii and hence we know that it percolates for large enough densities. And thus, in our graph with random radii and no interference $g_{(0,N_0,\tau)}(X^\lambda,\rho)$ percolation occurs.
\end{rem}

\section{Outlook}
In the previous chapters we have studied different approaches modeling a telecommunication system. They all aim to understand under which conditions users of the telecommunication system are able to communicate even if they are located far away from each other.
In order to be able to learn something from those models, it is important to create models which reflect the reality.
In this last chapter, we discuss the possibilities within the presented models which could make them even more realistic.
\medskip
\\ \indent
Remember that we discussed two possible sources of randomness in the models: The first one is the location of the users, expressed by a random point process, and the second one are the message trajectories.
As already mentioned before, there are different possibilities for the random point process. In this thesis we sometimes considered the Poisson point process (PPP) and sometimes the Cox point process (CPP). 
Thereby the latter one is more sophisticated, but as well more realistic. 
So far, the Boolean model has been studied for CPPs in \cite{HirschJahnelCali2018} and the the SINR model for CPPs has been analysed in \cite{Tobias2019}. It remains to consider the SINR model with random transmission powers for CPPs.
\medskip
\\ \indent
Let us now consider the second source of randomness, the message trajectories.
Remember that we incorporated this factor when we considered the Boolean model with random radii in Chapter 2.2.3.
Moreover, we included random transmission powers in the SINR model in Chapter 4.
Still, there are further modifications possible. We always assumed that the random radii, respectively the random transmission powers, do not depend on the users to which they are attached. Hence, it would be interesting to analyse the case where there exists a dependence between the random radii/transmission powers and the position of the users in space.

\addcontentsline{toc}{section}{References}

\begin{thebibliography}{WWWW98}
\bibitem[BB09]{Baccelli2009}
{\sc Baccelli, F. and Blaszczyszyn, B.},
\emph{Stochastic Geometry and Wireless Networks: Volume I: Theory}, Now Publishers Inc., 2009.
\bibitem[BB10]{Blaszczyszyn2010}
{\sc Blaszczyszyn, B. and Yogeshwaran, D.},
  {Connectivity in Sub-Poisson Networks},
 \emph{Proc. of 48th Annual Allerton Conference, {\normalfont see also:} arXiv:1009.5696},
  2010.
\bibitem[BB13]{Blaszczyszyn2013}
{\sc Blaszczyszyn, B. and Yogeshwaran, D.},
{Clustering and percolation of point processes},
  \emph{Electron. J. Probab. {\normalfont \textbf{18}, 1--20}}, 2013.
\bibitem[DBT05]{Dousse2005}
{\sc Dousse, O. and Baccelli, F. and Thiran, P.}
  {Impact of Interferences on Connectivity in Ad Hoc Networks},
\emph{IEEE/ACM Trans. Networking {\normalfont \textbf{1}, 425--436}},
2005.
\bibitem[DFM$^+$06]{Dousse2006}
{\sc Dousse, O. and Franceschetti, M. and Macris, N. and Meester, R. and Thiran, P.},
{Percolation in the Signal to Interference Ratio Graph},
\emph{J. Appl. Prob. {\normalfont \textbf{43}, 552--562}},
  2006.
\bibitem[FM07]{Franceschetti2007}
{\sc Franceschetti, M. and Meester, R.},
\emph{Random Networks for Communication -- From Statistical Physics to Information Systems}, {Cambridge University Press},
  2007.
\bibitem[Gil61]{Gilbert1961}
{\sc Gilbert, E.N.}
{Random plane networks},
 \emph{Journal of the Society for Industrial and Appl. Math. {\normalfont \textbf{9}, 533--543}},
  1961.
\bibitem[GKP16]{Ghosh2016}
{\sc Ghosh, S. and Khrisnapur, M. and Peres, Y.},
{Continuum percolation for {G}aussian zeroes and {G}inibre eigenvalues},
\emph{Ann. Probab., {\normalfont \textbf{44:5}}},
  2016.
\bibitem[Hal85]{Hall1985}
{\sc Hall, P.},
{On continuum percolation},
 \emph{Ann. Prob. {\normalfont \textbf{13}, 1250--1266}},
 1985.
\bibitem[HJC17]{HirschJahnelCali2018}
{\sc Hirsch, C. and Jahnel, B. and Cali, E.},
{Continuum Percolation For {C}ox Point Processes},
 \emph{{\normalfont Accepted for publication in} Stochastic Processes and their Applications, {\normalfont see also: }arXiv:1710.11407},
 2017.
\bibitem[JK18]{JahnelKoenig2018}
{\sc Jahnel, B. and König, W.},
\emph{Probabilistic Methods In Telecommunication}, {Lecture Notes at TU Berlin},
 2018.
\bibitem[Jan16]{Jansen2016}
{\sc Jansen, S.},
{Continuum percolation for {G}ibbsian point processes with attractive interactions},
\emph{Electron. J. Probab., {\normalfont \textbf{21:47}, 22}},
  2016.
\bibitem[KY07]{KongYeh2007}
{\sc Kong, Z. and Yeh, E. M.},
{Directed Percolation in Wireless Networks with Interference and Noise},
 \emph{arXiv:0712.2469},
2007.
\bibitem[LP17]{Last2017}
{\sc Last, G. and Penrose, M.},
\emph{Lectures on the Poisson Process},
{IMS Textbook, Cambridge University Press},
2017.
\bibitem[LSS97]{Liggett1997}
{\sc Liggett, T. M. and Schonmann, R. H. and Stacey, A.M.},
{Domination by product measures},
  \emph{Ann. Prob. {\normalfont \textbf{25:1}, 71--95}},
  1997.
\bibitem[MR96]{MeesterRoy1996}
{\sc Meester, R. and Roy, R.},
  {\emph Continuum Percolation},
  {Cambridge University Press},
  1996.
  \bibitem[Stu13]{Stucki2013}
{\sc Stucki, K.},
{Continuum percolation for {G}ibbs point processes},
\emph{Electron. Commun. Probab., {\normalfont \textbf{18:67}, 10}},
  2013.
\bibitem[Tób18]{Tobias2018}
{\sc Tóbiás, A.},
{Signal to interference ratio percolation for Cox point processes},
\emph{arXiv:1808.09857},
2018.
\bibitem[Tób19]{Tobias2019}
{\sc Tóbiás, A.},
 \emph{Message routeing and percolation in interference limited multihop networks},
 {depositonce.tu-berlin.de/handle/11303/9293},
2019.

\end{thebibliography}

\end{spacing}
\end{document}